\documentclass[12pt]{amsart}
\usepackage{amscd}     
\usepackage{amssymb}
\usepackage{amsmath, amsthm, graphics, dsfont}
\usepackage{xypic}     
\LaTeXdiagrams        
\usepackage[all]{xy}
\xyoption{2cell} \UseAllTwocells \xyoption{frame} \CompileMatrices
\allowdisplaybreaks[3]
\usepackage{amsfonts}
\usepackage[normalem]{ulem}
\usepackage{hyperref}
\usepackage{tikz}
\usepackage{mathtools}
\usepackage{verbatim} 
\usepackage{MnSymbol}
\usepackage{latexsym}
\usepackage{epsfig}
\usepackage{enumerate}
\usepackage{times}
\usepackage{xcolor}
\usepackage{geometry}

\newgeometry{vmargin={25mm,25mm}, hmargin={25mm,25mm}}   

\newtheorem{prop}{Proposition}[section]
\newtheorem{lem}[prop]{Lemma}

\newtheorem{cor}[prop]{Corollary}
\newtheorem{thm}[prop]{Theorem}
\newtheorem{rem}[prop]{Remark}

\newtheorem*{thm*}{Main Theorem}
\newtheorem*{thm**}{Finite Generation Property}

\numberwithin{equation}{section}

\newcommand{\CnZn}{[\mathbb{C}^n/\mathbb{Z}_n]}
\newcommand{\KP}{K\mathbb{P}}

\DeclareMathOperator{\diag}{diag}

\title{Higher genus Gromov--Witten theory of $[\mathbb{C}^n/\mathbb{Z}_n]$ II:\\ { }crepant resolution correspondence \hfill}

\author[Genlik]{Deniz Genlik}
\address{Department of Mathematics\\University of Illinois at Urbana--Champaign\\1409 W. Green Street (MC-382)\\Urbana, IL 61801\\ USA}
\email{genlik@illinois.edu}

\author[Tseng]{Hsian-Hua Tseng}
\address{Department of Mathematics\\ Ohio State University\\ 100 Math Tower, 231 West 18th Ave. \\ Columbus,  OH 43210\\ USA}
\email{hhtseng@math.ohio-state.edu}

\begin{document}



\maketitle

\begin{abstract}
We study the structure of the higher genus Gromov--Witten theory of the total space $\KP^{n-1}$ of the canonical bundle of the projective space $\mathbb{P}^{n-1}$. We prove the finite generation property for the Gromov--Witten potential of $\KP^{n-1}$ by working out the details of its cohomological field theory (CohFT). More precisely, we prove that the Gromov--Witten potential of $\KP^{n-1}$ lies in an explicit polynomial ring using the Givental--Teleman classification of the semisimple CohFTs. 

In \cite{gt}, we carried out a parallel study for $\CnZn$ and proved that the Gromov--Witten potential of $\CnZn$ lies in a similar polynomial ring. The main result of this paper is a crepant resolution correspondence for higher genus Gromov--Witten theories of $\KP^{n-1}$ and $\CnZn$, which is proved by establishing an isomorphism between the polynomial rings associated to $\KP^{n-1}$ and $\CnZn$. This paper generalizes the works of Lho--Pandharipande \cite{lho-p2} for the case of $[\mathbb{C}^3/\mathbb{Z}_3]$ and Lho \cite{lho} for the case $[\mathbb{C}^5/\mathbb{Z}_5]$ to arbitrary $n\geq 3$.
\end{abstract}

\tableofcontents

\section{Introduction}

In this paper, which is a sequel to \cite{gt}, we continue our study of Gromov--Witten theory of the orbifold $\CnZn$.
\subsection{Basic set-up}
Here we record some basic notations to be used in this paper.

Let $$\mathrm{T}=(\mathbb{C}^*)^n.$$
In what follows, we denote\footnote{We denote the localized $\mathrm{T}$-equivariant orbifold cohomology as $H_{\mathrm{T,Orb}}^*(-)$.} by 
$$H_\mathrm{T}^*(-),$$ the {\em localized} $\mathrm{T}$-equivariant cohomology of a $\mathrm{T}$-space.

We consider the action of the cyclic group $\mathbb{Z}_n$ on $\mathbb{C}^n$ defined via sending its generator $1\in\mathbb{Z}_n$ to $n\times n$ matrix
$$\text{diag}(e^{\frac{2\pi\sqrt{-1}}{n}},..., e^{\frac{2\pi\sqrt{-1}}{n}}).$$ The quotient $\CnZn$ is a smooth Deligne--Mumford stack. Let the torus $\mathrm{T}$ act on $\CnZn$ via the diagonal action of $\mathrm{T}$ on $\mathbb{C}^n$ with weights
\begin{equation*}
\lambda_0, \ldots , \lambda_{n-1},
\end{equation*} 
and 
\begin{equation*}
    \phi_0=1\in H^0_\mathrm{T}(\CnZn), \phi_k=1\in H^0_\mathrm{T}(B\mathbb{Z}_n), 1\leq k\leq n-1,
\end{equation*}
be an additive basis of $H^*_\mathrm{T,Orb}(\CnZn)$.

The Gromov--Witten potentials associated to $\phi_{c_{1}}, \ldots,\phi_{c_{m}}\in H^{\star}_{\mathrm{T,Orb}}\left(\left[\mathbb{C}^n/\mathbb{Z}_n\right]\right)$ are defined by
\begin{equation*}
\begin{aligned}
\mathcal{F}_{g, m}^{\left[\mathbb{C}^{n} / \mathbb{Z}_{n}\right]}\left(\phi_{c_{1}}, \ldots, \phi_{c_{m}}\right)
&=\sum_{d=0}^{\infty} \frac{\Theta^{d}}{d !}\left\langle \phi_{c_1},\ldots,\phi_{c_m},\phi_1,\ldots,\phi_1\right\rangle_{m+d}^{\CnZn}\\
&=\sum_{d=0}^{\infty} \frac{\Theta^{d}}{d !} \int_{\left[\overline{M}_{g, m+d}^{\mathrm{orb}}\left(\left[\mathbb{C}^{n} / \mathbb{Z}_{n}\right], 0\right)\right]^{\mathrm{vir}}} \prod_{k=1}^{m} \mathrm{ev}_{i}^{*}\left(\phi_{c_{k}}\right) \prod_{i=m+1}^{m+d} \mathrm{ev}_{i}^{*}\left(\phi_{1}\right).
\end{aligned}  
\end{equation*}

Let the torus $\mathrm{T}=(\mathbb{C}^*)^n$ act on $\mathbb{P}^{n-1}$ with weights
\begin{equation}
-\chi_0,\ldots,-\chi_{n-1}.
\end{equation}
This $\mathrm{T}$-action admits a canonical lift to the total space $\KP^{n-1}$ of the canonical bundle of $\mathbb{P}^{n-1}$.
Let 
\begin{equation*}
p_i=[0:\cdots:0:\underbrace{1}_{i^\text{th}}:0:\cdots:0]\in \mathbb{P}^{n-1}, \quad 0\leq i \leq n-1
\end{equation*}
be the $\mathrm{T}$-fixed points. The $\mathrm{T}$-weight\footnote{Recall that $\KP^{n-1}\simeq \mathcal{O}_{\mathbb{P}^{n-1}}(-n)$ as line bundles.} of $\KP^{n-1}\to \mathbb{P}^{n-1}$ at $p_i$ is $-n\chi_i$. 

Let
$$1=H^0, H, H^2, ..., H^{n-1}$$ be the additive basis of
\begin{equation*}
H_\mathrm{T}^*(K\mathbb{P}^{n-1})\simeq H_{\mathrm{T}}^*\left(\mathbb{P}^{n-1}\right) \simeq \mathbb{Q}(\chi_0, \ldots, \chi_{n-1})\left[H\right] \big/\left(\prod_{i=0}^{n-1}\left(H-\chi_i\right)\right),
\end{equation*}
 where $H=c_1^{\mathrm{T}}(\mathcal{O}_{\mathbb{P}^{n-1}}(1))$.

The Gromov--Witten potentials associated to $H^{c_1},\ldots,H^{c_m}\in H_\mathrm{T}^*(K\mathbb{P}^{n-1})$ are defined by
\begin{equation}
\begin{aligned}
\mathcal{F}_{g, m}^{\KP^{n-1}}\left(H^{c_1}, \ldots, H^{c_m}\right)
&=\sum_{d=0}^{\infty} {Q^d}\left\langle H^{c_1},\ldots,H^{c_m} \right\rangle_{g,m,d}^{\KP^{n-1}}\\
&=\sum_{d=0}^{\infty} {Q^d} \int_{\left[\overline{M}_{g, m}\left(\KP^{n-1}, d\right)\right]^{\mathrm{vir}}} \prod_{k=1}^m \mathrm{ev}_i^*\left(H^{c_k}\right).
\end{aligned}
\end{equation}

In this paper, we impose the following specializations of equivariant parameters: for $0\leq i \leq n-1$,
\begin{equation}\label{eqn:specialization}
   \lambda_i=\begin{cases} 
      e^{\frac{2\pi\sqrt{-1}i}{n}}e^{\frac{\pi\sqrt{-1}}{n}}& \text{if $n$ is even,}\\
      e^{\frac{2\pi\sqrt{-1}i}{n}} & \text{if $n$ is odd},
     \end{cases}
\end{equation}
and
\begin{equation}\label{eqn:specialization_KP}
\chi_i= e^{\frac{2\pi\sqrt{-1}i}{n}}.
\end{equation}

\subsection{Results}
The scheme-theoretic quotient $\mathbb{C}^n/\mathbb{Z}_n$ is a singular variety, with a unique singular point. The stack quotient $\CnZn$ is smooth, and the coarsening map 
\begin{equation}\label{eqn:mapCnZn}
\CnZn\to \mathbb{C}^n/\mathbb{Z}_n    
\end{equation}
is birational and crepant.

Blowing up the unique singular point of $\mathbb{C}^n/\mathbb{Z}_n$ yields $\KP^{n-1}$. The blow-up map
\begin{equation}\label{eqn:mapKP}
\KP^{n-1}\to \mathbb{C}^n/\mathbb{Z}_n    
\end{equation}
is birational and crepant. 

Both maps (\ref{eqn:mapCnZn}) and (\ref{eqn:mapKP}) are {\em crepant resolutions} of the singular variety $\mathbb{C}^n/\mathbb{Z}_n$. The crepant resolution conjecture \cite{bg}, \cite{cit}, \cite{cr} predicts that $\CnZn$ and $\KP^{n-1}$ have equivalent Gromov--Witten theories. In genus $0$, such an equivalence is a special case of the main result of \cite{cij} for toric orbifolds.

It is possible to lift the results of \cite{cij} to higher genus using Givental--Teleman classification of semisimple cohomological field theories (\cite{t}, see also \cite{ppz} and \cite{Picm}). A main difficulty for doing this is establishing analytic properties of higher genus Gromov--Witten potentials. For {\em compact} toric orbifolds, this is achieved in \cite{ci0} and a higher genus crepant resolution correspondence is derived for compact toric orbifolds in that paper. 

A similar analysis of higher genus Gromov--Witten theories of the {\em non-compact} targets $[\mathbb{C}^3/\mathbb{Z}_3]$ and $K\mathbb{P}^2$, which is the $n=3$ case of our setup, is carried out in \cite{ci}. As a consequence, \cite{ci} contains a formulation and proof of a higher genus crepant resolution correspondence for the case $n=3$. Other results about Gromov--Witten theory of $K\mathbb{P}^2$, such as modularity, are also obtained in \cite{ci}.

An alternative formulation of higher genus crepant resolution correspondence for the case $n=3$ is found and proven in \cite{lho-p}. The version in \cite{lho-p} in somewhat simpler and the analytic issues are easier to handle in the setup of \cite{lho-p}.

According to \cite[Section 10.7]{ci}, the version of crepant resolution correspondence in \cite{ci} implies the version in \cite{lho-p}.

In this paper, we establish a crepant resolution correspondence for all cases $n\geq 3$. Our approach is parallel to that of \cite{lho-p}. 

In \cite{gt}, we construct a ring $$\mathds{F}_{\CnZn}\coloneqq \mathbb{C}[(L^{\CnZn})^{\pm{1}}][\mathfrak{S}^{\CnZn}_n][\mathfrak{C}^{\CnZn}_n]$$ whose generators are explicit functions, and we show that the generating functions of Gromov--Witten theory of $\CnZn$ are contained in this ring,
\begin{equation*}
\mathcal{F}_{g, m}^{\left[\mathbb{C}^{n} / \mathbb{Z}_{n}\right]}\left(\phi_{c_{1}}, \ldots, \phi_{c_{m}}\right)\in\mathds{F}_{\CnZn},
\end{equation*}
see \cite[Corollary 3.4]{gt}. In other words, we prove a finite generation property for $\mathcal{F}_{g, m}^{\left[\mathbb{C}^{n} / \mathbb{Z}_{n}\right]}\left(\phi_{c_{1}}, \ldots, \phi_{c_{m}}\right)$.

In this paper, we obtain a parallel result for $\KP^{n-1}$. More precisely, we construct a similar ring $$\mathds{F}_{\KP^{n-1}} \coloneqq \mathbb{C}[(L^{\KP^{n-1}})^{\pm{1}}][\mathfrak{S}_n^{\KP^{n-1}}][\mathfrak{C}_n^{\KP^{n-1}}]$$ for $\KP^{n-1}$ and show the following: 
\begin{thm**}[=Corollary \ref{cor:VertexEdgeCont}]
The Gromov--Witten potential of $\KP^{n-1}$ satisfies
\begin{equation*}
 \mathcal{F}_{g, m}^{\KP^{n-1}}\left(H^{c_{1}}, \ldots, H^{c_{m}}\right)\in \mathds{F}_{\KP^{n-1}}.    
\end{equation*}
\end{thm**}

In Section \ref{subsubsec:Change_of_variables}, we construct a ring map $$\Upsilon: \mathds{F}_{\KP^{n-1}}\to  \mathds{F}_{\CnZn},$$
which depends on $\rho$, a chosen $n$-th root of $-1$. The main result of this paper is the following identification of Gromov--Witten generating functions via $\Upsilon$:

\begin{thm*}[=Theorem \ref{thm:Main_Theorem}]
For $g$ and $m$ in the stable range $2g-2+m>0$, the ring map $\Upsilon$ yields
\begin{equation*}
\mathcal{F}_{g, m}^{\CnZn}\left(\phi_{c_{1}}, \ldots, \phi_{c_{m}}\right)=(-1)^{1-g}\rho^{3g-3+m}\Upsilon \left(\mathcal{F}_{g, m}^{\KP^{n-1}}\left(H^{c_{1}}, \ldots, H^{c_{m}}\right)\right).
\end{equation*}    
\end{thm*}
As the $n=3$ case in \cite{lho-p2}, we interpret this result as a crepant resolution correspondence for $\CnZn$ and $\KP^{n-1}$. We remark that crepant resolution correspondence for the $n=5$ case was studied in \cite{lho}.

As the $n=3$ case treated in \cite{lho-p2}, we prove our crepant resolution correspondence result by analyzing the semisimple CohFT structures of Gromov--Witten theories of $\CnZn$ and $\KP^{n-1}$. For $\CnZn$, this is done in our previous paper \cite{gt}. A parallel study\footnote{For the $n=3$ case, the required results for $\KP^2$ are obtained by studying stable quotient theory \cite{lho-p}.} for $\KP^{n-1}$ is carried out in this paper. Using the Givental--Teleman classification for semisimple CohFTs, we reduce the correspondence to an identification of their $R$-matrices. A comparison of the flatness equations (which determine $R$-matrices) reduces the identification of $R$-matrices to an identity, see Lemma \ref{lem:Final_of_proof_strategy}. We prove the required identity by studying asymptotic expansions of oscillatory integrals of the Landau--Ginzburg mirror of $\KP^{n-1}$. 

\subsection{Outline}
The rest of this paper is organized as follows. Section \ref{sec:GWCnZn} concerns Gromov--Witten theory of $\CnZn$, which was studied in detail in \cite{gt}. The main new thing here is the quantum Riemann--Roch operator determined in Section \ref{sec:QRRCnZn}. Section \ref{sec:GWKP} is devoted to the study of Gromov--Witten theory of $\KP^{n-1}$. We analyze the $I$-function of $\KP^{n-1}$ in Section \ref{sec:I_func} and use it to calculate genus $0$ invariants in Section \ref{sec:genus0KP}. We calculate ingredients of Frobenius structures in Section \ref{sec:FrobKP}. Finally, we determine the quantum Riemann--Roch operator arising from {\em degree zero} Gromov--Witten invariants of $\KP^{n-1}$ in Section \ref{sec:QRRKP}. Section \ref{sec:ringKP} is devoted to the construction of the ring $\mathds{F}_{\KP^{n-1}}$ for the Gromov--Witten theory of $\KP^{n-1}$. 

In Section \ref{sec:CohFT_iden}, we develop the main results of this paper. Section \ref{sec:iden} is devoted to constructing and studying the map $\Upsilon$. We introduce change of variables and the map $\Upsilon$ in Section \ref{subsubsec:Change_of_variables}. In Section \ref{sec:PF_iden}, we compare Picard--Fuchs equations of $\CnZn$ and $\KP^{n-1}$ under the change of variables. In Section \ref{sec:flatness_iden}, we compare the modified flatness equations needed to study $R$-matrices. In Section \ref{sec:genus0_iden}, we compare genus $0$ invariants. In Section \ref{sec:R_iden}, we reduce the comparison of $R$-matrices to an identity in Lemma \ref{lem:Final_of_proof_strategy}. In Section \ref{sec:formulaGW}, we apply previous results to deduce an identification of generating functions. First, we explicitly write down formulae for generating functions for $\KP^{n-1}$ in Section \ref{sec:FgmKP}. Using these formulae, along with those for $\CnZn$ given in \cite[Proposition 3.3]{gt}, we deduce the Main Theorem in Section \ref{sec:CRC}.

Section \ref{sec:asymptotics_of_oscillatory_integrals} contains a proof of the required $R$-matrix identity stated in Lemma \ref{lem:Final_of_proof_strategy}. Appendix \ref{Appendix:Analysis_of_I_function} contains some analytic properties of the $I$-functions of $\KP^{n-1}$.

\subsection{Notation}
The Gromov--Witten theory of $\CnZn$ was studied in detail in our previous paper \cite{gt}. In this paper, we freely use the results obtained in \cite{gt}. We place ``$\CnZn$'' as a superscript or subscript whenever we refer to an object in \cite{gt}. In general, the notation exactly matches with \cite{gt} when ``$\CnZn$'' is removed. If there is a mismatch in the notation after removing ``$\CnZn$'', we either redefine the object or emphasize the difference.

We also use the following double-bracket notations for Gromov--Witten potentials, 
\begin{equation*}
\begin{aligned}
\left\langle\left\langle\phi_{c_{1}}, \ldots, \phi_{c_{m}}\right\rangle\right\rangle_{g, m}^{\left[\mathbb{C}^{n} / \mathbb{Z}_{n}\right]}
&\coloneq \mathcal{F}_{g, m}^{\left[\mathbb{C}^{n} / \mathbb{Z}_{n}\right]}\left(\phi_{c_{1}}, \ldots, \phi_{c_{m}}\right),\\
\left\langle\left\langle H^{c_{1}}, \ldots, H^{c_{m}}\right\rangle\right\rangle_{g, m}^{\KP^{n-1}}
&\coloneq \mathcal{F}_{g, m}^{\KP^{n-1}}\left(H^{c_{1}}, \ldots, H^{c_{m}}\right).
\end{aligned}
\end{equation*}

Additionally, the following involutions are used throughout the paper:
$$\mathrm{Inv}:\{0,...,n-1\}\rightarrow\{0,...,n-1\},$$ with $\mathrm{Inv}(0)=0$ and $\mathrm{Inv}(i)=n-i$ for $1\leq i \leq n-1$, and 
$$\mathrm{Ion}:\{0,...,n\}\rightarrow\{0,...,n\},$$ with $\mathrm{Ion}(0)=n$, and $\mathrm{Ion}(i)=i$ for $1\leq i \leq n-1$.

\subsection{Acknowledgment}
We thank R. Pandharipande for helpful comments. D. G. is supported in part by a Special Graduate Assignment fellowship by the OSU Department of Mathematics, and H.-H. T. is supported in part by a Simons Foundation collaboration grant.




\section{Orbifold Gromov--Witten theory of \texorpdfstring{$[\mathbb{C}^n/\mathbb{Z}_n]$}{CnZn}}\label{sec:GWCnZn}

 In this section, we first provide a brief account of certain results about the orbifold Gromov--Witten theory of $\CnZn$ obtained in \cite{gt}, and then compute the quantum Riemann--Roch operator for $\CnZn$.
 
 In the specialization (\ref{eqn:specialization}), the $I$-function for  $[\mathbb{C}^n/\mathbb{Z}_n]$ is given by
\begin{equation}\label{def:I-function}
I^{\CnZn}\left(x, z\right)=
        \sum_{k=0}^{\infty}\frac{x^k}{{z^k}k!}\prod_{\substack{b:0\leq b<\frac{k}{n} \\ \langle b \rangle=\langle\frac{k}{n}\rangle}}\left(1+(-1)^n(bz)^n\right)\phi_k,
\end{equation}
and we can calculate the orbifold Poincar\'e pairing to be
\begin{equation}\label{eqn:metric_CnZn}
g^{\CnZn}(\phi_i, \phi_j)=\frac{1}{n}\delta_{\text{Inv}(i),j}.
\end{equation}
Let $\mathsf{D}_{\CnZn}$ be the operator defined by $$\mathsf{D}_{\CnZn}\coloneqq x\frac{d}{dx}.$$
The $I$-function for  $[\mathbb{C}^n/\mathbb{Z}_n]$ is the solution of the following Picard--Fuchs equation\footnote{Throughout the paper, we omit the variables in most of the places when it is clear.}
\begin{equation}\label{eqn:PF_For_CnZn}
\frac{1}{x^n}\prod_{i=0}^{n-1}\left(\mathsf{D}_{\CnZn}-i\right)I^{\CnZn}-(-1)^n\left(\frac{1}{n}\right)^n\mathsf{D}_{\CnZn}^nI^{\CnZn}=\left(\frac{1}{z}\right)^nI^{\CnZn}.
\end{equation}
This equation can be rewritten as\footnote{Here $s_{n,k}$ is a Stirling number of the first kind. A short discussion on Stirling numbers and references for more detailed treatments can be found in \cite{gt}.}
\begin{equation}\label{eqn:PF_for_CnZn}
\mathsf{D}_{\CnZn}^{n}I^{\CnZn}+\frac{\mathsf{D}_{\CnZn} L^{\CnZn}}{L^{\CnZn}} \sum_{k=1}^{n-1} s_{n,k} \mathsf{D}_{\CnZn}^{k}I^{\CnZn}=\left(\frac{L^{\CnZn}}{z}\right)^nI^{\CnZn}
\end{equation}
where
\begin{equation}\label{eqn:Lseries_for_CnZn}
    L^{\CnZn}=x\left(1-(-1)^n\left(\frac{x}{n}\right)^n\right)^{-\frac{1}{n}}.
\end{equation}
In \cite{gt}, certain power series 
$A_i^{\CnZn}$, $C_i^{\CnZn}$, $K_i^{\CnZn}$ and $X_i^{\CnZn}$ in $\mathbb{C}[\![x]\!]$ are defined and used to study the Gromov--Witten theory of $\CnZn$. By \cite[Section 1.4]{gt}, the genus $0$, $3$-point Gromov--Witten invariants of $\CnZn$
\begin{equation}\label{eqn:3pt_inv_CnZn}
\left\langle\left\langle\phi_i,\phi_j,\phi_k\right\rangle\right\rangle_{0,3}^{\left[\mathbb{C}^n/ \mathbb{Z}_n\right]}=\frac{K^{\CnZn}_{i+j}}{K^{\CnZn}_iK^{\CnZn}_j}\frac{1}{n}\delta_{\text{Inv}(i+j \text{ mod }n), k}. \end{equation}

\subsection{Quantum Riemann--Roch operator for \texorpdfstring{$[\mathbb{C}^n/\mathbb{Z}_n]$}{CnZn}}\label{sec:QRRCnZn}
The stack $\CnZn$ may be viewed as the total space of a vector bundle $$\mathcal{V}\to B\mathbb{Z}_n$$ over the stack $B\mathbb{Z}_n$. The $\mathrm{T}$-equivariant Gromov--Witten theory of $\CnZn$ is the same as the Gromov--Witten theory of $B\mathbb{Z}_n$ twisted by the vector bundle $\mathcal{V}$ and the inverse $\mathrm{T}$-equivariant Euler class $e_{\mathrm{T}}^{-1}(-)$. The orbifold quantum Riemann--Roch theorem \cite{Tseng} shows that the $\mathrm{T}$-equivariant Gromov--Witten theory of $\CnZn$ is related to the Gromov--Witten theory of $B\mathbb{Z}_n$ by an operator 
$$\mathsf{Q}^{\CnZn}\in \text{End}(H^*_{\mathrm{T, Orb}}(\CnZn))[\![z]\!].$$ 
We need to calculate $\mathsf{Q}^{\CnZn}$ explicitly.

Recall that, for a $\mathbb{C}^*$ acting on a vector bundle $E$ by scaling the fibers, the inverse $\mathbb{C}^*$-equivariant Euler class of $E$ satisfies
\begin{equation*}
    e_{\mathbb{C}^*}^{-1}(E)=\exp\left(-\ln\lambda\text{ch}_0(E)+\sum_{k>0}\frac{(-1)^k(k-1)!}{\lambda^k}\text{ch}_k(E) \right),
\end{equation*}
here $\lambda$ is the equivariant parameter, see e.g. \cite[Section 4]{cg}. This yields the following values of the parameters 
\begin{equation}\label{eqn:QRR_parameters}
s_k(\lambda)=
\begin{cases}
-\ln\lambda & \text{ if } k=0\\
\frac{(-1)^k(k-1)!}{\lambda^k} & \text{ if } k>0.
\end{cases}    
\end{equation}
These parameters will be needed when applying (orbifold) quantum Riemann--Roch theorem.

By the definition of $\CnZn$, the vector bundle $\mathcal{V}\to B\mathbb{Z}_n$ is a direct sum of line bundles $$\mathcal{V}=\mathcal{L}^{\oplus n},$$ where $\mathcal{L}\to B\mathbb{Z}_n$ is defined by the following $\mathbb{Z}_n$-character $$\mathbb{Z}_n\to \mathbb{C}^*, \quad \mathbb{Z}_n\ni 1\mapsto e^{\frac{2\pi\sqrt{-1}}{n}}\in \mathbb{C}^*.$$

Recall that the Bernoulli polynomials $B_m(x)$ are defined by $$\frac{te^{tx}}{e^t-1}=\sum_{m\geq 0}\frac{B_m(x)t^m}{m!},$$
and $B_m:=B_m(0)$ are the Bernoulli numbers. We have 
\begin{equation}\label{eqn:bernoulli_reflect}
B_m(1-x)=(-1)^mB_m(x).
\end{equation}

It follows from the orbifold quantum Riemann--Roch theorem \cite{Tseng} that the restriction $\mathsf{Q}^{\CnZn}|_{H^*_{\mathrm{T}}(\text{pt})\cdot\phi_i}$ to $H^*_{\mathrm{T}}(\text{pt})\cdot\phi_i\subset H^*_{\mathrm{T,Orb}}(\CnZn)$ is the multiplication by
\begin{equation}\label{eqn:oqrr1}
\prod_{j=0}^{n-1}\exp\left(\sum_{k>0}\frac{(-1)^k}{k(k+1)}B_{k+1}\left(\frac{i}{n}\right)\frac{z^k}{\lambda_j^k} \right).
\end{equation}

In the specialization (\ref{eqn:specialization}), we calculate $\sum_{j=0}^{n-1}\frac{1}{\lambda_j^k}$ as follows. When $n$ is odd, we have 
\begin{equation}
\sum_{j=0}^{n-1}\frac{1}{\lambda_j^k}=
    \begin{cases}
    n & \text{ if } k\equiv 0 \text{ mod }n\\
    0 & \text{ otherwise}.
    \end{cases}
\end{equation}
Therefore, the expression (\ref{eqn:oqrr1}) becomes
\begin{equation}\label{eqn:oqrr_odd}
    \exp\left(n\sum_{l>0}(-1)^{nl}\frac{B_{nl+1}\left(\frac{i}{n}\right)}{nl+1}\frac{z^{nl}}{nl} \right).
\end{equation}

When $n$ is even, we have 
\begin{equation}
\sum_{j=0}^{n-1}\frac{1}{\lambda_j^k}=
    \begin{cases}
    (-1)^{l}n & \text{ if } k=nl\equiv 0 \text{ mod }n\\
    0 & \text{ otherwise}.
    \end{cases}
\end{equation}
Consequently, the expression (\ref{eqn:oqrr1}) reads as
\begin{equation}\label{eqn:oqrr_even}
    \exp\left(n\sum_{l>0}(-1)^{(n+1)l}\frac{B_{nl+1}\left(\frac{i}{n}\right)}{nl+1}\frac{z^{nl}}{nl} \right).
\end{equation}

If $n$ is odd, then $nl=(n-1)l+l\equiv l \text{ mod }2$. If $n$ is even, then $(n+1)l=nl+l\equiv l \text{ mod }2$. Thus (\ref{eqn:oqrr_odd}) and (\ref{eqn:oqrr_even}) can be written uniformly as
\begin{equation}\label{eqn:oqrr_final}
    \exp\left(n\sum_{l>0}(-1)^{l}\frac{B_{nl+1}\left(\frac{i}{n}\right)}{nl+1}\frac{z^{nl}}{nl} \right).
\end{equation}
Consequently, the restriction $\mathsf{Q}^{\CnZn}|_{H^*_{\mathrm{T}}(\text{pt})\cdot\phi_i}$ to $H^*_{\mathrm{T}}(\text{pt})\cdot\phi_i\subset H^*_{\mathrm{T,Orb}}(\CnZn)$ is the multiplication by (\ref{eqn:oqrr_final}).

\section{Gromov--Witten theory of \texorpdfstring{$K\mathbb{P}^{n-1}$}{KP{n-1}}}\label{sec:GWKP}

In this section, we study the Gromov--Witten theory of $\KP^{n-1}$.

In the specializations (\ref{eqn:specialization_KP}), we have $0=\prod_{i=0}^{n-1}\left(H-\chi_i\right)=H^n-1$ in $H_{\mathrm{T}}^*(\KP^{n-1})$. Hence, we see that 
\begin{equation}\label{eq:Hnisequalto1}
H^n=1.
\end{equation}

The twisted Poincar\'e pairing for $K\mathbb{P}^{n-1}$ is given by
\begin{equation}\label{eqn:metric_KP}
\begin{split}
g^{{K\mathbb{P}^{n-1}}}(H^i,H^j)
=\int_{\mathbb{P}^{n-1}}\frac{H^{i+j}}{c_1^T(\mathcal{O}(-n))}
=-\frac{1}{n}\int_{\mathbb{P}^{n-1}}{H^{i+j-1}}
=-\frac{1}{n}\delta_{\mathrm{Inv}(i),j},
\end{split}
\end{equation}
so the Poincar\'e dual of $H^i$ is $(H^i)^\vee=-nH^{\text{Inv}(i)}$.

\subsection{Basic properties of the \texorpdfstring{$I$}{I}-function}\label{sec:I_func}

\vspace{1em}
The (small) $I$-function of $K\mathbb{P}^{n-1}$, which has been known for some time \cite{g0}, may be obtained by applying the recipe of \cite{ccit} to the $\mathrm{T}$-equivariant $J$-function of $\mathbb{P}^{n-1}$: 
\begin{equation}\label{eqn:small_I}
I^{K\mathbb{P}^{n-1}}(q,z)
=\sum_{d\geq{0}}q^d(-1)^{nd}\frac{\prod_{k=0}^{nd-1}(nH+kz)}{\prod_{k=1}^d\prod_{i=0}^{n-1}(H-\chi_i+kz)}.
\end{equation}
Dividing the numerator and denominator of (\ref{eqn:small_I}) by $z^{nd}$, we see that
\begin{equation}\label{eqn:I_KP_interms_of_F_minus_1}
\begin{aligned}
I^{K\mathbb{P}^{n-1}}(q,z)
=&\sum_{d\geq{0}}q^d(-1)^{nd}\frac{\prod_{k=0}^{nd-1}(n\frac{H}{z}+k)}{\prod_{k=1}^d\prod_{i=0}^{n-1}(\frac{H}{z}+k-\frac{\chi_i}{z})}\\
=&\sum_{d\geq{0}}q^d(-1)^{nd}\frac{\prod_{k=0}^{nd-1}(n\frac{H}{z}+k)}{\prod_{k=1}^d((\frac{H}{z}+k)^n-\frac{H^n}{z^n})}=\mathcal{F}_{-1}\left(H/z,(-1)^nq\right)
\end{aligned}
\end{equation}
by specializations (\ref{eqn:specialization_KP}) and equation (\ref{eq:Hnisequalto1}), where $\mathcal{F}_{-1}(-,-)$ is the hypergeometric series\footnote{While refering to \cite{zz}, we used their notation for the hypergeometric series. The notation $\mathcal{F}_{-1}(-,-)$ should not be confused with the Gromov--Witten potential notation we used in our paper.} in \cite[Section 2]{zz}.

We expand $I^{K\mathbb{P}^{n-1}}(q,z)$ into a $1/z$ series as follows. For $k\neq 0$, if we rewrite 
\begin{equation*}
nH+kz=\left(\frac{nH}{kz}+1\right)(kz), \quad \text{and} \quad H-\chi_i+kz=\left(\left(\frac{H-\chi_i}{kz}\right)+1\right)(kz),
\end{equation*}
then we obtain
\begin{equation}\label{eqn:I_func_analysis_for_mirrorthm}
\begin{split}
I^{K\mathbb{P}^{n-1}}(q,z)
=&1+\sum_{d\geq{1}}q^d(-1)^{nd}\frac{\prod_{k=0}^{nd-1}(\left(\frac{nH}{kz}+1\right)(kz))}{\prod_{k=1}^d\prod_{i=0}^{n-1}(\left(\left(\frac{H-\chi_i}{kz}\right)+1\right)(kz))}\\
=&1+\sum_{d\geq 1}q^d(-1)^{nd}\frac{nH(nd-1)!z^{nd-1}\prod_{k=1}^{nd-1}\left(\frac{nH}{kz}+1\right)}{(d!)^nz^{nd}\prod_{k=1}^d\prod_{i=0}^{n-1}\left(\left(\frac{H-\chi_i}{kz}\right)+1\right)}\\
=&1+\sum_{d\geq 1}q^d(-1)^{nd}\frac{nH(nd-1)!}{(d!)^n}\frac{1}{z}\prod_{k=1}^{nd-1}\left(\frac{nH}{kz}+1\right)\prod_{k=1}^d\prod_{i=0}^{n-1}\left(\frac{1}{1+\left(\frac{H-\chi_i}{kz}\right)} \right)\\
=&1+\sum_{d\geq 1}q^d(-1)^{nd}\frac{nH(nd-1)!}{(d!)^n}\frac{1}{z}\prod_{k=1}^{nd-1}\left(\frac{nH}{kz}+1\right)\prod_{k=1}^d\prod_{i=0}^{n-1}\left(\sum_{l\geq 0}(-1)^l\left(\frac{H-\chi_i}{kz}\right)^l\right)\\
=&1+\sum_{d\geq 1}q^d(-1)^{nd}\frac{nH(nd-1)!}{(d!)^n}\frac{1}{z}\prod_{k=1}^{nd-1}\left(\frac{nH}{kz}+1\right)\prod_{k=1}^d\prod_{i=0}^{n-1}\left(1-\left(\frac{H-\chi_i}{kz}\right)+O(1/z^2)\right)\\
=&1+\sum_{d\geq 1}q^d(-1)^{nd}\frac{nH(nd-1)!}{(d!)^n}\frac{1}{z}\left(1+\left( \sum_{k=1}^{nd-1}\frac{1}{k}\right)\frac{nH}{z}-\sum_{k=1}^d\sum_{i=0}^{n-1}\frac{H-\chi_i}{kz}+O(1/z^2)\right).
\end{split}
\end{equation}
In the specialization (\ref{eqn:specialization_KP}), we have $\sum_{i=0}^{n-1}\chi_i=0$. Thus the above becomes
\begin{equation*}
\begin{split}
&1+\sum_{d\geq 1}q^d(-1)^{nd}\frac{nH(nd-1)!}{(d!)^n}\frac{1}{z}\left(1+\left(\sum_{k=1}^{nd-1}\frac{1}{k}\right)\frac{nH}{z}-\sum_{k=1}^d\frac{nH}{kz}+O(1/z^2)\right)\\
=&1+\sum_{d\geq 1}q^d(-1)^{nd}\frac{nH(nd-1)!}{(d!)^n}\frac{1}{z}\left(1+\frac{nH}{z}\left(\left(\sum_{k=1}^{nd-1}\frac{1}{k}\right)-\sum_{k=1}^d\frac{1}{k}\right)+O(1/z^2)\right)\\
=&1+\frac{1}{z}\underbrace{n\Phi_0(q)}_{I^{\KP^{n-1}}_1}H+\frac{1}{z^2}\underbrace{n^2\Phi_1(q)}_{I^{\KP^{n-1}}_2}H^2+O(1/z^3),
\end{split}
\end{equation*}
where
\begin{equation*}
\Phi_0(q)=\sum_{d\geq 1}q^d(-1)^{nd}\frac{(nd-1)!}{(d!)^n}, \quad \Phi_1(q)=\sum_{d\geq 1}q^d(-1)^{nd}\frac{(nd-1)!}{(d!)^n}\left(\left(\sum_{k=1}^{nd-1}\frac{1}{k}\right)-\sum_{k=1}^d\frac{1}{k}\right).    
\end{equation*}

Define the operator $$\mathsf{D}_{K\mathbb{P}^{n-1}}:\mathbb{C}[\![q]\!]\rightarrow \mathbb{C}[\![q]\!]$$ and its inverse $$\mathsf{D}_{K\mathbb{P}^{n-1}}^{-1}:q\mathbb{C}[\![q]\!]\rightarrow q\mathbb{C}[\![q]\!]$$ by
\begin{equation*}
\mathsf{D}_{K\mathbb{P}^{n-1}}f(q)=q\frac{df(q)}{dq}, \quad \mathsf{D}_{K\mathbb{P}^{n-1}}^{-1}f(q)=\int_{0}^q\frac{f(t)}{t}dt.
\end{equation*}
Set
\begin{equation*}
E^{\KP^{n-1}}(q,z)\coloneqq I^{K\mathbb{P}^{n-1}}(q,z)\big\vert_{H=1}=\mathcal{F}_{-1}\left(z^{-1},(-1)^nq\right).
\end{equation*}
Taking this change of variables into account, we define the operator $\mathsf{M}$ by
\begin{equation}\label{eqn:defnM}
\mathsf{M}F(q,z)=z\mathsf{D}_z\left(\frac{F(q,z)}{F(q,\infty)}\right)\quad\text{where}\quad \mathsf{D}_z=\frac{1}{z}+\mathsf{D}_{K\mathbb{P}^{n-1}}.
\end{equation}
Define
\begin{equation}\label{eqn:defnEC}
E_i^{\KP^{n-1}}(q,z)=\mathsf{M}^iE^{\KP^{n-1}}(q,z)\quad\text{and}\quad C^{\KP^{n-1}}_i(q)=E_i^{\KP^{n-1}}(q,\infty)\quad\text{for }i\geq{0}.
\end{equation}
(Note that $C^{\KP^{n-1}}_0=1$.) Then, by equation (\ref{eqn:I_KP_interms_of_F_minus_1}), Theorem 1  and Theorem 2 of \cite{zz} directly imply the following result.

\begin{lem}\label{lem:properties_of_C_functions}
For the series $C^{\KP^{n-1}}_i\in\mathbb{C}[\![q]\!]$, we have
\begin{enumerate}
    \item $C^{\KP^{n-1}}_{i+n}=C^{\KP^{n-1}}_i$ for $i\geq{1}$,
    \item $\prod_{i=1}^nC^{\KP^{n-1}}_i=\prod_{i=0}^nC^{\KP^{n-1}}_i=(L^{\KP^{n-1}})^n$,
    \item $C^{\KP^{n-1}}_i=C^{\KP^{n-1}}_{n+1-i}$ for $1\leq{i}\leq{n}$
\end{enumerate}
where $$L^{\KP^{n-1}}=(1-(-n)^nq)^{-1/n}\in 1+q\mathbb{Q}[\![q]\!].$$
\end{lem}

We now describe an equivalent way to define $C^{\KP^{n-1}}_i$. First note  
\begin{equation}\label{eqn:commutation_of_D_q_and_elogq}
e^{\frac{H}{z}\log q}\left(z\mathsf{D}_{K\mathbb{P}^{n-1}}+H\right)F(q,z)=z\mathsf{D}_{K\mathbb{P}^{n-1}}\left(e^{\frac{H}{z}\log q}F(q,z)\right),    
\end{equation}
and
\begin{equation*}
F(q,\infty)=\left(e^{\frac{H}{z}\log q}F(q,z)\right)\big\vert_{z=\infty}.
\end{equation*}
Now define the operator $\widetilde{\mathsf{M}}$ via
\begin{equation}\label{eqn:defntM}
\widetilde{\mathsf{M}}F(q,z)=z\mathsf{D}_{K\mathbb{P}^{n-1}}\left(\frac{F(q,z)}{F(q,\infty)}\right).
\end{equation}
Observe the following fact
\begin{equation*}
\begin{aligned}
e^{\frac{\log q}{z}}\mathsf{M}F(q,z)
&=ze^{\frac{\log q}{z}}\mathsf{D}_z\left(\frac{F(q,z)}{F(q,\infty)}\right)\\
&=z\mathsf{D}_{K\mathbb{P}^{n-1}}\left(e^{\frac{\log q}{z}}\frac{F(q,z)}{F(q,\infty)}\right)\\
&=z\mathsf{D}_{K\mathbb{P}^{n-1}}\left(\frac{e^{\frac{\log q}{z}}F(q,z)}{\left(e^{\frac{H}{z}\log q}F(q,z)\right)\big\vert_{z=\infty}}\right)=\widetilde{\mathsf{M}}\mathsf{F}(q,z)
\end{aligned}
\end{equation*}
where
\begin{equation*}
\mathsf{F}(q,z)\coloneqq{e^{\frac{\log q}{z}}F(q,z)}.
\end{equation*}
Hence, inductively we obtain 
\begin{equation*}
e^{\frac{\log q}{z}}\mathsf{M}^iF(q,z)=\widetilde{\mathsf{M}}^i\mathsf{F}(q,z).
\end{equation*}
Then, we see that
\begin{equation}\label{eqn:Ci_defined_by_Ei}
\begin{aligned}
C_i^{\KP^{n-1}}
&=E_i^{\KP^{n-1}}(q,\infty)\\
&=\left(e^{\frac{\log q}{z}}E_i^{\KP^{n-1}}(q,z)\right)\big\vert_{z=\infty}\\
&=\left(e^{\frac{\log q}{z}}\mathsf{M}^iE^{\KP^{n-1}}(q,z)\right)\big\vert_{\infty}=\left(\widetilde{\mathsf{M}}^i\mathsf{E}^{\KP^{n-1}}(q,z)\right)\big\vert_{\infty}
\end{aligned}
\end{equation}
where
\begin{equation*}
\mathsf{E}^{\KP^{n-1}}(q,z)\coloneqq e^{\frac{\log q}{z}}E^{\KP^{n-1}}(q,z).
\end{equation*}

The analysis (\ref{eqn:I_func_analysis_for_mirrorthm}) shows that the small $I$-function ${I}^{\KP^{n-1}}(q,z)$ is of the form
\begin{equation}\label{eqn:}
{I}^{\KP^{n-1}}(q,z)=\sum_{k=0}^\infty {I}^{\KP^{n-1}}_k(q)\left(\frac{H}{z}\right)^k=\sum_{i=0}^{n-1}\tilde{{I}}^{\KP^{n-1}}_i(q,z)H^i,
\end{equation}
and hence
\begin{equation*}
E^{\KP^{n-1}}(q,z)=\sum_{k=0}^\infty \frac{{I}^{\KP^{n-1}}_k(q)}{z^k}.
\end{equation*}
This also implies that $e^{H\log q/z}I^{\KP^{n-1}}(q,z)$ takes the same form:
\begin{equation}\label{eqn:formula_for_mathsfI}
\mathsf{I}^{\KP^{n-1}}(q,z)\coloneqq e^{H\log q/z}I^{\KP^{n-1}}(q,z)=\sum_{k=0}^\infty \mathsf{I}_k^{\KP^{n-1}}(q)\left(\frac{H}{z}\right)^k,
\end{equation}
and hence
\begin{equation}\label{eqn:E_function_as_a_sum_of_Iks}
\mathsf{E}^{\KP^{n-1}}(q,z)=\sum_{k=0}^\infty \frac{\mathsf{I}^{\KP^{n-1}}_k(q)}{z^k}.
\end{equation}
For $i\geq 1$, we can inductively show that
\begin{equation*}
\widetilde{\mathsf{M}}^i\mathsf{E}^{\KP^{n-1}}(q,z)=\sum_{k=i}^{\infty}\frac{1}{z^{k-i}}\mathsf{D}_{K\mathbb{P}^{n-1}}{{\mathfrak{L}}}_{i-1}\ldots{\mathfrak{L}}_0\mathsf{I}^{\KP^{n-1}}_k
\end{equation*}
where
$${\mathfrak{L}}_i=\frac{1}{\mathsf{D}_{K\mathbb{P}^{n-1}}{{\mathfrak{L}}}_{i-1}\ldots{\mathfrak{L}}_0\mathsf{I}^{\KP^{n-1}}_i}\mathsf{D}_{K\mathbb{P}^{n-1}}$$ 
for $i\geq 1$ and ${\mathfrak{L}}_0$ is the identity.
Then, for $i\geq 1$, equation (\ref{eqn:Ci_defined_by_Ei}) implies that we have
\begin{equation}\label{eqn:mathrfrak_L_operator_dfn_of_C_i}
C_i^{\KP^{n-1}}=\mathsf{D}_{K\mathbb{P}^{n-1}}{{\mathfrak{L}}}_{i-1}\ldots{\mathfrak{L}}_0\mathsf{I}^{\KP^{n-1}}_i\quad\text{with}\quad {\mathfrak{L}}_i=\frac{1}{C_i}\mathsf{D}_{K\mathbb{P}^{n-1}}.
\end{equation}

Now, define the following series in $\mathbb{C}[\![q]\!]$:
\begin{equation}\label{eqn:K_i_definition}
K^{\KP^{n-1}}_r=\prod_{i=0}^rC^{\KP^{n-1}}_i\quad\text{for } r\geq{0}.
\end{equation}
From Lemma \ref{lem:properties_of_C_functions}, the following result follows immediately.
\begin{lem}\label{lem:properties_of_K_series}
For the series $K^{\KP^{n-1}}_r\in\mathbb{C}[\![q]\!]$, we have
\begin{enumerate}
    \item $K^{\KP^{n-1}}_{n+r}=(L^{\KP^{n-1}})^nK^{\KP^{n-1}}_r$ for all $r\geq{0}$, in particular $K^{\KP^{n-1}}_n=(L^{\KP^{n-1}})^n$, \label{Kfunctions1}
    \item $K^{\KP^{n-1}}_rK^{\KP^{n-1}}_{n-r}=(L^{\KP^{n-1}})^n$ and $K^{\KP^{n-1}}_rK^{\KP^{n-1}}_{\mathrm{Inv}(r)}=(L^{\KP^{n-1}})^{r+\mathrm{Inv}(r)}$ for all $0\leq r \leq n-1$.
\end{enumerate}
\end{lem}

The Picard--Fuchs equation for $I^{\KP^{n-1}}$ is
\begin{equation}\label{eqn:PF1}
\prod_{i=0}^{n-1}\left(z\mathsf{D}_{K\mathbb{P}^{n-1}}+H-\chi_i\right)I^{\KP^{n-1}}(q,z)=(-1)^nq\prod_{i=0}^{n-1}\left(n\left(z\mathsf{D}_{K\mathbb{P}^{n-1}}+H\right)+iz\right)I^{\KP^{n-1}}(q,z).    
\end{equation}
Using the specialization (\ref{eqn:specialization_KP}), we may rewrite this as
\begin{equation}\label{eqn:PF2}
\left(\left(z\mathsf{D}_{K\mathbb{P}^{n-1}}+H\right)^n-1\right)I^{\KP^{n-1}}(q,z)=(-1)^nq\prod_{i=0}^{n-1}\left(n\left(z\mathsf{D}_{K\mathbb{P}^{n-1}}+H\right)+iz\right)I^{\KP^{n-1}}(q,z).        
\end{equation}
The $n=3$ case of (\ref{eqn:PF2}) is \cite[Equation (26)]{lho-p}.

By equation (\ref{eqn:commutation_of_D_q_and_elogq}), we have 
\begin{equation}\label{eqn:PF3}
\left(\left(z\mathsf{D}_{K\mathbb{P}^{n-1}}\right)^n-1\right)\mathsf{I}^{\KP^{n-1}}(q,z)=(-1)^nq\prod_{i=0}^{n-1}\left(n\left(z\mathsf{D}_{K\mathbb{P}^{n-1}}\right)+iz\right)\mathsf{I}^{\KP^{n-1}}(q,z).        
\end{equation}
PF equations read as
\begin{equation*}
\begin{aligned}
\left(z^n\mathsf{D}_{\KP^{n-1}}^n-1\right)\mathsf{I}^{\KP^{n-1}}(q,z)
&=(-1)^nq\prod_{i=0}^{n-1}\left(nz\mathsf{D}_{\KP^{n-1}}+iz\right)\mathsf{I}^{\KP^{n-1}}(q,z)\\
&=(-1)^nz^nq\prod_{i=0}^{n-1}\left(n\mathsf{D}_{\KP^{n-1}}+i\right)\mathsf{I}^{\KP^{n-1}}(q,z)\\
&=(-1)^nz^nq\sum_{k=0}^{n}(-1)^{n-k}{s}_{n,k}n^k\mathsf{D}_{\KP^{n-1}}^k\mathsf{I}^{\KP^{n-1}}(q,z)\\
&=(-1)^nz^nq\left(n^n\mathsf{D}_{\KP^{n-1}}^n+\sum_{k=0}^{n-1}(-1)^{n-k}{s}_{n,k}n^k\mathsf{D}_{\KP^{n-1}}^k\right)\mathsf{I}^{\KP^{n-1}}(q,z)
\end{aligned}
\end{equation*}
which is equivalent to
\begin{equation}\label{eqn:PF_Eq_for_with_logq_factor}
\left({(1-(-n)^nq)}\mathsf{D}_{\KP^{n-1}}^n-(-1)^nq\sum_{k=0}^{n-1}(-1)^{n-k}{s}_{n,k}n^k\mathsf{D}_{\KP^{n-1}}^k\right)\mathsf{I}^{\KP^{n-1}}(q,z)=z^{-n}\mathsf{I}^{\KP^{n-1}}(q,z).
\end{equation}
Observe that 
\begin{equation}\label{eqn:DLL}
\begin{aligned}
\mathsf{D}_{\KP^{n-1}}L^{\KP^{n-1}}
&=-\frac{1}{n}(1-(-n)^nq)^{-\frac{1}{n}-1}\left(-(-n)^nq\right)\\
&=\frac{1}{n}(1-(-n)^nq)^{-\frac{1}{n}-1}(-n)^nq\\
&=\frac{1}{n}L^{\KP^{n-1}}\frac{(-n)^nq}{(1-(-n)^nq)}.
\end{aligned}
\end{equation}
So, we obtain 
\begin{equation}\label{eqn:PF_with_stirling_for_vert_I}
\left(\mathsf{D}_{\KP^{n-1}}^n-\frac{\mathsf{D}_{\KP^{n-1}}L^{\KP^{n-1}}}{n^{n-1}L^{\KP^{n-1}}}\sum_{k=0}^{n-1}(-1)^{n-k}{s}_{n,k}n^k\mathsf{D}_{\KP^{n-1}}^k\right)\mathsf{I}^{\KP^{n-1}}(q,z)=\left(\frac{L^{\KP^{n-1}}}{z}\right)^n\mathsf{I}^{\KP^{n-1}}(q,z).
\end{equation}

Also, substituting equation (\ref{eqn:E_function_as_a_sum_of_Iks}) into Picard--Fuchs equation (\ref{eqn:PF_with_stirling_for_vert_I}) and analyzing the coefficients of $z^k$'s on the both sides, we obtain
\begin{equation}\label{eqn:PF_for_Iks_annihilation}
\left(\mathsf{D}_{\KP^{n-1}}^n-\frac{\mathsf{D}_{\KP^{n-1}}L^{\KP^{n-1}}}{n^{n-1}L^{\KP^{n-1}}}\sum_{k=0}^{n-1}(-1)^{n-k}{s}_{n,k}n^k\mathsf{D}_{\KP^{n-1}}^k\right)\mathsf{I}_k^{\KP^{n-1}}=0
\end{equation}
for $0\leq k \leq n-1$.

\subsection{Genus \texorpdfstring{$0$}{0} invariants}\label{sec:genus0KP}
Consider the (small) $J$-function of $\KP^{n-1}$:
\begin{equation}\label{eqn:small_J}
J^{\KP^{n-1}}(Q,z)=1+\sum_{j=0}^{n-1}\sum_{d\neq 0}Q^d\left\langle\frac{H^j}{z(z-\psi)} \right\rangle_{0,1,d}^{\KP^{n-1}}(H^j)^\vee.    
\end{equation}

The mirror theorem (as a consequence of the main result of \cite{ccit}) implies the equality
\begin{equation}\label{eqn:mirror_theorem_small}
e^{H\log Q/z}J^{\KP^{n-1}}(Q,z) = e^{H\log q/z}I^{\KP^{n-1}}(q,z),
\end{equation}
subject to the change of variables (mirror map)
\begin{equation}\label{eqn:mirror_map_small}
\log Q= \log q + n\Phi_0(q)=\mathsf{I}_1^{\KP^{n-1}}(q).    
\end{equation}
Also, $ze^{H\log Q/z}J^{\KP^{n-1}}(Q,z)$ lies on Givental's Lagrangian cone for $\KP^{n-1}$.

Extracting the $1/z^2$-term of $J^{\KP^{n-1}}(q,z)$, using (\ref{eqn:mirror_theorem_small}), we have 
\begin{equation}\label{eqn:2pt0_small}
\begin{split}
\sum_{j=0}^{n-1}\sum_{d\neq 0} Q^d\langle H^j\rangle_{0,1,d}^{\KP^{n-1}}(H^j)^\vee&= 
\left(I_2^{\KP^{n-1}}(q)-\frac{1}{2}I_1^{\KP^{n-1}}(q)^2\right)H^2\\
&=\left(\mathsf{I}_0^{\KP^{n-1}}(q)\frac{(\log Q)^2}{2}-\mathsf{I}_1^{\KP^{n-1}}(q)\log Q+\mathsf{I}_2^{\KP^{n-1}}(q)\right)H^2.  
\end{split}
\end{equation}

Consider $S^{\KP^{n-1}}(Q,z)$ defined by 
\begin{equation}\label{eqn:defnS}
g^{\KP^{n-1}}(a, S^{\KP^{n-1}}(Q,z)(b)):=g^{\KP^{n-1}}(a,b)+\sum_{k=0}^\infty \frac{1}{z^{1+k}}\sum_{d\neq 0}Q^d\langle a, b\psi^k \rangle_{0,2,d}^{\KP^{n-1}}.        
\end{equation}
Then, by e.g. the discussion following \cite[Equation (5)]{g3}, we have 
\begin{equation}
J^{\KP^{n-1}}(Q, z)=S^{\KP^{n-1}}(Q, z)^*(1).   
\end{equation}

Properties of Givental's cone imply that for $i \geq 1$,
\begin{equation}\label{eqn:BFi_small}
\begin{split}
e^{H\log Q/z}S^{\KP^{n-1}}(Q, z)^*(H^i)&=\frac{\left(z\mathsf{D}_{K\mathbb{P}^{n-1}}\right)\left(e^{H\log Q/z}S^{\KP^{n-1}}(Q, z)^*(H^{i-1})\right)}{\left(z\mathsf{D}_{K\mathbb{P}^{n-1}}\right)\left(e^{H\log Q/z}S^{\KP^{n-1}}(Q, z)^*(H^{i-1})\right)\big|_{H=1,z=\infty}}\\
&=H^i+\sum_{k\geq 1} \mathsf{C}_{i,k}(q)H^{i+k}z^{-k}.
\end{split}
\end{equation}

Here
\begin{equation}\label{eqn:C_ik_inductive_relation}
\mathsf{C}_{i,k}=\frac{\mathsf{D}_{K\mathbb{P}^{n-1}}\mathsf{C}_{i-1,k+1}}{\mathsf{D}_{K\mathbb{P}^{n-1}}\mathsf{C}_{i-1,1}}, \quad k\geq 1.    
\end{equation}

We find 
\begin{equation}\label{eqn:2pt1_small}
\sum_{j=0}^{n-1}\sum_{d\neq 0}Q^d\langle H, H^j\rangle_{0,2,d}^{\KP^{n-1}}(H^j)^\vee=\left(\frac{q\frac{d}{dq}\mathsf{I}^{\KP^{n-1}}_2}{q\frac{d}{dq}\mathsf{I}^{\KP^{n-1}}_1} -\log Q\right) H^2. \end{equation}
More generally, 
\begin{equation}\label{eqn:2pti_small}
\sum_{j=0}^{n-1}\sum_{d\neq 0}Q^d\langle H^i, H^j\rangle_{0,2,d}^{\KP^{n-1}}(H^j)^\vee=\left(\mathsf{C}_{i,1} -\log Q\right) H^{i+1}.    
\end{equation}
By the divisor equation, 
\begin{equation*}
\begin{split}
\sum_{j=0}^{n-1}\sum_{d\neq 0}Q^d\langle H, H^i, H^j\rangle_{0,3,d}^{\KP^{n-1}}(H^j)^\vee
=&\sum_{j=0}^{n-1}\sum_{d\neq 0}dQ^d\langle H^i, H^j\rangle_{0,2,d}^{\KP^{n-1}}(H^j)^\vee\\
=&Q\frac{d}{dQ}\sum_{j=0}^{n-1}\sum_{d\neq 0}Q^d\langle H^i, H^j\rangle_{0,2,d}^{\KP^{n-1}}(H^j)^\vee=\left(Q\frac{d}{dQ}\mathsf{C}_{i,1}-1\right)H^{i+1}.
\end{split}
\end{equation*}

By the definition of small quantum product $\bullet$, we have 
\begin{equation*}
H\bullet H^i= H\cdot H^i+\sum_{j=0}^{n-1}\sum_{d\neq 0}Q^d\langle H, H^i, H^j\rangle_{0,3,d}^{\KP^{n-1}}(H^j)^\vee=  \left(Q\frac{d}{dQ}\mathsf{C}_{i,1}\right)H^{i+1}. \end{equation*}

Thus, by associativity of $\bullet$, we have
\begin{equation*}
\begin{split}
\underbrace{H\bullet...\bullet H}_{i}&=\left(Q\frac{d}{dQ}\mathsf{C}_{1,1}\right)(H^2\bullet \underbrace{H)\bullet...\bullet H}_{i-2}\\
&=\left(Q\frac{d}{dQ}\mathsf{C}_{1,1}Q\frac{d}{dQ}\mathsf{C}_{2,1}\right)(H^3\bullet \underbrace{H)\bullet...\bullet H}_{i-3}\\
&=...=\left(\prod_{k=1}^{i-1}Q\frac{d}{dQ}\mathsf{C}_{k,1}\right)H^i,
\end{split}
\end{equation*}
and
\begin{equation*}
H^i\bullet H^j=\frac{\left(\prod_{k=1}^{i+j-1}Q\frac{d}{dQ}\mathsf{C}_{k,1}\right)}{\left(\prod_{k=1}^{i-1}Q\frac{d}{dQ}\mathsf{C}_{k,1}\right)\left(\prod_{k=1}^{j-1}Q\frac{d}{dQ}\mathsf{C}_{k,1}\right)} H^{i+j}.  \end{equation*}

By (\ref{eqn:mirror_map_small}), we have 
\begin{equation}\label{eqn:QddQ_to_qddq}
Q\frac{d}{dQ}=\frac{q\frac{d}{dq}}{q\frac{d}{dq}\mathsf{I}_1^{\KP^{n-1}}}=\frac{1}{\mathsf{D}_{K\mathbb{P}^{n-1}}\mathsf{C}_{0,1}}\mathsf{D}_{K\mathbb{P}^{n-1}},
\end{equation}
so we see that
\begin{equation*}
Q\frac{d}{dQ}\mathsf{C}_{k,1}=\frac{\mathsf{D}_{K\mathbb{P}^{n-1}}\mathsf{C}_{k,1}}{\mathsf{D}_{K\mathbb{P}^{n-1}}\mathsf{C}_{0,1}}, \end{equation*}
and
\begin{equation*}
H^i\bullet H^j=\frac{\left(\prod_{k=0}^{i+j-1}\mathsf{D}_{K\mathbb{P}^{n-1}}\mathsf{C}_{k,1}\right)}{\left(\prod_{k=0}^{i-1}\mathsf{D}_{K\mathbb{P}^{n-1}}\mathsf{C}_{k,1}\right)\left(\prod_{k=0}^{j-1}\mathsf{D}_{K\mathbb{P}^{n-1}}\mathsf{C}_{k,1}\right)} H^{i+j}. \end{equation*}

\begin{lem}
For all $i\geq{1}$, we have
\begin{equation*}
\mathsf{D}_{K\mathbb{P}^{n-1}}\mathsf{C}_{i-1,1}=C^{\KP^{n-1}}_i.
\end{equation*}
\end{lem}
\begin{proof}
We do induction on $i$. For the base case $i=1$, observe that we have
\begin{equation*}
\mathsf{D}_{K\mathbb{P}^{n-1}}\mathsf{C}_{0,1}=\mathsf{D}_{K\mathbb{P}^{n-1}}\mathsf{I}_1^{\KP^{n-1}}=C_1^{\KP^{n-1}}.
\end{equation*}
by equation (\ref{eqn:mathrfrak_L_operator_dfn_of_C_i}).

For the inductive step we have
\begin{equation*}
\begin{split}
\mathsf{C}_{i-1,1}=\frac{\mathsf{D}_{K\mathbb{P}^{n-1}}\mathsf{C}_{i-2,2}}{\mathsf{D}_{K\mathbb{P}^{n-1}}\mathsf{C}_{i-2,1}}
&=\frac{1}{C^{\KP^{n-1}}_{i-1}}\mathsf{D}_{K\mathbb{P}^{n-1}}\mathsf{C}_{i-2,2}\\
&=\frac{1}{C^{\KP^{n-1}}_{i-1}}\mathsf{D}_{K\mathbb{P}^{n-1}}\left(\frac{\mathsf{D}_{K\mathbb{P}^{n-1}}\mathsf{C}_{i-3,3}}{\mathsf{D}_{K\mathbb{P}^{n-1}}\mathsf{C}_{i-3,1}}\right)\\
&=\frac{1}{C^{\KP^{n-1}}_{i-1}}\mathsf{D}_{K\mathbb{P}^{n-1}}\left(\frac{1}{C^{\KP^{n-1}}_{i-2}}\mathsf{D}_{K\mathbb{P}^{n-1}}\mathsf{C}_{i-3,3}\right)\\
&\,\,\,\vdots\\
&=\mathfrak{L}_{i-1}\cdots\mathfrak{L}_0\mathsf{C}_{0,i}=\mathfrak{L}_{i-1}\cdots\mathfrak{L}_0\mathsf{I}^{\KP^{n-1}}_i
\end{split}
\end{equation*}
So, we get $\mathsf{D}_{\KP^{n-1}}\mathsf{C}_{i-1,1}=\mathsf{D}_{\KP^{n-1}}\mathfrak{L}_{i-1}\cdots\mathfrak{L}_0\mathsf{I}^{\KP^{n-1}}_i$
which is $C^{\KP^{n-1}}_i$ by equation (\ref{eqn:mathrfrak_L_operator_dfn_of_C_i}). 
\end{proof}

It follows that we can rewrite equation (\ref{eqn:mirror_theorem_small}) and equation (\ref{eqn:BFi_small}) as
\begin{equation}\label{eqn:Birkhoff_Fact_with_L_i}
\begin{split}
e^{H\log Q/z}S^{\KP^{n-1}}(Q, z)^*(1)&=\mathsf{I}^{\KP^{n-1}}(q,z)\\
e^{H\log Q/z}S^{\KP^{n-1}}(Q, z)^*(H^i)&=z\mathfrak{L}_ie^{H\log Q/z}S^{\KP^{n-1}}(Q, z)^*(H^{i-1})\quad\text{for all }i\geq{1},
\end{split}
\end{equation}
and we have 
\begin{equation*}
\prod_{k=0}^{r-1}\mathsf{D}_{K\mathbb{P}^{n-1}}\mathsf{C}_{k,1}=\prod_{k=1}^rC^{\KP^{n-1}}_k=\prod_{k=0}^rC^{\KP^{n-1}}_k=K^{\KP^{n-1}}_r.
\end{equation*}
Hence, we see that the small quantum product is given by
\begin{equation*}
H^i\bullet H^j=\frac{K^{\KP^{n-1}}_{i+j}}{K^{\KP^{n-1}}_{i}K^{\KP^{n-1}}_{j}} H^{i+j}.
\end{equation*}
This equation holds for any $i,j\geq{0}$ by the properties of functions $K^{\KP^{n-1}}_{i}$ given in Lemma \ref{lem:properties_of_K_series}.

By the definition of small quantum product $\bullet$, we have 
\begin{equation}\label{eqn:3ptGWKP}
\begin{split}
\sum_{d=0}^{\infty}Q^d\langle H^i, H^j, H^k\rangle_{0,3,d}^{\KP^{n-1}}=g^{\KP^{n-1}}(H^i\bullet H^j, H^k)
&=\frac{K^{\KP^{n-1}}_{i+j}}{K^{\KP^{n-1}}_{i}K^{\KP^{n-1}}_{j}} g^{\KP^{n-1}}(H^{i+j}, H^k)\\
&=-\frac{1}{n}\frac{K^{\KP^{n-1}}_{i+j}}{K^{\KP^{n-1}}_{i}K^{\KP^{n-1}}_{j}}\delta_{\text{Inv}(i+j \text{ mod }n), k}.
\end{split}
\end{equation}

\subsection{Frobenius Structures}\label{sec:FrobKP}
Let $\gamma=\sum_{i=0}^{n-1}\tau_iH^i\in H^{\star}_{\mathrm{T}}\left(\KP^{n-1}\right)$. Then, the full genus $0$ Gromov--Witten potential is defined to be 
\begin{equation}\label{eqn:full_GW_potential}
\begin{split}
    \mathcal{F}_0^{\KP^{n-1}}(\mathbf{\tau}, Q)
    =\sum_{m=0}^{\infty} \sum_{d=0}^{\infty} \frac{Q^d}{m !}\int_{\left[\overline{M}_{0, m}\left(\KP^{n-1}, d\right)\right]^{\mathrm{vir}}}  \prod_{i=1}^m \operatorname{ev}_i^*(\gamma)
    =\sum_{m=0}^{\infty} \sum_{d=0}^{\infty} \frac{Q^d}{m !}\left\langle \underbrace{\gamma,...,\gamma}_{m}\right\rangle_{0, m, d}^{\KP^{n-1}}.
\end{split}
\end{equation}
Let the $R$-matrix of the Frobenius manifold\footnote{The Frobenius manifold here is over the ring $\mathbb{C}[\![Q]\!]$, or can be considered over the ring $\mathbb{C}[\![q]\!]$ by the mirror map (\ref{eqn:mirror_map_small}).} structure associated to the ($\mathrm{T}$-equivariant) Gromov--Witten theory of $\KP^{n-1}$ near the semisimple point $0\in H^*_{\mathrm{T}}(\KP^{n-1}) $ be denoted by 
\begin{equation*}
    \mathsf{R}^{\KP^{n-1}}(z)=\text{Id}+\sum_{k\geq 1} R^{\KP^{n-1}}_k z^k\in \text{End}(H^*_{\mathrm{T}}(\KP^{n-1}))[\![z]\!].
\end{equation*}
The $R$-matrix plays a crucial role in the Givental--Teleman classification of semisimple cohomological field theories. By the definition, $R$-matrix satisfies the symplectic condition 
\begin{equation*}
    \mathsf{R}^{\KP^{n-1}}(z)\cdot \mathsf{R}^{\KP^{n-1}}(-z)^*=\mathsf{Id},
\end{equation*}
where $(-)^*$ adjoint with respect to metric $g^{\KP^{n-1}}$.

For all $i\geq{0}$, define
\begin{equation}\label{def:phi_tilde}
\widetilde{H}_i=\frac{K^{\KP^{n-1}}_i}{(L^{\KP^{n-1}})^i}H^i.
 \end{equation}
This is a normalization of $H^i$'s in the sense that we have $\widetilde{H}_{i+n}=\widetilde{H}_{i}$ and $\widetilde{H}_i\bullet\widetilde{H}_j=\widetilde{H}_{i+j}$ for all $i,j\geq{0}$. As a result, the quantum product at $0\in H^{\star}_{\mathrm{T}}\left(\KP^{n-1}\right)$ is semisimple with the idempotent basis $\{e_\alpha\}$ given by 
\begin{equation}\label{def:idempotent}
    e_{\alpha}=\frac{1}{n}\sum_{i=0}^{n-1}\zeta^{-\alpha{i}}\widetilde{H}_i\quad\text{for}\quad 0\leq\alpha\leq n-1,
\end{equation}
where $\zeta=e^{\frac{2\pi\sqrt{-1}}{n}}$ is an $n^{\text{th}}$ root of unity.

We calculate the metric $g^{\KP^{n-1}}$ in the idempotent basis $\{e_\alpha\}$:
\begin{align*}
 g^{\KP^{n-1}}\left(e_{\alpha},e_{\alpha}\right)
 =&g^{\KP^{n-1}}\left(\frac{1}{n}\sum_{i=0}^{n-1}\zeta^{-\alpha{i}}\frac{K^{\KP^{n-1}}_i}{(L^{\KP^{n-1}})^i}{H}^i,\frac{1}{n}\sum_{j=0}^{n-1}\zeta^{-\alpha{j}}\frac{K^{\KP^{n-1}}_j}{(L^{\KP^{n-1}})^j}{H}^j\right)\\
 =&-\frac{1}{n^2}\sum_{i=0}^{n-1}\sum_{j=0}^{n-1}\zeta^{-\alpha{(i+j)}}\frac{K^{\KP^{n-1}}_iK^{\KP^{n-1}}_j}{(L^{\KP^{n-1}})^{i+j}}\frac{1}{n}\delta_{\mathrm{Inv}(i)j}\\
  =&-\frac{1}{n^3}\sum_{i=0}^{n-1}\zeta^{-\alpha{(i+{\mathrm{Inv}(i)})}}\frac{K_iK_{\mathrm{Inv}(i)}}{(L^{\KP^{n-1}})^{i+{\mathrm{Inv}(i)}}}
 =-\frac{1}{n^2},
\end{align*}
where the last equality follows from  Lemma \ref{lem:properties_of_K_series}, and by the identity $$i+{\mathrm{Inv}(i)}=0\mod{n}.$$
The normalized idempotents are
\begin{equation}\label{def:normalized_idempotent}
\widetilde{e}_{\alpha}=\frac{e_{\alpha}}{\sqrt{g\left(e_{\alpha},e_{\alpha}\right)}}=\frac{e_{\alpha}}{\sqrt{-\frac{1}{n^2}}}=-n\sqrt{-1}e_{\alpha}.
\end{equation}
The transition matrix $\Psi$ is given by $\Psi_{\alpha{i}}=g^{\KP^{n-1}}\left(\widetilde{e}_{\alpha},H^i\right)$ where $0\leq \alpha,i\leq n-1$. We calculate
\begin{equation*}
\Psi_{\alpha{i}}
=g^{\KP^{n-1}}\left(\widetilde{e}_{\alpha},H^i\right)
=g^{\KP^{n-1}}\left(-\sqrt{-1}\sum_{j=0}^{n-1}\zeta^{-\alpha{j}}\frac{K^{\KP^{n-1}}_j}{(L^{\KP^{n-1}})^j}H^j,H^i\right)
=\sqrt{-1}\sum_{j=0}^{n-1}\zeta^{-\alpha{j}}\frac{K^{\KP^{n-1}}_j}{(L^{\KP^{n-1}})^j}\frac{1}{n}\delta_{\mathrm{Inv}(i),j}.
\end{equation*}
So, $\Psi_{\alpha{i}}$ is given by
\begin{equation*}
\begin{split}
\Psi_{\alpha{i}}
&=\frac{\sqrt{-1}}{n}\zeta^{-\alpha{\mathrm{Inv}(i)}}\frac{K^{\KP^{n-1}}_{\mathrm{Inv}(i)}}{(L^{\KP^{n-1}})^{\mathrm{Inv}(i)}}=\frac{\sqrt{-1}}{n}\zeta^{-\alpha(n-i)}\frac{K^{\KP^{n-1}}_{n-i}}{(L^{\KP^{n-1}})^{n-i}}\\
&=\frac{\sqrt{-1}}{n}\zeta^{\alpha{i}}\frac{(L^{\KP^{n-1}})^{i}}{K^{\KP^{n-1}}_{i}}\quad\text{for}\quad 0\leq \alpha,i\leq n-1. 
\end{split}
\end{equation*}
The inverse of the transition matrix $\Psi^{-1}=\left[\Psi^{-1}_{\beta{j}}\right]$ is given by
\begin{equation*}
\Psi^{-1}_{{j}\beta}=-\sqrt{-1}\zeta^{-\beta{j}}\frac{K^{\KP^{n-1}}_{j}}{(L^{\KP^{n-1}})^{j}}\quad\text{where}\quad 0\leq\beta,j\leq n-1.
\end{equation*}
Let $\left\{u^{\alpha}\right\}_{\alpha=0}^{n-1}$ be canonical coordinates associated to the idempotent basis $\left\{e_{\alpha}\right\}_{\alpha=0}^{n-1}$. Since $e_{\alpha}=\frac{\partial}{\partial{u^{\alpha}}}$, we have
\begin{equation}\label{eq:idempotenttophi_1}
\sum_{\alpha=0}^{n-1}\frac{\partial {u^{\alpha}}}{\partial{\tau_1}}e_{\alpha}=H.
\end{equation}

\begin{lem}\label{lem:canonicalcoorder}
We have, at $\mathbf{\tau}=0$,
\begin{equation*}
\frac{du^{\alpha}}{d{\tau_1}}=\zeta^{\alpha}\frac{L^{\KP^{n-1}}}{C^{\KP^{n-1}}_1}.
\end{equation*}
\end{lem}

\begin{proof}
The result is obtained by the following calculation: at $\mathbf{\tau}=0$, we have
\begin{align*}
\frac{du^{\alpha}}{d{\tau_1}}{e_{\alpha}}=\sum_{\beta=0}^{n-1}\frac{d {u^{\beta}}}{d{\tau_1}}\delta_{\alpha,\beta}{e_{\alpha}}=H\bullet{e_{\alpha}}
=&\frac{1}{n}\sum_{i=0}^{n-1}\zeta^{-\alpha{i}}\frac{K^{\KP^{n-1}}_i}{(L^{\KP^{n-1}})^i}{H^i}\bullet{H}\\
=&\zeta^{\alpha}\frac{L^{\KP^{n-1}}}{C^{\KP^{n-1}}_1}\frac{1}{n}\sum_{i=0}^{n-1}\zeta^{-\alpha{(i+1)}}\frac{K^{\KP^{n-1}}_{i+1}}{(L^{\KP^{n-1}})^{i+1}}H^{i+1}\\
=&\zeta^{\alpha}\frac{L^{\KP^{n-1}}}{C^{\KP^{n-1}}_1}\underbrace{\frac{1}{n}\sum_{i=0}^{n-1}\zeta^{-\alpha{i}}\frac{K^{\KP^{n-1}}_i}{(L^{\KP^{n-1}})^i}H^i}_{=e_{\alpha}}.
\end{align*}
\end{proof}

Let $U$ be the diagonal matrix
\begin{equation*}
U=\diag(u^0,\ldots,u^{n-1}).
\end{equation*}
Then, the $R$-matrix also satisfies the following flatness equation
\begin{equation}\label{eqn:Basic_flatness_d}
z(d\Psi^{-1})R+z\Psi^{-1}(dR)+\Psi^{-1}R (dU)-\Psi^{-1}(dU) R=0,    
\end{equation}
see \cite[Chapter 1, Section 4.6]{lp} and \cite[Proposition 1.1]{g1}. Here, $d=\frac{d}{d\tau}$.
Note that the full genus $0$ potential (\ref{eqn:full_GW_potential}) is annihilated by the operator
\begin{equation}\label{eqn:annihilatoroperator}
\frac{\partial}{\partial{\tau_1}}-Q\frac{\partial}{\partial{Q}}.
\end{equation}
Similarly, this operator annihilates canonical coordinates $u^{\alpha}$. Hence, at $0\in H^{\star}_{\mathrm{T}}\left(\KP^{n-1}\right)$, we have
\begin{equation*}
\frac{du^{\alpha}}{d{\tau_1}}=Q\frac{du^{\alpha}}{dQ}=\frac{1}{C^{\KP^{n-1}}_1}q\frac{du^{\alpha}}{dq}
\end{equation*}
where the second equality follows from the mirror map (\ref{eqn:mirror_map_small}). Then, by Lemma \ref{lem:canonicalcoorder} we obtain
\begin{equation}
q\frac{du^{\alpha}}{dq}=L^{\KP^{n-1}}\zeta^{\alpha},
\end{equation}
so, we have
\begin{equation}\label{eqn:DU}
\mathsf{D}_{\KP^{n-1}}U=q\frac{d}{dq}U=\diag(L^{\KP^{n-1}},{\zeta}L^{\KP^{n-1}},...,\zeta^{n-1}L^{\KP^{n-1}}).
\end{equation}

The operator (\ref{eqn:annihilatoroperator}) also annihilates the transition matrix $\Psi$, and the $R$-matrix $\mathsf{R}^{\KP^{n-1}}(z)$.  When restricted to the line along $\tau_{i\neq 1}=0$, the flatness equation (\ref{eqn:Basic_flatness_d}) takes of the form
\begin{equation*}
   z(q\frac{d}{dq}\Psi^{-1})\mathsf{R}^{\KP^{n-1}}+z\Psi^{-1}(q\frac{d}{dq}\mathsf{R}^{\KP^{n-1}})+\Psi^{-1}\mathsf{R}^{\KP^{n-1}} (q\frac{d}{dq}U)-\Psi^{-1}(q\frac{d}{dq}U) \mathsf{R}^{\KP^{n-1}}=0
\end{equation*}
via the annihilation of $U$, $\Psi$, and $\mathsf{R}^{\KP^{n-1}}$ by the operator (\ref{eqn:annihilatoroperator}).
By equating coefficients of $z^k$, and multiplying with $\Psi^{-1}$, we obtain the following
\begin{equation}\label{eqn:Flatness}
\mathsf{D}_{\KP^{n-1}}\left(\Psi^{-1}R^{\KP^{n-1}}_{k-1}\right)+\left(\Psi^{-1}R^{\KP^{n-1}}_k\right)\mathsf{D}_{\KP^{n-1}}U-\Psi^{-1}\left(\mathsf{D}_{\KP^{n-1}}U\right)\Psi\left(\Psi^{-1}R^{\KP^{n-1}}_k\right)=0.
\end{equation}
Let $P_{i,j}^{k,\KP^{n-1}}$ denote the $(i,j)$ entry of the coefficient of $z^k$ in the matrix series defined by 
\begin{equation}\label{eqn:Flat_to_modflat_KP_1}
\mathsf{P}^{\KP^{n-1}}(z)=\Psi^{-1}\mathsf{R}^{\KP^{n-1}}(z)=\sum_{k=0}^{\infty} P_k^{\KP^{n-1}} z^{k}
\end{equation}
after being restricted to the semisimple point $0\in H^{\star}_{\mathrm{T}}\left(\KP^{n-1}\right)$ where $0\leq i,j \leq{n-1}$ and $k\geq 0$.
Then, equation (\ref{eqn:Flatness}) reads as
\begin{equation*}\label{eqn:Flatness2}
\mathsf{D}_{\KP^{n-1}}P^{\KP^{n-1}}_k=\Psi^{-1}\left(\mathsf{D}_{\KP^{n-1}}U\right)\Psi{P^{\KP^{n-1}}_k}-P^{\KP^{n-1}}_k\mathsf{D}_{\KP^{n-1}}U.
\end{equation*}

\begin{lem}\label{lem:P_Flatness_with_C}
For ${0\leq{i,j}\leq{n-1}}$ and $k\geq{0}$, we have
\begin{equation*}
\mathsf{D}_{\KP^{n-1}}P_{i,j}^{k-1,\KP^{n-1}}=
C_{\mathrm{Ion}(i)}^{\KP^{n-1}}P_{\mathrm{Ion}(i)-1,j}^{k,\KP^{n-1}}-P_{i,j}^{k,\KP^{n-1}}L^{\KP^{n-1}}\zeta^j.
\end{equation*}
\end{lem}

\begin{proof}
Observe the following computation:
\begin{align*}
\left(\Psi^{-1}\left(\mathsf{D}_{\KP^{n-1}}U\right)\Psi\right)_{ij}
&=\sum_{r=0}^{n-1}\left(\Psi^{-1}\left(\mathsf{D}_{\KP^{n-1}}U\right)\right)_{ir}\Psi_{rj}\\
&=\sum_{r=0}^{n-1}-\sqrt{-1}\zeta^{-r{i}}\frac{K^{\KP^{n-1}}_{i}}{(L^{\KP^{n-1}})^{i}}{\zeta^r}L^{\KP^{n-1}}\frac{\sqrt{-1}}{n}\zeta^{rj}\frac{(L^{\KP^{n-1}})^j}{K^{\KP^{n-1}}_j}\\
&=\frac{1}{n}\frac{K^{\KP^{n-1}}_i}{K^{\KP^{n-1}}_j}\frac{(L^{\KP^{n-1}})^{j+1}}{(L^{\KP^{n-1}})^i}\sum_{l=0}^{n-1}\zeta^{r(j-i+1)}\\
&=\begin{cases}
  \frac{K^{\KP^{n-1}}_i}{K^{\KP^{n-1}}_j}\frac{(L^{\KP^{n-1}})^{j+1}}{(L^{\KP^{n-1}})^i}&\quad\text{if}\quad i=j+1\mod{n},\\
  0&\quad\text{otherwise}
  \end{cases}\\
&=
  \begin{cases}
  C^{\KP^{n-1}}_i&\quad\text{if}\quad 1\leq i\leq {n-1}\quad\text{and}\quad j={i-1},\\
  C^{\KP^{n-1}}_n&\quad\text{if}\quad i=0\quad\text{and}\quad j={n-1},\\
  0&\quad\text{otherwise}
  \end{cases}
=C^{\KP^{n-1}}_{\mathrm{Ion}(i)}\delta_{\mathrm{Ion}(i)-1,j}
\end{align*}
where the last equality follows from Lemma \ref{lem:properties_of_K_series}.
Then, we have
\begin{equation*}
\begin{aligned}
\left(\Psi^{-1}\left(\mathsf{D}_{\KP^{n-1}}U\right)\Psi{P^{\KP^{n-1}}_k}\right)_{ij}
=\sum_{r=0}^{n-1}\left(\Psi^{-1}\mathsf{D}_{\KP^{n-1}}U\Psi\right)_{ir}P^{k,\KP^{n-1}}_{r,j}
=C^{\KP^{n-1}}_{\mathrm{Ion}(i)}P^{k,\KP^{n-1}}_{\mathrm{Ion}(i)-1,j}.
\end{aligned}
\end{equation*}
The rest of the proof follows from equation (\ref{eqn:Flatness2}).
\end{proof}

For $0\leq{i,j}\leq{n-1}$, define
\begin{equation}\label{eqn:DLj_mutildej}
P_{i,j}^{\KP^{n-1}}(z)=\sum_{k=0}^{\infty} P_{i,j}^{k,\KP^{n-1}} z^{k},\quad \mathsf{D}_{L_j}=\mathsf{D}_{\KP^{n-1}}+\frac{L^{\KP^{n-1}}_j}{z} \quad\text{and}\quad \widetilde{\mu}_j=\int_0^q\frac{L^{\KP^{n-1}}_j(u)}{u}du
\end{equation}
where $L_j^{\KP^{n-1}}=L^{\KP^{n-1}}\zeta^j$. Then, we can rewrite Lemma \ref{lem:P_Flatness_with_C} as:
\begin{lem}
For $0\leq{i,j}\leq{n-1}$, we have
$\mathsf{D}_{L_j}P_{i,j}^{\KP^{n-1}}(z)=C^{\KP^{n-1}}_{\mathrm{Ion}(i)}z^{-1}P_{\mathrm{Ion}(i)-1,j}^{\KP^{n-1}}(z)$.
\end{lem}
It immediately follows that $P_{0,j}^{\KP^{n-1}}(z)$ satisfies the following differential equation:
\begin{equation*}
\frac{1}{C^{\KP^{n-1}}_1}\mathsf{D}_{L_j} \cdots \frac{1}{C^{\KP^{n-1}}_n}\mathsf{D}_{L_j}P_{0,j}^{\KP^{n-1}}(z)=z^{-n}P^{\KP^{n-1}}_{0,j}(z).
\end{equation*}
By the following commutation rule
\begin{equation}\label{eqn:Commutation_of_DL}
\mathsf{D}_{\KP^{n-1}}(e^{\frac{\widetilde{\mu}_{j}}{z}}F)=e^{\frac{\widetilde{\mu}_{j}}{z}}\mathsf{D}_{L_j}F,
\end{equation}
and by the definition of $\mathfrak{L}_i$, the differential equation above can be rewritten as 
\begin{equation}\label{eqn:P0jODE}
\mathfrak{L}_1\cdots\mathfrak{L}_n\left(e^{\frac{\widetilde{\mu}_{j}}{z}}P_{0,j}^{\KP^{n-1}}(z)\right)=z^{-n}e^{\frac{\widetilde{\mu}_{j}}{z}}P_{0,j}^{\KP^{n-1}}(z).
\end{equation}

\begin{lem}\label{lem:PF_Factorization}
\begin{equation*}
\mathfrak{L}_1\cdots\mathfrak{L}_n=(nL^{\KP^{n-1}})^{-n}\left(n^n\mathsf{D}_{\KP^{n-1}}^n-\frac{n\mathsf{D}_{\KP^{n-1}}L^{\KP^{n-1}}}{L^{\KP^{n-1}}}\sum_{k=0}^{n-1}(-1)^{n-k}{s}_{n,k}n^k\mathsf{D}_{\KP^{n-1}}^k\right).
\end{equation*}
\end{lem}
\begin{proof}
Firstly, observe that we have
\begin{equation*}
\mathfrak{L}_n\cdots\mathfrak{L}_1=\mathfrak{L}_1\cdots\mathfrak{L}_n
\end{equation*}
by the definition of $\mathfrak{L}_i$ and the part (3) of Lemma \ref{lem:properties_of_C_functions}.
By the re-interpretation (\ref{eqn:Birkhoff_Fact_with_L_i}) of Birkhoff factorization, we see that
\begin{equation*}
\mathfrak{L}_n\cdots\mathfrak{L}_1\mathsf{I}^{\KP^{n-1}}(q,z)=z^{-n}\mathsf{I}^{\KP^{n-1}}(q,z). 
\end{equation*}
Moreover, equation (\ref{eqn:PF_Eq_for_with_logq_factor}) gives us
\begin{equation*}
(nL^{\KP^{n-1}})^{-n}\left(n^n\mathsf{D}_{\KP^{n-1}}^n-\frac{n\mathsf{D}_{\KP^{n-1}}L^{\KP^{n-1}}}{L^{\KP^{n-1}}}\sum_{k=0}^{n-1}(-1)^{n-k}{s}_{n,k}n^k\mathsf{D}_{\KP^{n-1}}^k\right)\mathsf{I}^{\KP^{n-1}}(q,z)=z^{-n}\mathsf{I}^{\KP^{n-1}}(q,z).
\end{equation*}

Since both differential equations have the same phase space and their right-hand sides match, we conclude that their left-hand sides must also match. This completes the proof.
\end{proof}

An immediate consequence of Lemma \ref{lem:PF_Factorization} and equation (\ref{eqn:P0jODE}) is the following result.
\begin{cor}\label{cor:PF_for_P0jz}
The series $e^{\frac{\widetilde{\mu}_{j}}{z}}P_{0,j}^{\KP^{n-1}}(z)$ satisfies the Picard--Fuchs equation 
\begin{equation*}
(L^{\KP^{n-1}})^{-n}\left(\mathsf{D}_{\KP^{n-1}}^n-\frac{{\mathsf{D}}_{\KP^{n-1}}L_j^{\KP^{n-1}}}{n^{n-1}L_j^{\KP^{n-1}}}\sum_{k=0}^{n-1}(-1)^{n-k}{s}_{n,k}n^k\mathsf{D}_{\KP^{n-1}}^k\right)\left(e^{\frac{\widetilde{\mu}_{j}}{z}}P_{0,j}^{\KP^{n-1}}(z)\right)=z^{-n}e^{\frac{\widetilde{\mu}_{j}}{z}}P_{0,j}^{\KP^{n-1}}(z).
\end{equation*}
\end{cor}
In other words, $P_{0,j}^{\KP^{n-1}}(z)$ satisfies the conditions of Lemma \ref{lem:AymptoticPFSolution}. As a result, we obtain the following polynomiality statement.

\begin{cor}\label{cor:polynomiality_of_P_0j_KP}
 For any $k\geq 0$, we have $P_{0,j}^{k,\KP^{n-1}}\in\mathbb{C}[L^{\KP^{n-1}}]$ and they satisfy the following identity
\begin{equation}\label{eqn:P0jk_satisfying_mathdsLjk_equation}
\mathds{L}_{j,1}(P_{0,j}^{k,\KP^{n-1}})+\frac{1}{(L_j^{K\mathbb{P}^{n-1}})}\mathds{L}_{j,2}(P_{0,j}^{k-1,\KP^{n-1}})+\cdots+\frac{1}{(L_j^{K\mathbb{P}^{n-1}})^{n-1}}\mathds{L}_{j,n}(P_{0,j}^{k+1-n,\KP^{n-1}})=0
\end{equation}
where $\mathds{L}_{j,k}$ is defined by equation (\ref{eqn:mathdsLjk}).
\end{cor}

\subsection{Quantum Riemann--Roch operator for \texorpdfstring{$K\mathbb{P}^{n-1}$}{KP{n-1}}}\label{sec:QRRKP}
The degree $0$ (i.e. $q=0$) sector of the $\mathrm{T}$-equivariant Gromov--Witten theory of $\KP^{n-1}$, which is defined by virtual localization \cite{grpan}, is the Gromov--Witten theory of the $\mathrm{T}$-fixed locus $(\KP^{n-1})^{\mathrm{T}}$ twisted by the normal bundle $N_{(\KP^{n-1})^{\mathrm{T}}/\KP^{n-1}}$ and the inverse $\mathrm{T}$-equivariant Euler class $e_{\mathrm{T}}^{-1}(-)$. By quantum Riemann--Roch theorem \cite{cg}, the degree $0$ sector of the $\mathrm{T}$-equivariant Gromov--Witten theory of $\KP^{n-1}$ is related to the Gromov--Witten theory of $(\KP^{n-1})^{\mathrm{T}}$ by an operator 
$$\mathsf{Q}^{\KP^{n-1}}\in \text{End}(H^*_{\mathrm{T}}(\KP^{n-1}))[\![z]\!].$$ 
We need to calculate $\mathsf{Q}^{\KP^{n-1}}$ explicitly. 

The $\mathrm{T}$-fixed locus is a union of $n$ points,
$$(\KP^{n-1})^{\mathrm{T}}=(\mathbb{P}^{n-1})^{\mathrm{T}}=\{p_0,...,p_{n-1}\}.$$
At the fixed point $p_i$, we have $$N_{(\KP^{n-1})^{\mathrm{T}}/\KP^{n-1}}|_{p_i}=T_{p_i}\mathbb{P}^{n-1}\oplus \KP^{n-1}|_{p_i}.$$
The weights of $\mathrm{T}$ on the tangent space $T_{p_i}\mathbb{P}^{n-1}$ are
\begin{equation*}
    \chi_i-\chi_0,...,\widehat{\chi_i-\chi_i},..., \chi_i-\chi_{n-1}.
\end{equation*}
The weight of $\mathrm{T}$ on $\KP^{n-1}|_{p_i}$ is $-n\chi_i$.

It follows from the quantum Riemann--Roch theorem \cite{cg} that the restriction $\mathsf{Q}^{\KP^{n-1}}|_{p_i}$ to the fixed point $p_i$ is the multiplication by 
\begin{equation}\label{eqn:qrr1}
\exp\left(\sum_{m>0}N_{2m-1, i}\frac{(-1)^{2m-1}B_{2m}}{2m(2m-1)}z^{2m-1}\right).    
\end{equation}
Here 
\begin{equation}
N_{2m-1,i}=\frac{1}{(-n\chi_i)^{2m-1}}+\frac{1}{(\chi_i-\chi_0)^{2m-1}}+...+\widehat{\frac{1}{(\chi_i-\chi_i)^{2m-1}}}+...+\frac{1}{(\chi_i-\chi_{n-1})^{2m-1}}.
\end{equation}
In the specializations (\ref{eqn:specialization_KP}), we get
\begin{equation}
N_{2m-1,i}=\frac{1}{(\zeta^i)^{2m-1}}\left(\frac{1}{(-n)^{2m-1}}+\sum_{l=1}^{n-1}\frac{1}{(1-\zeta^l)^{2m-1}}\right)
\end{equation}
after rearranging terms. Note also that 
\begin{equation}
N_{2m-1,i}=\frac{N_{2m-1,0}}{\zeta^{i(2m-1)}}
\end{equation}
for all $m\geq{1}$.

 Let $p_i=[0:\cdots:0:1:0:\cdots:0]$ be the $i$-th fixed point of this action, then the restriction map $H_{\mathrm{T}}^*\left(\mathbb{P}^{n-1}\right)\rightarrow{H_{\mathrm{T}}^*\left(p_i\right)}$ sends $H$ to $\chi_i$ and the Gysin map ${H_{\mathrm{T}}^*\left(p_i\right)}\rightarrow{H_{\mathrm{T}}^*\left(\mathbb{P}^{n-1}\right)}$ sends $1$ to
\begin{equation*}
\Xi_i=\prod_{\substack{0\leq{j}\leq{n-1}\\j\neq{i}}}(H-\zeta^j).
\end{equation*}
These $\Xi_i$'s give another basis of $H_{\mathrm{T}}^*\left(\mathbb{P}^{n-1}\right)$ which we call the \textit{fixed point basis} of $H_{\mathrm{T}}^*\left(\mathbb{P}^{n-1}\right)$. Observe the following computation:
\begin{equation}\label{eqn:fixedpt_bss_as_a_sum}
\begin{split}
\Xi_i=\prod_{\substack{0\leq{j}\leq{n-1}\\j\neq{i}}}(H-\zeta^j)
&=\zeta^{i(n-1)}\prod_{\substack{0\leq{j}\leq{n-1}\\j\neq{i}}}\frac{(H-\zeta^j)}{\zeta^{i}}\\
&=\zeta^{-i}\prod_{0\leq{j}\leq{i-1}}\left(\frac{H}{\zeta^{i}}-\zeta^{j-i}\right)\prod_{i+1\leq{j}\leq{n-1}}\left(\frac{H}{\zeta^{i}}-\zeta^{j-i}\right)\\
&=\zeta^{-i}\prod_{n-i\leq{j}\leq{n-1}}\left(\frac{H}{\zeta^{i}}-\zeta^{j-n}\right)\prod_{1\leq{j}\leq{n-i-1}}\left(\frac{H}{\zeta^{i}}-\zeta^{j}\right)\\
&=\zeta^{-i}\prod_{0\leq{j}\leq{n-1}}\left(\frac{H}{\zeta^i}-\zeta^j\right)\\
&=\zeta^{-i}\sum_{j=0}^{n-1}\left(\frac{H}{\zeta^i}\right)^{j}\\
&=\zeta^{-i}\sum_{j=0}^{n-1}H^j\zeta^{-ji}.
\end{split}
\end{equation}

\begin{lem}\label{lem:K_over_L_q_0}
For all $0\leq{i}\leq{n-1}$, we have
\begin{equation*}
\frac{K^{\KP^{n-1}}_i}{(L^{\KP^{n-1}})^i}\Big\vert_{q=0}=1.
\end{equation*}
\end{lem}

\begin{proof}
Note that $L^{\KP^{n-1}}|_{q=0}=1$. By definition, $K_i^{\KP^{n-1}}=\prod_{j=0}^{i}C_j^{\KP^{n-1}}$. Therefore the Lemma follows from the statement $C_j^{\KP^{n-1}}|_{q=0}=1$. This is clearly true for $C_0^{\KP^{n-1}}=1$. For $C_j^{\KP^{n-1}}$ with $j>0$, this can be seem by induction on $j$. Assume that $C_k^{\KP^{n-1}}=1+O(q)$ as $q\to 0$ for $k<j$. By the definition of $\mathsf{I}_j^{\KP^{n-1}}$ in (\ref{eqn:formula_for_mathsfI}), we can see that\footnote{Recall that for $k>0, l\geq 0$, $q^k(\log q)^l\to 0$ as $q\to 0$.} $\mathsf{I}_j^{\KP^{n-1}}=\frac{(\log q)^j}{j!}+o(1)$ as $q\to 0$. The formula (\ref{eqn:mathrfrak_L_operator_dfn_of_C_i}) for $C_j^{\KP^{n-1}}$ then implies the Lemma.  
\end{proof}

Then, by Lemma \ref{lem:K_over_L_q_0} and the definition (\ref{def:normalized_idempotent}) we have for $0\leq{i}\leq n-1$
\begin{equation*}
\begin{split}
\widetilde{e}_i\big\vert_{q=0}
=-\sqrt{-1}\sum_{j=0}^{n-1}\zeta^{-i{j}}H^j 
=-\zeta^i\sqrt{-1}\Xi_i.
\end{split}
\end{equation*}
So, when restricted to $q=0$, the base change matrix from $\{\Xi_i\}$ basis to $\{e_i\}$ basis is given by the diagonal matrix
\begin{equation*}
\mathsf{B}\coloneqq -\sqrt{-1}\diag\left(1,\zeta,\ldots,\zeta^{n-1}\right).
\end{equation*}

\section{Ring of functions for \texorpdfstring{$\KP^{n-1}$}{KP{n-1}}}\label{sec:ringKP}

\subsection{Preparations}
We define the following series in $\mathbb{C}[\![q]\!]$ :
\begin{equation*}
X^{\KP^{n-1}}_{k,l}=\frac{\mathsf{D}_{\KP^{n-1}}^lC^{\KP^{n-1}}_k}{C^{\KP^{n-1}}_k}
\end{equation*}
for all $k,l\geq{0}$. We denote $X^{\KP^{n-1}}_{k,1}$ just by $X^{\KP^{n-1}}_k$. Also, we note that $X^{\KP^{n-1}}_0=0$ since $C^{\KP^{n-1}}_0=1$. A quick observation is
\begin{equation*}
X^{\KP^{n-1}}_{k,l}=\left(\mathsf{D}_{\KP^{n-1}}+X^{\KP^{n-1}}_k\right)X^{\KP^{n-1}}_{k,l-1}.
\end{equation*}
for all $k \geq 0$, and $l \geq 1$. This implies the following result.

\begin{lem}\label{lem:Xkl_generation}
We have
\begin{equation*}
X^{\KP^{n-1}}_{k,l}=\left(\mathsf{D}_{\KP^{n-1}}+X^{\KP^{n-1}}_k\right)^{l-1}X^{\KP^{n-1}}_k
\end{equation*}
for all $k\geq{0}$ and $l\geq{1}$. In particular, $X^{\KP^{n-1}}_{k,l}$ is a polynomial in $$\{X^{\KP^{n-1}}_k, \mathsf{D}_{\KP^{n-1}}X^{\KP^{n-1}}_k,\ldots,\mathsf{D}_{\KP^{n-1}}^{l-1}X^{\KP^{n-1}}_k\},$$ and $\mathsf{D}_{\KP^{n-1}}^{l-1}X^{\KP^{n-1}}_k$ is a polynomial in $$\{X^{\KP^{n-1}}_{k,1},\ldots,X^{\KP^{n-1}}_{k,l}\}.$$
\end{lem}

Furthermore, the series $X^{\KP^{n-1}}_i$, and $L^{\KP^{n-1}}$ satisfy the following properties.

\begin{lem}\label{lem:DLLemma}
We have
\begin{align}
\frac{\mathsf{D}_{\KP^{n-1}}L^{\KP^{n-1}}}{L^{\KP^{n-1}}}&=\frac{1}{n}\left(-1+(L^{\KP^{n-1}})^n\right),\label{eqn:DLLLemma1}\\
\frac{\mathsf{D}_{\KP^{n-1}}K^{\KP^{n-1}}_i}{K^{\KP^{n-1}}_i}&=\sum_{r=0}^iX^{\KP^{n-1}}_r,\label{eqn:DLLLemma2}\\
\frac{\mathsf{D}_{\KP^{n-1}}L^{\KP^{n-1}}}{L^{\KP^{n-1}}}&=\sum_{r=0}^nX^{\KP^{n-1}}_r,\label{eqn:DLLLemma3}
\end{align}
for $0\leq{i}\leq{n}$.
\end{lem}

\begin{proof}
The first equality (\ref{eqn:DLLLemma1}) follows from equation (\ref{eqn:DLL}):
\begin{equation*}
\mathsf{D}_{\KP^{n-1}}L^{\KP^{n-1}}
=\frac{1}{n}L^{\KP^{n-1}}\frac{(-n)^nq}{(1-(-n)^nq)}
=\frac{1}{n}L^{\KP^{n-1}}\frac{(-n)^nq-1+1}{(1-(-n)^nq)}
=\frac{1}{n}L^{\KP^{n-1}}\left(-1+(L^{\KP^{n-1}})^n\right).
\end{equation*}

Equation (\ref{eqn:DLLLemma2}, and equation (\ref{eqn:DLLLemma3}) directly follow from equation (\ref{eqn:K_i_definition}), and part $(1)$ of Lemma \ref{lem:properties_of_K_series}, respectively.
\end{proof}

For $0\leq{i,j}\leq{n-1}$ and $k\geq{0}$, let 
\begin{equation}\label{eqn:Flat_to_modflat_KP_2}
\widetilde{P}_{i,j}^{k,\KP^{n-1}}=\frac{(L^{\KP^{n-1}})^i}{K^{\KP^{n-1}}_i}P_{i,j}^{k,\KP^{n-1}}\zeta^{(k+i)j}.
\end{equation}

Then we obtain the following reformulation of Lemma \ref{lem:P_Flatness_with_C}; in other words, we rewrite the flatness equation in Lemma \ref{lem:P_Flatness_with_C} after the change (\ref{eqn:Flat_to_modflat_KP_2}).

\begin{lem}\label{lem:Modified_Flatness_before_Ai}
For $0\leq i \leq n-1$ and $k\geq 0$, we have
\begin{equation*}
\widetilde{P}_{\mathrm{Ion}(i)-1,j}^{k,\KP^{n-1}}=\widetilde{P}_{i,j}^{k,\KP^{n-1}}+\frac{1}{L^{\KP^{n-1}}}\left(\mathsf{D}_{\KP^{n-1}}+\sum_{r=0}^i X^{\KP^{n-1}}_r-i\frac{\mathsf{D}_{\KP^{n-1}}L^{\KP^{n-1}}}{L^{\KP^{n-1}}}\right)\widetilde{P}_{i,j}^{k-1,\KP^{n-1}}.
\end{equation*}
\end{lem}

\begin{proof}
The LHS of Lemma \ref{lem:P_Flatness_with_C} becomes
\begin{equation*}
\begin{aligned}
\mathsf{D}_{\KP^{n-1}}P_{i,j}^{k-1,\KP^{n-1}}
=&\left(\frac{\mathsf{D}_{\KP^{n-1}}K^{\KP^{n-1}}_i}{(L^{\KP^{n-1}})^i}-i\frac{K^{\KP^{n-1}}_i}{(L^{\KP^{n-1}})^i}\frac{\mathsf{D}_{\KP^{n-1}}L^{\KP^{n-1}}}{L^{\KP^{n-1}}}\right)\widetilde{P}_{i,j}^{k-1,\KP^{n-1}}\zeta^{-(k-1+i)j}\\
&+\frac{K^{\KP^{n-1}}_i}{(L^{\KP^{n-1}})^i}\mathsf{D}_{\KP^{n-1}}\widetilde{P}_{i,j}^{k-1,\KP^{n-1}}\zeta^{-(k-1+i)j}
\end{aligned}
\end{equation*}
and RHS of Lemma \ref{lem:P_Flatness_with_C} becomes
\begin{align*}
&C^{\KP^{n-1}}_{\mathrm{Ion}(i)}\frac{K^{\KP^{n-1}}_{{\mathrm{Ion}(i)}-1}}{(L^{\KP^{n-1}})^{{\mathrm{Ion}(i)}-1}}\widetilde{P}_{{\mathrm{Ion}(i)}-1,j}^{k,\KP^{n-1}}\zeta^{-(k-1+{\mathrm{Ion}(i)})j}-\frac{K^{\KP^{n-1}}_i}{(L^{\KP^{n-1}})^{i-1}}\widetilde{P}_{i,j}^{k,\KP^{n-1}}\zeta^{-(k-1+i)j}\\
=&\frac{K^{\KP^{n-1}}_{{\mathrm{Ion}(i)}}}{(L^{\KP^{n-1}})^{{\mathrm{Ion}(i)}-1}}\widetilde{P}_{{\mathrm{Ion}(i)}-1,j}^{k,\KP^{n-1}}\zeta^{-(k-1+{\mathrm{Ion}(i)})j}-\frac{K^{\KP^{n-1}}_i}{(L^{\KP^{n-1}})^{i-1}}\widetilde{P}_{i,j}^{k,\KP^{n-1}}\zeta^{-(k-1+i)j}\\
=&\frac{K^{\KP^{n-1}}_i}{(L^{\KP^{n-1}})^{i-1}}\widetilde{P}_{{\mathrm{Ion}(i)}-1,j}^{k,\KP^{n-1}}\zeta^{-(k-1+i)j}-\frac{K^{\KP^{n-1}}_i}{(L^{\KP^{n-1}})^{i-1}}\widetilde{P}_{i,j}^{k,\KP^{n-1}}\zeta^{-(k-1+i)j}.
\end{align*}

Putting these together, using the definition of $K^{\KP^{n-1}}_i$, Lemma \ref{lem:properties_of_K_series}, and cancelling out some common factors we obtain
\begin{equation*}
\left(\frac{\mathsf{D}_{\KP^{n-1}}K^{\KP^{n-1}}_i}{K^{\KP^{n-1}}_i}-i\frac{\mathsf{D}_{\KP^{n-1}}L^{\KP^{n-1}}}{L^{\KP^{n-1}}}+\mathsf{D}_{\KP^{n-1}}\right)\widetilde{P}_{i,j}^{k-1,\KP^{n-1}}=
\widetilde{P}_{{\mathrm{Ion}(i)}-1,j}^{k,\KP^{n-1}}L^{\KP^{n-1}}-\widetilde{P}_{i,j}^{k,\KP^{n-1}}L^{\KP^{n-1}}.
\end{equation*}

The rest follows from (\ref{eqn:DLLLemma2}) of Lemma \ref{lem:DLLemma}.
\end{proof}

For any $m\geq{1}$, define the following series in $x$:
\begin{equation*}
Z^{\KP^{n-1}}_{m,k}=
\begin{cases}
\mathsf{D}_{\KP^{n-1}}^{-1}C^{\KP^{n-1}}_{k+1}...\mathsf{D}_{\KP^{n-1}}^{-1}C^{\KP^{n-1}}_{m}\quad&\text{if}\quad 0\leq k \leq m-1,\\
1\quad&\text{if}\quad k=m,\\
0\quad&\text{if}\quad k>m.
\end{cases}
\end{equation*}
From the definition of $Z^{\KP^{n-1}}_{m,k}$, we easily see that
\begin{equation}\label{eqn:Zmkrelation}
\mathsf{D}_{\KP^{n-1}}Z^{\KP^{n-1}}_{m,k}=C^{\KP^{n-1}}_{k+1}Z^{\KP^{n-1}}_{m,k+1}
\end{equation}
for all $k\geq{0}$. We also recall that, by equation (\ref{eqn:mathrfrak_L_operator_dfn_of_C_i}) , for $m\geq{1}$, $$\mathsf{I}^{\KP^{n-1}}_m=\mathsf{D}_{\KP^{n-1}}^{-1}C^{\KP^{n-1}}_{1}...\mathsf{D}_{\KP^{n-1}}^{-1}C^{\KP^{n-1}}_{m}$$ which is just $Z^{\KP^{n-1}}_{m,0}$. Now for $k\geq 1$ define the following series in $q$:

\begin{equation*}
B^{\KP^{n-1}}_{k,p}=
\begin{cases}
\mathsf{D}_{\KP^{n-1}}^{k-1}C^{\KP^{n-1}}_1\quad&\text{if}\quad p=1,\\
\sum\limits_{k_2=p-1}^{k_1-1}...\sum\limits_{k_p=1}^{k_{p-1}-1}\left(\prod\limits_{i=1}^{p-1}\binom{k_{i}-1}{k_{i+1}}\mathsf{D}_{\KP^{n-1}}^{k_i-1-k_{i+1}}C^{\KP^{n-1}}_i\right)\mathsf{D}_{\KP^{n-1}}^{k_p-1}C^{\KP^{n-1}}_p \quad&\text{if}\quad 2\leq{p}\leq{k},\\
0\quad&\text{if}\quad p>k
\end{cases}
\end{equation*}
where $k_1=k$.

\begin{lem}\label{lem:DkImCommutator}
For all $k,m\geq{1}$, we have
\begin{equation*}
\mathsf{D}_{\KP^{n-1}}^k\mathsf{I}^{\KP^{n-1}}_m=\sum\limits_{p=1}^kB^{\KP^{n-1}}_{k,p}Z^{\KP^{n-1}}_{m,p}.
\end{equation*}
\end{lem}

\begin{proof}
Inductively, we show that multiplication by $A\in\mathbb{C}[\![q]\!]$ followed by the operator $\mathsf{D}_{\KP^{n-1}}^i$ is given by
\begin{equation}\label{eq:HigherDiA}
\mathsf{D}_{\KP^{n-1}}^iA=\sum_{j=0}^i\binom{i}{j}(\mathsf{D}_{\KP^{n-1}}^jA)\mathsf{D}_{\KP^{n-1}}^{i-j}.
\end{equation}

Using the fact that for $m\geq{1}$, $\mathsf{I}^{\KP^{n-1}}_m=\mathsf{D}_{\KP^{n-1}}^{-1}C^{\KP^{n-1}}_{1}...\mathsf{D}_{\KP^{n-1}}^{-1}C^{\KP^{n-1}}_{m}=Z^{\KP^{n-1}}_{m,0}$ together with equations (\ref{eqn:Zmkrelation}) and (\ref{eq:HigherDiA}), we inductively complete the proof.
\end{proof}

\begin{lem}
For all $1\leq{m}\leq{n-1}$, we have
\begin{equation}\label{PFforGraded}
\begin{aligned}
B^{\KP^{n-1}}_{n,m}
&=\frac{\mathsf{D}_{\KP^{n-1}}L^{\KP^{n-1}}}{n^{n-1}L^{\KP^{n-1}}}\sum\limits_{k=1 }^{n-1}(-1)^{n-k}{s}_{n,k}n^kB^{\KP^{n-1}}_{k,m}\\
&=\frac{\mathsf{D}_{\KP^{n-1}}L^{\KP^{n-1}}}{n^{n-1}L^{\KP^{n-1}}}\sum\limits_{k=m }^{n-1}(-1)^{n-k}{s}_{n,k}n^kB^{\KP^{n-1}}_{k,m}.
\end{aligned}
\end{equation}
\end{lem}

\begin{proof}
The second equality follows from the fact that $B^{\KP^{n-1}}_{k,m}=0$ for $m>k$. For the first equality, we use induction on $m$. For $m=1$, it follows from $B^{\KP^{n-2}}_{k,1}=\mathsf{D}_{\KP^{n-1}}^{k-1}C^{\KP^{n-1}}_1=\mathsf{D}_{\KP^{n-1}}^k\mathsf{I}^{\KP^{n-1}}_1$ and equation (\ref{eqn:PF_for_Iks_annihilation}). The following completes the inductive step: 
\begin{align*}
0=&\mathsf{D}_{\KP^{n-1}}^n\mathsf{I}_k^{\KP^{n-1}}-\frac{\mathsf{D}_{\KP^{n-1}}L^{\KP^{n-1}}}{n^{n-1}L^{\KP^{n-1}}}\sum_{k=0}^{n-1}(-1)^{n-k}{s}_{n,k}n^k\mathsf{D}_{\KP^{n-1}}^k\mathsf{I}_k^{\KP^{n-1}}
\quad\quad\text{by equation (\ref{eqn:PF_for_Iks_annihilation})}\\
=&\sum\limits_{p=1}^nB^{\KP^{n-1}}_{n,p}Z^{\KP^{n-1}}_{m,p}-\frac{\mathsf{D}_{\KP^{n-1}}L^{\KP^{n-1}}}{n^{n-1}L^{\KP^{n-1}}}\sum\limits_{k=1}^{n-1}(-1)^{n-k}s_{n,k}n^k\sum\limits_{p=1}^kB^{\KP^{n-1}}_{k,p}Z^{\KP^{n-1}}_{m,p}\quad\text{by Lemma \ref{lem:DkImCommutator}}.
\end{align*}

Since $B^{\KP^{n-1}}_{k,p}=0$ for $p>k$, and $Z^{\KP^{n-1}}_{m,p}=0$ for $p>m$, we have
\begin{align*}
0
=&\sum\limits_{p=1}^mB^{\KP^{n-1}}_{n,p}Z^{\KP^{n-1}}_{m,p}-\frac{\mathsf{D}_{\KP^{n-1}}L^{\KP^{n-1}}}{n^{n-1}L^{\KP^{n-1}}}\sum\limits_{k=1}^{n-1}(-1)^{n-k}s_{n,k}n^k\sum\limits_{p=1}^mB^{\KP^{n-1}}_{k,p}Z^{\KP^{n-1}}_{m,p}\\
=&\sum\limits_{p=1}^m\left(B^{\KP^{n-1}}_{n,p}-\frac{\mathsf{D}_{\KP^{n-1}}L^{\KP^{n-1}}}{n^{n-1}L^{\KP^{n-1}}}\sum\limits_{k=1}^{n-1} (-1)^{n-k}s_{n,k}n^kB^{\KP^{n-1}}_{k,p}\right)Z^{\KP^{n-1}}_{m,p}\\
=&B^{\KP^{n-1}}_{n,m}-\frac{\mathsf{D}_{\KP^{n-1}}L^{\KP^{n-1}}}{n^{n-1}L^{\KP^{n-1}}}\sum\limits_{k=1}^{n-1}(-1)^{n-k}s_{n,k}n^kB^{\KP^{n-1}}_{k,m}\\
&+\sum\limits_{p=1}^{m-1}\underbrace{\left(B^{\KP^{n-1}}_{n,p}-\frac{\mathsf{D}_{\KP^{n-1}}L^{\KP^{n-1}}}{n^{n-1}L^{\KP^{n-1}}}\sum\limits_{k=1}^{n-1}(-1)^{n-k}s_{n,k}n^kB^{\KP^{n-1}}_{k,p}\right)}_{=0\text{ by inductive hypothesis.}}Z^{\KP^{n-1}}_{m,p}.
\end{align*}
\end{proof}

\subsection{Descriptions of the rings}
Set 
\begin{equation*}
\mathbb{C}[(L^{\KP^{n-1}})^{\pm 1}][\mathcal{D}_{\KP^{n-1}}\mathcal{X}]:=\mathbb{C}[(L^{\KP^{n-1}})^{\pm 1}][\{\mathsf{D}_{\KP^{n-1}}^jX^{\KP^{n-1}}_i\}_{1\leq i\leq n-1, j\geq 0}],
\end{equation*}
and
\begin{align*}
\mathfrak{X}
:=&\{\mathsf{D}_{\KP^{n-1}}^jX^{\KP^{n-1}}_i\}_{1\leq i\leq n-2, 0\leq j\leq n-2-i}\\
=&\{X^{\KP^{n-1}}_1,...,\mathsf{D}_{\KP^{n-1}}^{n-3}X_1\}\cup \ldots \cup\{X^{\KP^{n-1}}_i,...,\mathsf{D}_{\KP^{n-1}}^{n-2-i}X_i\}\cup\ldots\cup\{X^{\KP^{n-1}}_{n-2}\}.
\end{align*}

\begin{lem}\label{lem:D_Graded_Ring_X}
 $\mathbb{C}[(L^{\KP^{n-1}})^{\pm 1}][\mathcal{D}_{\KP^{n-1}}\mathcal{X}]$ is a quotient of the ring $\mathbb{C}[(L^{\KP^{n-1}})^{\pm 1}][\mathfrak{X}]$.
\end{lem}

\begin{proof}
Now, for any $1\leq{p}\leq{k-1}$, define
\begin{align*}
\mathcal{Z}_{p,k}=&\{X^{\KP^{n-1}}_{1,1},...,X^{\KP^{n-1}}_{1,k-p},...,X^{\KP^{n-1}}_{p,1},...,X^{\KP^{n-1}}_{p,k-p}\},\\
\widetilde{\mathcal{Z}}_{p,k}=&\{X^{\KP^{n-1}}_1,...,\mathsf{D}_{\KP^{n-1}}^{k-p-1}X^{\KP^{n-1}}_1,...,X^{\KP^{n-1}}_p,...,\mathsf{D}_{\KP^{n-1}}^{k-p-1}X^{\KP^{n-1}}_p\},\\
\mathcal{S}_{p,k}=&\mathcal{Z}_{p,k}\setminus\{X^{\KP^{n-1}}_{p,k-p}\},\\
\widetilde{\mathcal{S}}_{p,k}=&\widetilde{\mathcal{Z}}_{p,k}\setminus\{\mathsf{D}^{k-p-1}X^{\KP^{n-1}}_p\}.
\end{align*}
For each of these sets, and for a fixed $p$ we have
\begin{equation}\label{eqn:set_inclusions_for_ZpkSpk}
\mathcal{S}_{p,k}\subseteq\mathcal{Z}_{p,k}\subseteq\mathcal{S}_{p,k+1}\subseteq\mathcal{Z}_{p,k+1},\quad\text{and}\quad\widetilde{\mathcal{S}}_{p,k}\subseteq\widetilde{\mathcal{Z}}_{p,k}\subseteq\widetilde{\mathcal{S}}_{p,k+1}\subseteq\widetilde{\mathcal{Z}}_{p,k+1}.
\end{equation}

Note that for any $1\leq{p}\leq{k-1}$, directly by the definitions we have
\begin{equation}
B^{\KP^{n-1}}_{n,m}
=\frac{\mathsf{D}_{\KP^{n-1}}L^{\KP^{n-1}}}{n^{n-1}L^{\KP^{n-1}}}\sum\limits_{k=m }^{n-1}(-1)^{n-k}{s}_{n,k}n^kB^{\KP^{n-1}}_{k,m}.
\end{equation}
\begin{equation}\label{eqn:Zpkp}
\frac{B^{\KP^{n-1}}_{k,p}}{K_p}=X^{\KP^{n-1}}_{p,k-p}+\widetilde{B}^{\KP^{n-1}}_{k,p}
\end{equation}
where $\widetilde{B}^{\KP^{n-1}}_{k,p}$ is a polynomial in elements of $\mathcal{S}_{p,k}$. Then, dividing both sides of equation (\ref{PFforGraded}) by $K^{\KP^{n-1}}_m$ for any $1\leq{m}\leq{n-1}$, we obtain
\begin{equation}\label{eqn:Xmn_minus_m}
\begin{split}
\frac{B^{\KP^{n-1}}_{n,m}}{K^{\KP^{n-1}}_m}
=&\frac{\mathsf{D}_{\KP^{n-1}}L^{\KP^{n-1}}}{n^{n-1}L^{\KP^{n-1}}}\sum\limits_{k=m }^{n-1}(-1)^{n-k}{s}_{n,k}n^kB^{\KP^{n-1}}_{k,m}\\
=&X^{\KP^{n-1}}_{m,n-m}+\widetilde{B}^{\KP^{n-1}}_{n,m}+\frac{\mathsf{D}_{\KP^{n-1}}L^{\KP^{n-1}}}{n^{n-1}L^{\KP^{n-1}}}\underbrace{\sum\limits_{k=m }^{n-1}(-1)^{n-k}{s}_{n,k}n^k\frac{B^{\KP^{n-1}}_{k,m}}{K^{\KP^{n-1}}_m}}_{(\star)}.
\end{split}
\end{equation}
We see that $(\star)$ is a polynomial in elements of $\mathcal{Z}_{m,n-1}$ by the inclusions (\ref{eqn:set_inclusions_for_ZpkSpk}) and equation (\ref{eqn:Zpkp}). We already know $\widetilde{B}^{\KP^{n-1}}_{n,m}$ is a polynomial in element of $\mathcal{S}_{m,n}$ and $\mathcal{Z}_{m,n-1}\subseteq\mathcal{S}_{m,n}$; hence, it follows that $X^{\KP^{n-1}}_{m,n-m}$ is a polynomial in elements of $\mathcal{S}_{m,n}\cup\{{(L^{\KP^{n-1}})^{\pm1}}\}$ by equation (\ref{eqn:Xmn_minus_m}) and equation (\ref{eqn:DLLLemma1}). This implies that $\mathsf{D}_{\KP^{n-1}}^{n-m-1}X^{\KP^{n-1}}_m$ is a polynomial in elements of $\widetilde{\mathcal{S}}_{m,n}\cup\{(L^{\KP^{n-1}})^{\pm1}\}$ by Lemma \ref{lem:Xkl_generation}. This completes the proof.
\end{proof}

Now, we define the series $A^{\KP^{n-1}}_i$ for $0\leq i\leq n$ by
\begin{equation*}
A^{\KP^{n-1}}_i=\frac{1}{L^{\KP^{n-1}}}\left(i\frac{\mathsf{D}_{\KP^{n-1}}L^{\KP^{n-1}}}{L^{\KP^{n-1}}}-\sum_{r=0}^{i} X^{\KP^{n-1}}_{r}\right).
\end{equation*}

Then, the flatness equation in Lemma \ref{lem:Modified_Flatness_before_Ai} becomes
\begin{equation}\label{eqn:modflateqn_for_KP}
\widetilde{P}_{\mathrm{Ion}(i)-1,j}^{k,\KP^{n-1}}=\widetilde{P}_{i,j}^{k,\KP^{n-1}}+\frac{1}{L^{\KP^{n-1}}}\mathsf{D}_{\KP^{n-1}}\widetilde{P}_{i,j}^{k-1,\KP^{n-1}}+A^{\KP^{n-1}}_{n-i}\widetilde{P}_{i,j}^{k-1,\KP^{n-1}}.
\end{equation}
We call (\ref{eqn:modflateqn_for_KP}) the \textbf{\textit{modified flatness equations}} for $\KP^{n-1}$.

Set 
\begin{equation*}
\mathbb{C}[(L^{\KP^{n-1}})^{\pm 1}][\mathcal{D}_{\KP^{n-1}}\mathcal{A}]:=\mathbb{C}[L^{\pm 1}][\{\mathsf{D}_{\KP^{n-1}}^jA^{\KP^{n-1}}_i\}_{1\leq i\leq n-1, j\geq 0}],
\end{equation*}
and
\begin{align*}
\mathfrak{A}^{\KP^{n-1}}
:=&\{\mathsf{D}_{\KP^{n-1}}^jA^{\KP^{n-1}}_i\}_{1\leq i\leq n-2, 0\leq j\leq n-2-i}\\
=&\{A^{\KP^{n-1}}_1,...,\mathsf{D}_{\KP^{n-1}}^{n-3}A^{\KP^{n-1}}_1\}\cup,\ldots\cup\{A^{\KP^{n-1}}_i,...,\mathsf{D}_{\KP^{n-1}}^{n-2-i}A^{\KP^{n-1}}_i\}\cup\ldots\cup\{A^{\KP^{n-1}}_{n-2}\}.
\end{align*}
The following is immediate from Lemma \ref{lem:D_Graded_Ring_X}.
\begin{cor}\label{cor:DGradedRingA}The ring
$\mathbb{C}[(L^{\KP^{n-1}})^{\pm1}][\mathcal{D}_{\KP^{n-1}}\mathcal{A}]$ is a quotient of the ring $\mathbb{C}[(L^{\KP^{n-1}})^{\pm1}][\mathfrak{A}^{\KP^{n-1}}]$.
\end{cor}

In what follows we further simplify the ring $\mathbb{C}[(L^{\KP^{n-1}})^{\pm1}][\mathfrak{A}^{\KP^{n-1}}]$.

\begin{lem}\label{lem:Properties_of_Ai_s}
For the series $A^{\KP^{n-1}}_i$, we have the following
\begin{enumerate}
    \item $A^{\KP^{n-1}}_i=-A^{\KP^{n-1}}_{n-i}$ for all $0\leq{i}\leq{n}$,
    \item $A^{\KP^{n-1}}_0=A^{\KP^{n-1}}_n=0$, and $A^{\KP^{n-1}}_{\frac{n}{2}}=0$ if $n$ is even,
    \item $\sum_{i=0}^{n}A^{\KP^{n-1}}_i=0$.
\end{enumerate}
\end{lem}

\begin{proof}
By Lemma \ref{lem:properties_of_C_functions}, we have $C^{\KP^{n-1}}_i=C^{\KP^{n-1}}_{n+1-i}$ for all $1\leq{i}\leq{n}$. Hence,  $X^{\KP^{n-1}}_i=X^{\KP^{n-1}}_{n+1-i}$ for all $1\leq{i}\leq{n}$. This gives the following reformulation of equation (\ref{eqn:DLLLemma3}) :
\begin{equation*}
\sum_{r=0}^{i} X^{\KP^{n-1}}_{r}-i\frac{\mathsf{D}_{\KP^{n-1}} L^{\KP^{n-1}}}{L^{\KP^{n-1}}}=(n-i) \frac{\mathsf{D}_{\KP^{n-1}} L^{\KP^{n-1}}}{L^{\KP^{n-1}}}-\left(\sum_{r=0}^{n-i} X^{\KP^{n-1}}_{r}\right)\quad\text{for all}\quad 0\leq{i}\leq{n}.
\end{equation*}
This proves the first part of the lemma. The other two parts follow immediately.
\end{proof}

Now we analyze (\ref{eqn:modflateqn_for_KP}). Let $k=0$. Then $\widetilde{P}_{\mathrm{Ion}(i)-1,j}^{0,\KP^{n-1}}=\widetilde{P}_{i,j}^{0,\KP^{n-1}}$ for all $0\leq{i}\leq{n-1}$. This means $\widetilde{P}_{i,j}^{0,\KP^{n-1}}=\widetilde{P}_{0,j}^{0,\KP^{n-1}}$ for all $0\leq{i}\leq{n-1}$. Now, let $k=1$. Then, we have
\begin{equation}\label{eqn:abc_sum}
\underbrace{\sum_{i=0}^{n-1}\widetilde{P}_{\mathrm{Ion}(i)-1,j}^{1,\KP^{n-1}}}_{(a)}=\underbrace{\sum_{i=0}^{n-1}\widetilde{P}_{i,j}^{1,\KP^{n-1}}}_{(b)}+\frac{1}{L^{\KP^{n-1}}}\mathsf{D}_{\KP^{n-1}}\sum_{i=0}^{n-1}\widetilde{P}_{i,j}^{0,\KP^{n-1}}+\underbrace{\sum_{i=0}^{n-1}A^{\KP^{n-1}}_{n-i}\widetilde{P}_{0,j}^{0,\KP^{n-1}}}_{(c)}.
\end{equation}
The sums $(a)$ and $(b)$ are clearly the same. The sum $(c)$ is zero by Lemma \ref{lem:Properties_of_Ai_s}. Since we have $\widetilde{P}_{i,j}^{0,\KP^{n-1}}=\widetilde{P}_{0,j}^{0,\KP^{n-1}}$, the equation (\ref{eqn:abc_sum}) becomes 
\begin{equation*}
\frac{n}{L^{\KP^{n-1}}}\mathsf{D}_{\KP^{n-1}}\widetilde{P}_{0,j}^{0,\KP^{n-1}}=0.
\end{equation*}
So, $\widetilde{P}_{i,j}^{0,\KP^{n-1}}=\widetilde{P}_{0,j}^{0,\KP^{n-1}}$ is a constant, and its value depends on the initial conditions of (\ref{eqn:modflateqn_for_KP}). Now, consider the equation (\ref{eqn:modflateqn_for_KP}), and add these equations side by side for $i=0,n-1,\ldots,n-i+1$. Then, setting $k=1$ yields
\begin{equation}\label{eqn:P1_sum_KP}
\widetilde{P}_{n-i,j}^{1,\KP^{n-1}}=\widetilde{P}_{0, j}^{1,\KP^{n-1}}+\sum_{r=0}^{i-1} A^{\KP^{n-1}}_{r} \widetilde{P}_{0, j}^{0,\KP^{n-1}}\quad\text{for}\quad 1\leq{i}\leq{n}.
\end{equation}
Now, let $k=2$ in equation (\ref{eqn:modflateqn_for_KP}), and substitute the above equation (\ref{eqn:P1_sum_KP}) into (\ref{eqn:modflateqn_for_KP}). This gives us 
\begin{equation*}
\begin{split}
\widetilde{P}_{\mathrm{Ion}(i)-1,j}^{2,\KP^{n-1}}
=&\widetilde{P}_{i,j}^{2,\KP^{n-1}}+\frac{1}{L^{\KP^{n-1}}}\mathsf{D}_{\KP^{n-1}}\widetilde{P}_{0 ,j}^{1,\KP^{n-1}}+\frac{1}{L^{\KP^{n-1}}}\sum_{r=0}^{n-i-1} \left(\mathsf{D}_{\KP^{n-1}}A^{\KP^{n-1}}_{r}\right) \widetilde{P}_{0, j}^{0,\KP^{n-1}}\\
&+A^{\KP^{n-1}}_{n-i}\widetilde{P}_{0, j}^{1,\KP^{n-1}}+\sum_{r=0}^{n-i-1}A^{\KP^{n-1}}_{n-i}A^{\KP^{n-1}}_{r} \widetilde{P}_{0, j}^{0,\KP^{n-1}}.
\end{split}
\end{equation*}
Summing this equality over $0\leq{i}\leq{n-1}$, cancelling out $\sum_{i=0}^{n-1}\widetilde{P}_{\mathrm{Ion}(i)-1,j}^{2,\KP^{n-1}}=\sum_{i=0}^{n-1}\widetilde{P}_{i,j}^{2,\KP^{n-1}}$, and noting that $\sum_{i=0}^{n-1}A^{\KP^{n-1}}_{n-i}\widetilde{P}_{0,j}^{1,\KP^{n-1}}=0$, we obtain
\begin{equation}\label{eq:nLDP1}
\frac{n}{L^{\KP^{n-1}}}\mathsf{D}_{\KP^{n-1}}\widetilde{P}_{0,j}^{1,\KP^{n-1}}+\frac{1}{L^{\KP^{n-1}}}\sum_{i=0}^{n-1}\sum_{r=0}^{n-i-1}\left(\mathsf{D}_{\KP^{n-1}}A^{\KP^{n-1}}_r\right)\widetilde{P}_{0,j}^{0,\KP^{n-1}}+\sum_{i=0}^{n-1}\sum_{r=0}^{n-i-1}A^{\KP^{n-1}}_{n-i}A^{\KP^{n-1}}_r\widetilde{P}_{0,j}^{0,\KP^{n-1}}=0.
\end{equation}

Setting $k=1$ in Corollary \ref{cor:polynomiality_of_P_0j_KP}, we obtain the following 
\begin{equation*}
\mathds{L}_{j,1}(P^{1,\KP^{n-1}}_{0,j})+\frac{1}{L^{\KP^{n-1}}_j}\mathds{L}_{j,2}(P^{0,\KP^{n-1}}_{0,j})=0
\end{equation*}
which reads as\footnote{The power series $X^{\KP^{n-1}}\in\mathbb{C}[\![q]\!]$ is defined in Appendix \ref{Appendix:Analysis_of_I_function}. It is $X^{\KP^{n-1}}=(L^{\KP^{n-1}})^n$.}
\begin{align*}
n\mathsf{D}_{\KP^{n-1}}\widetilde{P}^{1,\KP^{n-1}}_{0,j}
&=\frac{1}{L^{\KP^{n-1}}}\frac{1}{n^2}\binom{n+1}{4}(1-X^{\KP^{n-1}})X^{\KP^{n-1}}P^{0,\KP^{n-1}}_{0,j}\\
&=\frac{1}{n^2}\binom{n+1}{4}(1-(L^{\KP^{n-1}})^n)(L^{\KP^{n-1}})^{n-1}\widetilde{P}^{0,\KP^{n-1}}_{0,j}.
\end{align*}
Define $f_n(L^{\KP^{n-1}})\in\mathbb{C}[(L^{\KP^{n-1}})^{\pm{1}}]$ to be the right hand side of above equation without $\widetilde{P}^{0,\KP^{n-1}}_{0,j}$:
\begin{equation*}
f_n(L^{\KP^{n-1}})=\frac{1}{n^2}\binom{n+1}{4}(1-(L^{\KP^{n-1}})^n)(L^{\KP^{n-1}})^{n-1}.
\end{equation*}

\begin{lem}
For any $n\geq{3}$, we have
\begin{align*}
\sum_{i=0}^{n-1}\sum_{r=0}^{n-i-1}\mathsf{D}_{\KP^{n-1}}A^{\KP^{n-1}}_r&=\sum_{r=1}^{\lfloor\frac{n-1}{2}\rfloor}(n-2r)\mathsf{D}_{\KP^{n-1}}A^{\KP^{n-1}}_r,\\
\sum_{i=0}^{n-1}\sum_{r=0}^{n-i-1}A^{\KP^{n-1}}_{n-i}A^{\KP^{n-1}}_r&=-\sum_{r=1}^{\lfloor\frac{n-1}{2}\rfloor}\left(A^{\KP^{n-1}}_r\right)^2.
\end{align*}
\end{lem}

\begin{proof}
After cancellations due to Lemma \ref{lem:Properties_of_Ai_s}, we obtain the identities.
\end{proof}

\begin{lem}\label{lem:Equations_forDAl}
For any $n\geq{3}$, we have
\begin{equation*}
f_n(L^{\KP^{n-1}})+\sum_{r=1}^{\lfloor\frac{n-1}{2}\rfloor}(n-2r)\left(\mathsf{D}_{\KP^{n-1}}A^{\KP^{n-1}}_r\right)-{L^{\KP^{n-1}}}\sum_{r=1}^{\lfloor\frac{n-1}{2}\rfloor}\left(A^{\KP^{n-1}}_r\right)^2=0.
\end{equation*}

Equivalently, dividing into even and odd cases, we have
\begin{equation*}
\begin{split}
2\mathsf{D}_{\KP^{n-1}}A^{\KP^{n-1}}_{s-1}
&=\sum_{r=1}^{s-1}L^{\KP^{n-1}}\left(A^{\KP^{n-1}}_r\right)^2-\sum_{r=1}^{s-2}(n-2r)\mathsf{D}_{\KP^{n-1}}A^{\KP^{n-1}}_r-f_{2s}(L^{\KP^{n-1}})\quad\text{if}\quad n=2s\geq{4},\\
\mathsf{D}_{\KP^{n-1}}A^{\KP^{n-1}}_s
&=\sum_{r=1}^{s}L^{\KP^{n-1}}\left(A^{\KP^{n-1}}_r\right)^2-\sum_{r=1}^{s-1}(n-2r)\mathsf{D}_{\KP^{n-1}}A^{\KP^{n-1}}_r-f_{2s+1}(L^{\KP^{n-1}})\quad\text{if}\quad n=2s+1\geq{3}.
\end{split}
\end{equation*}
\end{lem}

Lemma \ref{lem:Equations_forDAl} generalizes \cite[Equation (7)]{lho} and \cite[Equation (32)]{lho-p}.
Let $n\geq 3$ be an odd number with $n=2s+1$, define
\begin{equation*}
\mathfrak{S}^{\KP^{n-1}}_{\text{odd}}=\{A^{\KP^{n-1}}_1,\ldots,\mathsf{D}_{\KP^{n-1}}^{n-3}A^{\KP^{n-1}}_1\}\cup\cdots\cup\{A^{\KP^{n-1}}_{s-1},\ldots,\mathsf{D}_{\KP^{n-1}}^{n-s+1}A^{\KP^{n-1}}_{s-1}\}\cup\{A^{\KP^{n-1}}_s\}.
\end{equation*}
Similarly, let $n\geq 4$ be an even number with $n=2s$, define
\begin{equation*}
\mathfrak{S}^{\KP^{n-1}}_{\text{even}}=\{A^{\KP^{n-1}}_1,\ldots,\mathsf{D}_{\KP^{n-1}}^{n-3}A_1\}\cup\cdots\cup\{A^{\KP^{n-1}}_{s-2},\ldots,\mathsf{D}_{\KP^{n-1}}^{n-s}A^{\KP^{n-1}}_{s-2}\}\cup\{A^{\KP^{n-1}}_{s-1}\}.
\end{equation*}
In either case, we denote both $\mathfrak{S}^{\KP^{n-1}}_{\text{odd}}$, and $\mathfrak{S}^{\KP^{n-1}}_{\text{even}}$ as $\mathfrak{S}^{\KP^{n-1}}_n$.

\begin{prop}\label{pro:CDA_Simplification}
$\mathbb{C}[(L^{\KP^{n-1}})^{\pm 1}][\mathcal{D}_{\KP^{n-1}}\mathcal{A}]$ is a quotient of the ring $\mathbb{C}[(L^{\KP^{n-1}})^{\pm 1}][\mathfrak{S}^{\KP^{n-1}}_n]$.
\end{prop}

\subsection{More on flatness equation} 
In this part, we will describe how each $\widetilde{P}_{i,j}^{k,\KP^{n-1}}$ lifts canonially to the free algebra $\mathbb{C}[(L^{\KP^{n-1}})^{\pm 1}][\mathfrak{S}^{\KP^{n-1}}_n]$.

We see that $\widetilde{P}_{i,j}^{k,\KP^{n-1}}\in\mathbb{C}[(L^{\KP^{n-1}})^{\pm 1}][\mathcal{D}_{\KP^{n-1}}\mathcal{A}]$ by the modified flatness equations (\ref{eqn:modflateqn_for_KP}), Lemma \ref{lem:Xkl_generation}, and Corollary \ref{cor:polynomiality_of_P_0j_KP}. Then, we obtain a canonical lift of each $\widetilde{P}_{i,j}^{k,\KP^{n-1}}$ to the free algebra $\mathbb{C}[(L^{\KP^{n-1}})^{\pm 1}][\mathfrak{S}^{\KP^{n-1}}_n]$ through  Lemmas \ref{lem:Properties_of_Ai_s},  \ref{lem:Equations_forDAl}, and the modified flatness equations (\ref{eqn:modflateqn_for_KP}) by the following procedure:
\begin{equation}\label{eq:ModifiedFlatnessLift}
\begin{split}
\widetilde{P}_{n-1,j}^{k,\KP^{n-1}}&=\widetilde{P}_{0,j}^{k,\KP^{n-1}}+\frac{1}{L^{\KP^{n-1}}}\mathsf{D}_{\KP^{n-1}}\widetilde{P}_{0,j}^{k-1,\KP^{n-1}}\in\mathbb{C}[(L^{\KP^{n-1}})^{\pm 1}]\\
\widetilde{P}_{n-2,j}^{k,\KP^{n-1}}&=\widetilde{P}_{n-1,j}^{k,\KP^{n-1}}+\frac{1}{L^{\KP^{n-1}}}\mathsf{D}_{\KP^{n-1}}\widetilde{P}_{n-1,j}^{k-1,\KP^{n-1}}+A^{\KP^{n-1}}_1\widetilde{P}_{n-1,j}^{k-1,\KP^{n-1}}\in\mathbb{C}[(L^{\KP^{n-1}})^{\pm 1}][A^{\KP^{n-1}}_1]\\
\vdots \,\,\,\,\,\, &=\,\,\,\,\,\,\,\,\,\,\,\,\,\,\,\,\,\,\,\,\,\,\,\,\,\,\,\,\,\,\vdots
\end{split}
\end{equation}
If we describe it in words, we start with $\widetilde{P}_{0,j}^{k,\KP^{n-1}}\in \mathbb{C}[(L^{\KP^{n-1}})^{\pm 1}]$ and use equation (\ref{eqn:modflateqn_for_KP}) for $i=n, n-1,..., 2$ to inductively lift $\widetilde{P}_{i,j}^{k,\KP^{n-1}}$ to $\mathbb{C}[(L^{\KP^{n-1}})^{\pm 1}][\mathfrak{S}^{\KP^{n-1}}_n]$ for $i=n-1, n-2,...,1$ in this descending order. In this lifting procedure, we eliminate the unnecessary $A^{\KP^{n-1}}_i$'s using Lemmas \ref{lem:Properties_of_Ai_s} and \ref{lem:Equations_forDAl}. We also see that the orders of derivatives are bounded for the lifts since the process is a finite step procedure, and the initial step starts with $\widetilde{P}_{0,j}^{k,\KP^{n-1}}\in \mathbb{C}[L^{\KP^{n-1}}]$. Moreover, the bounds of these derivatives do not exceed the bounds imposed by Lemma \ref{lem:D_Graded_Ring_X}.

\section{Comparison of cohomological field theories}\label{sec:CohFT_iden}
We identify\footnote{The specializations (\ref{eqn:specialization}) and (\ref{eqn:specialization_KP}) are imposed.} $H_{\mathrm{T}}^*(\KP^{n-1})$ and $H^*_{\mathrm{T,Orb}}(\CnZn)$ via the following grading-preserving map:
\begin{equation}\label{eqn:state_space_isom}
H^*(\KP^{n-1})\to H^*_{\mathrm{T,Orb}}(\CnZn), \quad H^i\mapsto \phi_i, \,\,\, 0\leq i\leq n-1.  \end{equation}

By (\ref{eqn:metric_CnZn}) and (\ref{eqn:metric_KP}), via (\ref{eqn:state_space_isom}), we have the following identification of metrics
\begin{equation}\label{eqn:metric_iden}
g^{\KP^{n-1}}\mapsto -g^{\CnZn}.  
\end{equation}

\subsection{Identifications}\label{sec:iden}
\subsubsection{Change of variables} \label{subsubsec:Change_of_variables}

Here, we spell out the details of change of variables. Consider the following identification
\begin{equation*}
n^n(L^{\KP^{n-1}})^n\mapsto{(-1)^{n+1}}(L^{\CnZn})^{n}
\end{equation*}
as an equality and observe the following computation:
\begin{equation}\label{eqn:L_to_the_n_equal}
\begin{aligned}
n^n(1-(-n)^nq)^{-1}
&=(-1)^{n+1}x^n\left(1-(-1)^n\left(\frac{x}{n}\right)^n\right)^{-1}\\
&=(-1)^{n+1}\left(x^{-n}-(-1)^n\frac{1}{n^n}\right)^{-1}\\
&=(-1)^{n+1}n^n\left(n^nx^{-n}-(-1)^n\right)^{-1}\\
&=n^n\left((-1)^{n+1}n^nx^{-n}-(-1)^{n+1}(-1)^n\right)^{-1}\\
&=n^n\left(1-(-n)^nx^{-n}\right)^{-1}.
\end{aligned}
\end{equation}
This implies that we have
\begin{equation}\label{eqn:variable_change}
    q=x^{-n}.
\end{equation}

Conversely, equation (\ref{eqn:variable_change}) implies
\begin{equation*}
n^n(L^{\KP^{n-1}})^n={(-1)^{n+1}}(L^{\CnZn})^{n}.
\end{equation*}
So, we see that $L^{\KP^{n-1}}$ and $L^{\CnZn}$ are identified via
\begin{equation}\label{eqn:identify_Ls}
nL^{\KP^{n-1}}=-{\rho}L^{\CnZn}
\end{equation}
where $\rho$ is an $n^{\text{th}}$ root of $-1$, i.e. $\rho^n=-1$. More precisely, (\ref{eqn:identify_Ls}) requires an analytic continuation of $L^{\CnZn}$ from $x=0$ to $x=\infty$ within a sector of the $x$-plane. The analytically continued $L^{\CnZn}$ is then compared with $L^{\KP^{n-1}}$ using (\ref{eqn:variable_change}). The value of $\rho$ is decided so that (\ref{eqn:identify_Ls}) holds.

By (\ref{eqn:variable_change}), we have
\begin{equation}\label{eqn:D_ident}
\mathsf{D}_{\KP^{n-1}}=q\frac{d}{dq}=-\frac{1}{n}x\frac{d}{dx}=-\frac{1}{n}\mathsf{D}_{\CnZn}.  \end{equation}

In addition, for $1\leq{i}\leq n-1$, we formally identify the following\footnote{These identifications are consistent with the definitions of these power series.}:
\begin{equation}\label{eqn:generators_iden}
\begin{split}
C_i^{\KP^{n-1}}&\mapsto -\frac{\rho}{n}C_i^{\CnZn},\\
X_i^{\KP^{n-1}}&\mapsto -\frac{1}{n}X_i^{\CnZn},\\   
A_i^{\KP^{n-1}}&\mapsto \frac{1}{\rho}A_i^{\CnZn}. \end{split}   
\end{equation}

Adjoining\footnote{We should note that $C^{\KP^{n-1}}_i$'s are related to each other via Lemma \ref{lem:properties_of_C_functions}. Hence $\mathfrak{C}^{\KP^{n-1}}_n$ can be taken as the set $\{C^{\KP^{n-1}}_1,\ldots,C^{\KP^{n-1}}_{\lfloor{\frac{n+1}{2}}\rfloor}\}$ as in \cite{gt} for the case $\CnZn$.} $\mathfrak{C}_n\coloneqq \{C_1^{\KP^{n-1}},\ldots,C_{n-1}^{\KP^{n-1}}\}$ to free polynomial ring appearing in Proposition \ref{pro:CDA_Simplification}, we define
\begin{equation*}
\mathds{F}_{\KP^{n-1}}\coloneqq \mathbb{C}[(L^{\KP^{n-1}})^{\pm{1}}][\mathfrak{S}_n^{\KP^{n-1}}][\mathfrak{C}_n^{\KP^{n-1}}].
\end{equation*}
In \cite[Proposition 2.11, and Corollary 3.4]{gt}, a similar ring is constructed and in this paper we denote it as
\begin{equation*}
\mathds{F}_{\CnZn}\coloneqq \mathbb{C}[(L^{\CnZn})^{\pm{1}}][\mathfrak{S}_n^{\CnZn}][\mathfrak{C}_n^{\CnZn}].
\end{equation*}
We write $$\Upsilon: \mathds{F}_{\KP^{n-1}} \to \mathds{F}_{\CnZn} $$ for the ring map generated by the above identifications (\ref{eqn:identify_Ls}) and (\ref{eqn:generators_iden}).

\subsubsection{Picard--Fuchs equations}\label{sec:PF_iden}
Here, we discuss how the identification (\ref{eqn:identify_Ls}) affects the Picard--Fuchs equations of $\KP^{n-1}$. In equation (\ref{eqn:PF3}), we showed that the function $\mathsf{I}^{\KP^{n-1}}(q,z)$ satisfies the following Picard--Fuchs equation
\begin{equation}\label{eqn:PF_Comparison_KP}
\left(z^n\mathsf{D}_{\KP^{n-1}}^n-1\right)\mathsf{I}
=(-1)^nqz^n\prod_{i=0}^{n-1}\left(n\mathsf{D}_{\KP^{n-1}}+i\right)\mathsf{I}
\end{equation}

It is proved in \cite[Proposition 1.3]{gt} that the $I$-function $I^{\CnZn}(x,z)$ of $[\mathbb{C}^n/\mathbb{Z}_n]$ satisfies the following Picard--Fuchs equation
\begin{equation*}
\frac{1}{x^n}\prod_{i=0}^{n-1}\left(\mathsf{D}_{\CnZn}-i\right)I-(-1)^n\left(\frac{1}{n}\right)^n\mathsf{D}_{\CnZn}^nI=\left(\frac{1}{z}\right)^nI
\end{equation*}
which turns into
\begin{equation*}
(-1)^nqz^n\prod_{i=0}^{n-1}\left(n\mathsf{D}_{\KP^{n-1}}+i\right)I-z^n\mathsf{D}_{\KP^{n-1}}^nI=I
\end{equation*}
via the change of variable $q=x^{-n}$. We can further re-organize this equation and obtain
\begin{equation*}
-z^n\mathsf{D}_{\KP^{n-1}}^nI-I=-(-1)^nqz^n\prod_{i=0}^{n-1}\left(n\mathsf{D}_{\KP^{n-1}}+i\right)I\,.
\end{equation*}
Replacing $z$ with ${\rho}z$ and comparing it to (\ref{eqn:PF_Comparison_KP}), we see that Picard--Fuchs equations of $\KP^{n-1}$ and $\CnZn$ match. So, we obtained the following result.

\begin{prop}\label{prop:PF_Match}
Picard--Fuchs equations satisfied by $\mathsf{I}^{\KP^{n-1}}(q,z)$ and $I^{\CnZn}(x,z)$ match after change of variables $q\mapsto{x^{-n}}$ and $z\mapsto{\rho{z}}$.
\end{prop}

\subsubsection{Modified flatness equations}\label{sec:flatness_iden} 
Recall the modified flatness equations (\ref{eqn:modflateqn_for_KP}):
\begin{equation*}
\widetilde{P}_{\mathrm{Ion}(i)-1,j}^{k,\KP^{n-1}}=\widetilde{P}_{i,j}^{k,\KP^{n-1}}+\frac{1}{L^{\KP^{n-1}}}\mathsf{D}_{\KP^{n-1}}\widetilde{P}_{i,j}^{k-1,\KP^{n-1}}+A^{\KP^{n-1}}_{n-i}\widetilde{P}_{i,j}^{k-1,\KP^{n-1}}.
\end{equation*}

Now, we analyze the effect of identifications on these equations:
\begin{equation*}
\begin{aligned}
\widetilde{P}_{\mathrm{Ion}(i)-1,j}^{k,\KP^{n-1}}
&=\widetilde{P}_{i,j}^{k,\KP^{n-1}}+\left(\frac{n}{-{\rho}L^{\CnZn}}\right)\left({-\frac{1}{n}\mathsf{D}_{\CnZn}}\widetilde{P}_{i,j}^{k-1,\KP^{n-1}}\right)+\frac{1}{\rho}A^{\CnZn}_{n-i}\widetilde{P}_{i,j}^{k-1,\KP^{n-1}}\\
&=\widetilde{P}_{i,j}^{k,\KP^{n-1}}+\frac{1}{{\rho}L^{\CnZn}}{\mathsf{D}_{\CnZn}}\widetilde{P}_{i,j}^{k-1,\KP^{n-1}}+\frac{1}{\rho}A^{\CnZn}_{n-i}\widetilde{P}_{i,j}^{k-1,\KP^{n-1}}.
\end{aligned}
\end{equation*}
Now, define
\begin{equation}\label{eqn:2nd_change_in_Pijk_KP}
\overline{P}_{i,j}^{k,\KP^{n-1}}\coloneqq \widetilde{P}_{i,j}^{k,\KP^{n-1}}\rho^k.
\end{equation}
Then, we obtain
\begin{equation*}
\overline{P}_{\mathrm{Ion}(i)-1,j}^{k,\KP^{n-1}}\rho^{-k}=\overline{P}_{i,j}^{k,\KP^{n-1}}\rho^{-k}+\frac{1}{{\rho}L^{\CnZn}}{\mathsf{D}_{\CnZn}}\overline{P}_{i,j}^{k-1,\KP^{n-1}}\rho^{-k+1}+\frac{1}{\rho}A^{\CnZn}_{n-i}\overline{P}_{i,j}^{k-1,\KP^{n-1}}\rho^{-k+1}.
\end{equation*}
Cancelling out the term $\rho^{-k}$, we obtain
\begin{equation*}
\overline{P}_{\mathrm{Ion}(i)-1,j}^{k,\KP^{n-1}}=\overline{P}_{i,j}^{k,\KP^{n-1}}+\frac{1}{L^{\CnZn}}{\mathsf{D}_{\CnZn}}\overline{P}_{i,j}^{k-1,\KP^{n-1}}+A^{\CnZn}_{n-i}\overline{P}_{i,j}^{k-1,\KP^{n-1}}
\end{equation*}
which are the modified flatness equations of $\CnZn$ \cite[Equation 2.10]{gt}. The change of variables (\ref{eqn:2nd_change_in_Pijk_KP}), is equivalent to replacing $z$ with ${\rho}z$. This is consistent with the above-proposed method to match Picard--Fuchs equations for $\KP^{n-1}$ and $\CnZn$. So, we established the following result.

\begin{prop}\label{prop:Modified_Flatness_Match}
The modified flatness equations (\ref{eqn:modflateqn_for_KP}) for $\KP^{n-1}$ match with the modified flatness equations \cite[Equation 2.10]{gt} of $\CnZn$ after the identifications in Section \ref{subsubsec:Change_of_variables} and the change of variables (\ref{eqn:2nd_change_in_Pijk_KP}).
\end{prop}



\subsubsection{Genus \texorpdfstring{$0$}{0} invariants}\label{sec:genus0_iden}
Recall (\ref{eqn:3ptGWKP}):
\begin{equation*}
\left\langle\left\langle H^i,H^j,H^k\right\rangle\right\rangle_{0,3}^{\KP^{n-1}}=\sum_{d=0}^{\infty}Q^d\langle H^i, H^j, H^k\rangle_{0,3,d}^{\KP^{n-1}}
=-\frac{1}{n}\frac{K^{\KP^{n-1}}_{i+j}}{K^{\KP^{n-1}}_{i}K^{\KP^{n-1}}_{j}}\delta_{\text{Inv}(i+j \text{ mod }n), k}.
\end{equation*}
Also recall (\ref{eqn:3pt_inv_CnZn}):
\begin{equation*}
\left\langle\left\langle\phi_i,\phi_j,\phi_k\right\rangle\right\rangle_{0,3}^{\CnZn}=\frac{K^{\CnZn}_{i+j}}{K^{\CnZn}_iK^{\CnZn}_j}\frac{1}{n}\delta_{\text{Inv}(i+j \text{ mod }n), k}. \end{equation*}

The identification (\ref{eqn:generators_iden}) yields a matching of generating functions of genus $0$, $3$-point invariants after a factor\footnote{This factor of $(-1)$ will be evident in Theorem \ref{thm:Main_Theorem}.} of $(-1)$.

\subsection{\texorpdfstring{$\mathsf{R}$}{R}-matrices}\label{sec:R_iden}
The $R$-matrices of $\CnZn$ and $\KP^{n-1}$ satisfy the flatness equation
\begin{equation}\label{eqn:flatness_wo_subscript}
D\left(\Psi^{-1}R_{k-1}\right)+\left(\Psi^{-1}R_k\right)DU-\Psi^{-1}\left(DU\right)\Psi\left(\Psi^{-1}R_k\right)=0.   
\end{equation}

Define $\widetilde{\mathsf{R}}^{\CnZn}(z)$ and $\widetilde{\mathsf{R}}^{\KP^{n-1}}(z)$ to be the solutions of equation (\ref{eqn:flatness_wo_subscript}) with the initial conditions
\begin{equation*}
\widetilde{\mathsf{R}}^{\CnZn}(z)\big\vert_{x=0}=\widetilde{\mathsf{R}}^{\KP^{n-1}}(z)\big\vert_{q=0}=\mathsf{Id}.
\end{equation*}
Let ${\mathsf{R}}^{\CnZn}(z)$ be the true $R$-matrix of $\CnZn$, and set $$\mathsf{P}^{\CnZn}(z) \coloneqq \Psi^{-1}_{\CnZn}{\mathsf{R}}^{\CnZn}(z)$$
and let $$\mathsf{P}_{i,j}^{\CnZn}(z)=\sum_{k\geq{0}}P_{i,j}^{k,\CnZn}z^k$$
be its entries.

\begin{lem}\label{lem:True_R_Matrix_CnZn}
The true $R$-matrix ${\mathsf{R}}^{\CnZn}(z)$ of $\CnZn$ satisfies
\begin{equation*}
{\mathsf{R}}^{\CnZn}(z)\big\vert_{x=0}=\Psi_{\CnZn}\big\vert_{x=0}\mathsf{Q}^{\CnZn}(z)\Psi^{-1}_{\CnZn}\big\vert_{x=0}
\end{equation*}
with 
\begin{equation*}
\mathsf{Q}^{\CnZn}(z)=\diag(\mathsf{Q}_0^{\CnZn}(z),\ldots,\mathsf{Q}_{n-1}^{\CnZn}(z)) 
\end{equation*}
where
\begin{equation*}
\mathsf{Q}_i^{\CnZn}(z)=\exp\left(n\sum_{l>0}(-1)^{l}\frac{B_{nl+1}\left(\frac{i}{n}\right)}{nl+1}\frac{z^{nl}}{nl} \right).
\end{equation*}
\end{lem}

\begin{proof}
The true $R$-matrix ${\mathsf{R}}^{\CnZn}(z)$ is in normalized idempotent basis, and the quantum Riemann--Roch operator found by equation (\ref{eqn:oqrr_final}) is in flat basis $\{\phi_0,\ldots,\phi_{n-1}\}$. After a base change, they agree when $x=0$ due to the orbifold quantum Riemann--Roch theorem.
\end{proof}

\begin{lem}\label{lem:True_R_Matrix_KP}
The true $R$-matrix ${\mathsf{R}}^{\KP^{n-1}}(z)$ of $\KP^{n-1}$ is given by
\begin{equation*}
{\mathsf{R}}^{\KP^{n-1}}(z)=\widetilde{\mathsf{R}}^{\KP^{n-1}}(z)\mathsf{Q}^{\KP^{n-1}}(z)
\end{equation*}
with 
\begin{equation*}
\mathsf{Q}^{\KP^{n-1}}(z)=\diag(\mathsf{Q}_0^{\KP^{n-1}}(z),\ldots,\mathsf{Q}_{n-1}^{\KP^{n-1}}(z)) 
\end{equation*}
where
\begin{equation*}
\mathsf{Q}_i^{\KP^{n-1}}(z)=\exp\left(\sum_{m>0}N_{2m-1, i}\frac{(-1)^{2m-1}B_{2m}}{2m(2m-1)}z^{2m-1}\right).    
\end{equation*}
\end{lem}

\begin{proof}
By quantum Riemann--Roch and the base change matrix $\mathsf{B}$, we have
\begin{equation*}
{\mathsf{R}}^{\KP^{n-1}}(z)\big\vert_{q=0}=\mathsf{B}\mathsf{Q}^{\KP^{n-1}}(z)\mathsf{B}^{-1}=\mathsf{Q}^{\KP^{n-1}}(z).
\end{equation*}
Also, observe that the matrix series
\begin{equation*}
  \widetilde{\mathsf{R}}^{\KP^{n-1}}(z)\mathsf{Q}^{\KP^{n-1}}(z)  
\end{equation*}
is a solution of flatness equation (\ref{eqn:flatness_wo_subscript}) since $\mathsf{Q}^{\KP^{n-1}}(z) $ is diagonal matrix and commutes with $\mathsf{D}_{\KP^{n-1}}U$.
\end{proof}

Recall, in Section \ref{sec:FrobKP}, we defined the following $$\mathsf{P}^{\KP^{n-1}}(z)=\Psi^{-1}_{\KP^{n-1}}{\mathsf{R}}^{\KP^{n-1}}(z)$$
and $$\mathsf{P}_{i,j}^{\KP^{n-1}}(z)=\sum_{k\geq{0}}P_{i,j}^{k,\KP^{n-1}}z^k$$
for its entries.  

The polynomiality of $P_{0,j}^{k,\CnZn}$ is proved in \cite{gt}, and the polynomiality of $P_{0,j}^{k,\KP^{n-1}}$ is given by Corollary \ref{cor:polynomiality_of_P_0j_KP}.

\begin{lem}\label{lem:mathcing_of_P0jzs}
The series $-\sqrt{-1}\mathsf{P}^{\CnZn}_{0,j}(z)$ and $\mathsf{P}^{\KP^{n-1}}_{0,j}({\rho}z)$ match after identification (\ref{eqn:identify_Ls}).
\end{lem}

\begin{cor}
The matrix series $-\sqrt{-1}\mathsf{P}^{\CnZn}(z)$ and $\mathsf{P}^{\KP^{n-1}}({\rho}z)$ match after identifications in Section \ref{subsubsec:Change_of_variables}. 
\end{cor} 

\begin{proof}
The proof relies on matching of lifting procedures of $-\sqrt{-1}\mathsf{P}^{\CnZn}(z)$, and $\mathsf{P}^{\KP^{n-1}}({\rho}z)$ after identifications in Section \ref{subsubsec:Change_of_variables}. Firstly, we already showed that the modified flatness equations match via these identifications in Proposition \ref{prop:Modified_Flatness_Match}. The other steps we use in the lifting procedure are  Lemma \ref{lem:Properties_of_Ai_s}, and Lemma  \ref{lem:Equations_forDAl} which also match with \cite[Lemma 2.8]{gt}, and \cite[Lemma 2.10]{gt} respectively. Hence, lifting procedures completely match via the identifications.
\end{proof}

In the rest of this subsection, we describe how to prove Lemma \ref{lem:mathcing_of_P0jzs}. In the Appendix \ref{Appendix:Analysis_of_I_function}, we have shown that under the change of variable (\ref{eqn:variable_change}) we have
\begin{equation*}
\mathds{L}_{j,k}=\frac{(-1)^k}{n^k}\mathds{L}_{j,k}^{\CnZn}.
\end{equation*}
Then, the equation (\ref{eqn:P0jk_satisfying_mathdsLjk_equation}) reads as
\begin{equation*}
\begin{aligned}
\frac{(-1)}{n}\mathds{L}^{\CnZn}_{j,1}(P_{0,j}^{k,\KP^{n-1}})
+\frac{n}{(-\rho L_j^{\CnZn})}&\frac{(-1)^2}{n^2}\mathds{L}^{\CnZn}_{j,2}(P_{0,j}^{k-1,\KP^{n-1}})
+\ldots \\ &\ldots
+\frac{n^{n-1}}{(-\rho L_j^{\CnZn})^{n-1}}\frac{(-1)^n}{n^n}\mathds{L}^{\CnZn}_{j,n}(P_{0,j}^{k+1-n,\KP^{n-1}})=0
\end{aligned}
\end{equation*}
which can be rewritten as
\begin{equation}\label{eqn:P0jk_satisfy_mathdsL_for_CnZn_after_identification}
\begin{aligned}
\mathds{L}^{\CnZn}_{j,1}(P_{0,j}^{k,\KP^{n-1}}\rho^k)
+\frac{n}{( L_j^{\CnZn})}&\mathds{L}^{\CnZn}_{j,2}(P_{0,j}^{k-1,\KP^{n-1}}\rho^{k-1})
+\ldots \\ &\ldots
+\frac{1}{( L_j^{\CnZn})^{n-1}}\mathds{L}^{\CnZn}_{j,n}(P_{0,j}^{k+1-n,\KP^{n-1}}\rho^{k+1-n})=0    
\end{aligned}
\end{equation}
after multiplying both sides with $-n\rho^k$. In \cite[Corollary 1.16]{gt}, we showed that $P_{0,j}^{k,\CnZn}$ satisfies the same equation:
\begin{equation}\label{eqn:CnZn_satisfy_mathdsL}
\begin{aligned}
\mathds{L}^{\CnZn}_{j,1}(P_{0,j}^{k,\CnZn})
+\frac{n}{( L_j^{\CnZn})}&\mathds{L}^{\CnZn}_{j,2}(P_{0,j}^{k-1,\CnZn})
+\ldots\\ &\ldots
+\frac{1}{( L_j^{\CnZn})^{n-1}}\mathds{L}^{\CnZn}_{j,n}(P_{0,j}^{k+1-n,\CnZn})=0.
\end{aligned}
\end{equation}

Since we have
\begin{equation*}
\mathsf{D}_{\KP^{n-1}}=(\mathsf{D}_{\KP^{n-1}}L^{\KP^{n-1}})\frac{d}{dL^{\KP^{n-1}}} \quad \text{and} \quad \mathsf{D}_{\CnZn}=(\mathsf{D}_{\CnZn}L^{\CnZn})\frac{d}{dL^{\CnZn}}
\end{equation*}
the operators $\mathds{L}_{j,k}$ and $\mathds{L}_{j,k}^{\CnZn}$ can be written purely in terms of in $L^{\KP^{n-1}}$ and $L^{\CnZn}$, respectively. Note also that we have
\begin{equation*}
\mathds{L}_{j,1}=n\mathsf{D}_{\KP^{n-1}} \quad \text{and} \quad
\mathds{L}^{\CnZn}_{j,1}=n\mathsf{D}_{\CnZn}.
\end{equation*}
This means if we know the constant terms of $\mathsf{P}^{\CnZn}_{0,j}(z)$ with respect to $L^{\CnZn}$ and $\mathsf{P}^{\KP^{n-1}}_{0,j}(z)$  with respect to $L^{\KP^{n-1}}$ then we can determine them by equation (\ref{eqn:P0jk_satisfying_mathdsLjk_equation}) and equation (\ref{eqn:CnZn_satisfy_mathdsL}).

Since the identification (\ref{eqn:identify_Ls}) turns equation (\ref{eqn:P0jk_satisfying_mathdsLjk_equation}) into equation (\ref{eqn:P0jk_satisfy_mathdsL_for_CnZn_after_identification}), we see that in order to prove Lemma \ref{lem:mathcing_of_P0jzs}, we need to show that the constant terms of the series $-\sqrt{-1}\mathsf{P}^{\CnZn}_{0,j}(z)$ with respect to $L^{\CnZn}$ and the series $\mathsf{P}^{\KP^{n-1}}_{0,j}({\rho}z)$ with respect to $L^{\KP^{n-1}}$ are the same.

Note that the constant term of $\mathsf{P}^{\CnZn}_{0,j}(z)$ with respect to $L^{\CnZn}$ is the same as its constant term with respect to $x$ since $L^{\CnZn}\big\vert_{x=0}=0$. Then, we need to find $(0,j)$-entry of
\begin{equation*}
\begin{split}
\mathsf{P}^{\CnZn}(z)\big\vert_{x=0}
&=\Psi^{-1}_{\CnZn}{\mathsf{R}}^{\CnZn}(z)\big\vert_{x=0}\\
&=\left(\Psi^{-1}_{\CnZn}\left(\Psi_{\CnZn}\big\vert_{x=0}\mathsf{Q}^{\CnZn}(z)\Psi^{-1}_{\CnZn}\big\vert_{x=0}\right)\right)\big\vert_{x=0}\\
&=\mathsf{Q}^{\CnZn}(z)\Psi^{-1}_{\CnZn}\big\vert_{x=0}.
\end{split}
\end{equation*}

In \cite{gt}, it is found that
\begin{equation*}
\left[\Psi^{-1}_{\CnZn}\right]_{j,\beta}=\zeta^{-\beta{j}}\frac{K^{\CnZn}_{j}}{(L^{\CnZn})^{j}}\quad\text{where}\quad 0\leq\beta,j\leq n-1.
\end{equation*}
So, the entries of the first row of $\Psi^{-1}_{\CnZn}$ are all $1$'s since $K_0^{\CnZn}=1$. Then, we have
\begin{equation*}
\begin{split}
\mathsf{P}^{\CnZn}_{0,j}(z)\big\vert_{L^{\CnZn}=0}
&=\mathsf{P}^{\CnZn}_{0,j}(z)\big\vert_{x=0}\quad \text{since} \quad L^{\CnZn}\big\vert_{x=0}=0\\
&=\mathsf{Q}_0^{\CnZn}(z)\,.    
\end{split}
\end{equation*}

Now, we focus on the other side of the medallion and find the constant term of $\mathsf{P}^{\KP^{n-1}}_{0,j}({\rho}z)$ with respect to $L^{\KP^{n-1}}$. Then, we need to find the $(0,j)$-entry of 
\begin{equation*}
\mathsf{P}^{\KP^{n-1}}({\rho}z)\big\vert_{L^{\KP^{n-1}}=0}=\Psi^{-1}_{\KP^{n-1}}\big\vert_{L^{\KP^{n-1}}=0}\widetilde{\mathsf{R}}^{\KP^{n-1}}({\rho}z)\big\vert_{L^{\KP^{n-1}}=0}\mathsf{Q}^{\KP^{n-1}}({\rho}z).
\end{equation*}

Note that
\begin{equation*}
\begin{split}
\Psi^{-1}_{\KP^{n-1}}\big\vert_{L^{\KP^{n-1}}=0}\widetilde{\mathsf{R}}^{\KP^{n-1}}({\rho}z)\big\vert_{L^{\KP^{n-1}}=0}
&=\left(\Psi^{-1}_{\KP^{n-1}}\widetilde{\mathsf{R}}^{\KP^{n-1}}({\rho}z)\right)\big\vert_{L^{\KP^{n-1}}=0}\\
&=\left(\Psi^{-1}_{\KP^{n-1}}\widetilde{\mathsf{R}}^{\KP^{n-1}}({\rho}z)\right)\big\vert_{q=\infty}
\end{split}
\end{equation*}
where $q=\infty$ means the limit of the analytic continuation\footnote{This arises from the analytic continuation involved in (\ref{eqn:identify_Ls}).} of $\Psi^{-1}_{\KP^{n-1}}\widetilde{\mathsf{R}}^{\KP^{n-1}}({\rho}z)$ as $q$ goes to $\infty$. Let $(0,j)$ entry of $\left(\Psi^{-1}_{\KP^{n-1}}\widetilde{\mathsf{R}}^{\KP^{n-1}}(z)\right)\big\vert_{q=\infty}$ be given by
\begin{equation}\label{eqn:const_at_infty}
\sum_{k\geq{0}}a_{0,j}^kz^k.
\end{equation}
Then, the equality we wanted to prove, 
\begin{equation*}
-\sqrt{-1}\mathsf{P}_{0,j}^{\CnZn}(z)\big\vert_{L^{\CnZn}=0}=\mathsf{P}_{0,j}^{\KP^{n-1}}({\rho}z)\big\vert_{L^{\KP^{n-1}}=0},    
\end{equation*}
reads as
\begin{equation*}
-\sqrt{-1}\mathsf{Q}_0^{\CnZn}(z)=\mathsf{Q}_j^{\KP^{n-1}}({\rho}z)\sum_{k\geq{0}}a_{0,j}^k({\rho}z)^k
\end{equation*}
which is
\begin{equation*}
-\sqrt{-1}\exp\left(n\sum_{l>0}(-1)^{l}\frac{B_{nl+1}\left(0\right)}{nl+1}\frac{z^{nl}}{nl} \right)
=\left(\sum_{k\geq{0}}a_{0,j}^k({\rho}z)^k\right)\exp\left(\sum_{m>0}N_{2m-1, j}\frac{(-1)^{2m-1}B_{2m}}{2m(2m-1)}({\rho}z)^{2m-1}\right).
\end{equation*}
Replacing $z$ with $\rho^{-1}z$ on both sides and noting that $N_{2m-1, j}=N_{2m-1, 0}\zeta^{-j(2m-1)}$, and $\rho^{-nl}=(-1)^l$ we get
\begin{equation*}
-\sqrt{-1}\exp\left(n\sum_{l>0}\frac{B_{nl+1}\left(0\right)}{nl+1}\frac{z^{nl}}{nl} \right)
=\left(\sum_{k\geq{0}}a_{0,j}^kz^k\right)\exp\left(\sum_{m>0}N_{2m-1, 0}\frac{(-1)^{2m-1}B_{2m}}{2m(2m-1)}\left(\frac{z}{\zeta^j}\right)^{2m-1}\right).
\end{equation*}

\begin{lem}
We have $a_{0,j}^k=a_{0,0}^k\zeta^{-jk}$ for all $k\geq{0}$ and $0\leq{j}\leq{n-1}$.
\end{lem}

\begin{proof}
Consider the matrix series 
\begin{equation*}
\widetilde{\mathsf{P}}^{\KP^{n-1}}(z)=\sum_{k\geq{0}}\widetilde{P}_{i,j}^{k,\KP^{n-1}}z^k
\end{equation*}
where $\widetilde{P}_{i,j}^{k,\KP^{n-1}}$ is defined via equation (\ref{eqn:Flat_to_modflat_KP_2}):
\begin{equation*}
\widetilde{P}_{i,j}^{k,\KP^{n-1}}=\frac{(L^{\KP^{n-1}})^i}{K^{\KP^{n-1}}_i}P_{i,j}^{k,\KP^{n-1}}\zeta^{(k+i)j}.
\end{equation*}
Then the initial conditions $\widetilde{P}_{i,j}^{k,\KP^{n-1}}\big\vert_{q=0}$ are given by
\begin{equation}\label{eqn:P_tilde_initial_q0}
\begin{split}
\widetilde{P}_{i,j}^{k,\KP^{n-1}}\big\vert_{q=0}
&=\left(\frac{(L^{\KP^{n-1}})^i}{K^{\KP^{n-1}}_i}P_{i,j}^{k,\KP^{n-1}}\zeta^{(k+i)j}\right)\Bigg\vert_{q=0}\\
&=\left(\frac{(L^{\KP^{n-1}})^i}{K^{\KP^{n-1}}_i}\left[\Psi^{-1}_{\KP^{n-1}}\right]_{i,j}\delta_{0,k}\zeta^{(k+i)j}\right)\Bigg\vert_{q=0}\\
&=-\sqrt{-1}\delta_{0,k}.
\end{split}
\end{equation}

The matrices $\widetilde{P}_{i,j}^{k,\KP^{n-1}}$ satisfy the modified flatness equations (\ref{eqn:modflateqn_for_KP}):
    \begin{equation*}
    \widetilde{P}_{\mathrm{Ion}(i)-1,j}^{k,\KP^{n-1}}=\widetilde{P}_{i,j}^{k,\KP^{n-1}}+\frac{1}{L^{\KP^{n-1}}}\mathsf{D}_{\KP^{n-1}}\widetilde{P}_{i,j}^{k-1,\KP^{n-1}}+A^{\KP^{n-1}}_{n-i}\widetilde{P}_{i,j}^{k-1,\KP^{n-1}}.
    \end{equation*}
These equations are independent of the index $j$. This means that their solutions are going to be independent of $j$ since the initial conditions $\widetilde{P}_{i,j}^{k,\KP^{n-1}}\big\vert_{q=0}$ are independent of $j$ by equation (\ref{eqn:P_tilde_initial_q0}). We know that
\begin{equation*}
\widetilde{P}_{0,j}^{k,\KP^{n-1}}=P_{0,j}^{k,\KP^{n-1}}\zeta^{jk}
\end{equation*}
where left-hand side is independent of $j$. So, we have
\begin{equation*}
P_{0,0}^{k,\KP^{n-1}}=\widetilde{P}_{0,0}^{k,\KP^{n-1}}=\widetilde{P}_{0,j}^{k,\KP^{n-1}}=P_{0,j}^{k,\KP^{n-1}}\zeta^{jk}.
\end{equation*}
Hence, letting $q=\infty$ in the analytic continuation completes the proof.
\end{proof}

Then, we see that Lemma \ref{lem:mathcing_of_P0jzs} is equivalent to the following statement:
\begin{lem}\label{lem:Final_of_proof_strategy}
We have
\begin{equation}\label{eqn:R_matrix_identity}
-\sqrt{-1}\exp\left(n\sum_{l>0}\frac{B_{nl+1}\left(0\right)}{nl+1}\frac{z^{nl}}{nl} \right)
=\left(\sum_{k\geq{0}}a_{0,0}^kz^k\right)\exp\left(\sum_{m>0}N_{2m-1, 0}\frac{(-1)^{2m-1}B_{2m}}{2m(2m-1)}z^{2m-1}\right).
\end{equation}
\end{lem}
We remark that Lemma \ref{lem:Final_of_proof_strategy} is the generalization of \cite[Lemma 22]{lho-p2}, and \cite[Proposition 11]{lho}. A proof of Lemma \ref{lem:Final_of_proof_strategy} is given in Section \ref{sec:asymptotics_of_oscillatory_integrals}. Hence, we complete the proof of Lemma \ref{lem:mathcing_of_P0jzs}.

\subsection{Formulas for Gromov--Witten potentials of \texorpdfstring{$\KP^{n-1}$}{KP{n-1}}}\label{sec:formulaGW}

The Gromov--Witten theory of $\KP^{n-1}$ has the structure of a cohomological field theory (CohFT). In Section \ref{sec:GWKP}, we explicitly showed that this CohFT is semisimple.

The Givental--Teleman classification for semisimple CohFTs \cite{g3}, \cite{t} establishes that a semisimple CohFT $\Omega$ can be reconstructed from its \textit{topological part} via the actions of $R$-matrix and $T$-vector. Here, the vector valued series $T(z)$ is defined as $z(\text{Id}-R(z))$ applied to the unit. Consequently, due to the Givental--Teleman classification, the generating functions of the CohFT $\Omega$ can be explicitly expressed as sums over graphs. For more detailed discussions on this topic, we refer the reader to consult \cite{Picm} and \cite{ppz}.

Section \ref{sec:GWKP} is devoted to the study of the $R$-matrix for the Gromov--Witten theory of $\KP^{n-1}$. Employing the general considerations on semisimple CohFTs, we obtain a formula for the Gromov--Witten potential $\mathcal{F}_{g, m}^{\KP^{n-1}}\left(H^{c_{1}}, \ldots, H^{c_{m}}\right)$. In the subsequent part of this subsection, we will elaborate on this formula in a comprehensive manner.
\subsubsection{Graphs}
We need to describe certain graphs to be able to state the formula for Gromov--Witten potentials.

A \textit{stable graph} $\Gamma$ is a tuple
\begin{equation*}
\left(\mathrm{V}_{\Gamma}, \mathrm{g}: \mathrm{V}_{\Gamma} \rightarrow \mathbb{Z}_{\geq 0},  \mathrm{H}_{\Gamma}, \iota: \mathrm{H}_{\Gamma} \rightarrow \mathrm{H}_{\Gamma}, \mathrm{E}_{\Gamma}, \mathrm{L}_{\Gamma}, \ell:\mathrm{L}_{\Gamma}\rightarrow\{1,\ldots,m\},  \nu: \mathrm{H}_{\Gamma} \rightarrow \mathrm{V}_{\Gamma}\right)
\end{equation*}
satisfying:
\begin{enumerate}
\item $\mathrm{V}_{\Gamma}$ is the vertex set, and $\mathrm{g}:\mathrm{V}_{\Gamma}\rightarrow\mathbb{Z}_{\geq 0}$ is a genus assignment,
    
\item $\mathrm{H}_{\Gamma}$ is the half-edge set, and $\iota: \mathrm{H}_{\Gamma} \rightarrow \mathrm{H}_{\Gamma}$ is an involution,
    
\item $\mathrm{E}_{\Gamma}$ is the set of edges\footnote{Self-edges are allowed.} defined by the orbits of $\iota: \mathrm{H}_{\Gamma} \rightarrow \mathrm{H}_{\Gamma}$ of size two, and the tuple $\left(\mathrm{V}_{\Gamma},\mathrm{E}_{\Gamma}\right)$ defines a connected graph,
    
\item $\mathrm{L}_{\Gamma}$ is the set of legs, the subset of $\mathrm{H}_{\Gamma}$ fixed by the involution $\iota: \mathrm{H}_{\Gamma} \rightarrow \mathrm{H}_{\Gamma}$ and the map $\ell:\mathrm{L}_{\Gamma}\rightarrow\{1,\ldots,m\}$ is an isomorphism labeling legs,

\item The map $\nu: \mathrm{H}_{\Gamma} \rightarrow \mathrm{V}_{\Gamma}$ is a vertex assignment,
    
\item For each vertex $\mathfrak{v}$, let $\mathrm{l}(\mathfrak{v})$ and $\mathrm{h}(\mathfrak{v})$ are the number of legs and the number of edges attached to the vertex $\mathfrak{v}$ respectively. If we denote $\mathrm{n}(\mathfrak{v})=\mathrm{l}(\mathfrak{v})+\mathrm{h}(\mathfrak{v})$ to be the valence of the vertex $\mathfrak{v}$, then for each vertex $\mathfrak{v}$ the following (stability) condition holds:
\begin{equation*}
2\mathrm{g}(\mathfrak{v})-2+\mathrm{n}(\mathfrak{v})>0.
\end{equation*}
\end{enumerate}

The \textit{genus} of $\Gamma$ is defined by
\begin{equation*}
\mathrm{g}(\Gamma)=h^1(\Gamma)+\sum_{\mathfrak{v}\in\mathrm{V}_{\Gamma}}\mathrm{g}(\mathfrak{v}).
\end{equation*}

We define a {\em decorated} stable graph $$\Gamma\in\mathrm{G}_{g,m}^{\text{Dec}}(n)$$ of order $n$ to be a stable graph $\Gamma\in\mathrm{G}_{g,m}$ equipped with an extra assignment $\mathrm{p}: \mathrm{V}_{\Gamma}\rightarrow \{0,...,n-1\}$ to each vertex $\mathfrak{v}\in\mathrm{V}_{\Gamma}$. For a decorated stable graph $\Gamma\in\mathrm{G}_{g,m}^{\text{Dec}}(n)$ we denote its underlying stable graph by $$\Gamma^{\mathrm{St}}\in\mathrm{G}_{g,m}$$ after forgetting the decoration.

In the formula graph sum for Gromov--Witten potentials, we work with decorated stable graphs. A detailed discussion on this can be found in \cite[Section 3.2]{gt}.

\subsubsection{Formula for \texorpdfstring{$\mathcal{F}_{g,m}$}{F{g,m}}}\label{sec:FgmKP}
By the discussions above, we have
\begin{equation}\label{eqn:formula_Fg}
\mathcal{F}_{g, m}^{\KP^{n-1}}\left(H^{c_{1}}, \ldots, H^{c_{m}}\right)=\sum_{\Gamma\in\mathrm{G}_{g,m}^{\text{Dec}}(n)}\mathrm{Cont}^{\KP^{n-1}}_{\Gamma}\left(H^{c_{1}}, \ldots, H^{c_{m}}\right).
\end{equation}

\begin{prop}\label{prop:contributions}
For each decorated stable graph $\Gamma\in\mathrm{G}_{g,m}^{\text{Dec}}(n)$, the associated contribution is given by
\begin{equation*}
\mathrm{Cont}^{\KP^{n-1}}_{\Gamma}\left(H^{c_{1}}, \ldots, H^{c_{m}}\right)=\frac{1}{|\mathrm{Aut}(\Gamma^{\mathrm{St}})|} \sum_{\mathrm{A} \in \mathbb{Z}_{\geq 0}^{\mathrm{F}(\Gamma)}} \prod_{\mathfrak{v} \in \mathrm{V}_{\Gamma}} \mathrm{Cont}_{\Gamma}^{\mathrm{A}}(\mathfrak{v}) \prod_{\mathfrak{e}\in \mathrm{E}_{\Gamma}} \mathrm{Cont}_{\Gamma}^{\mathrm{A}}(\mathfrak{e}) \prod_{\mathfrak{l} \in \mathrm{L}_{\Gamma}} \mathrm{Cont}_{\Gamma}^{\mathrm{A}}(\mathfrak{l})
\end{equation*}
where $\mathrm{F}(\Gamma)=\left\vert\mathrm{H}_{\Gamma}\right\vert$. Here, $\mathrm{Cont}_{\Gamma}^{\mathrm{A}}(\mathfrak{v})$, $\mathrm{Cont}_{\Gamma}^{\mathrm{A}}(\mathfrak{e})$, and $\mathrm{Cont}_{\Gamma}^{\mathrm{A}}(\mathfrak{l})$ are the {\em vertex}, {\em edge} and {\em leg} contributions with flag $\mathrm{A}-$values\footnote{Notation: The values ${b_{\mathfrak{v}1}},\ldots,{b_{\mathfrak{v}\mathrm{h}(\mathfrak{v})}}$ are the entries of $(a_1,\ldots,a_m,b_{m+1},\ldots,b_{\left\vert\mathrm{H}_{\Gamma}\right\vert})$ corresponding to $\mathrm{Cont}_{\Gamma}^{\mathrm{A}}(\mathfrak{v})$; where as, the values $b_{\mathfrak{e}1},b_{\mathfrak{e}2}$ are the entries of $(a_1,\ldots,a_m,b_{m+1},\ldots,b_{\left\vert\mathrm{H}_{\Gamma}\right\vert})$ corresponding to $\mathrm{Cont}_{\Gamma}^{\mathrm{A}}(\mathfrak{e})$.} $(a_1,\ldots,a_m,b_{m+1},\ldots,b_{\left\vert\mathrm{H}_{\Gamma}\right\vert})$ respectively, and they are given by
\begin{equation*}
\begin{split}
    \mathrm{Cont}_{\Gamma}^{\mathrm{A}}(\mathfrak{v})
    =&\sum_{k \geq 0} \frac{g^{\KP^{n-1}}({e}_{\mathrm{p}(\mathfrak{v})},{e}_{\mathrm{p}(\mathfrak{v})})^{-\frac{2\mathrm{g}(\mathfrak{v})-2+\mathrm{n}(\mathfrak{v})+k}{2}}}{k !}\\
    &\times\int_{\overline{M}_{\mathrm{g}(\mathfrak{v}),\mathrm{n}(\mathfrak{v})+k}}\psi_1^{a_{\mathfrak{v}1}}\cdots\psi_{\mathrm{l}(\mathfrak{v})}^{a_{\mathfrak{v}\mathrm{l}(\mathfrak{v})}}\psi_{\mathrm{l}(\mathfrak{v})+1}^{b_{\mathfrak{v}1}}\cdots\psi_{\mathrm{n}(\mathfrak{v})}^{b_{\mathfrak{v}\mathrm{h}(\mathfrak{v})}}t_{\mathrm{p}(\mathfrak{v})}(\psi_{\mathrm{n}(\mathfrak{v})+1})\cdots t_{\mathrm{p}(\mathfrak{v})}(\psi_{\mathrm{n}(\mathfrak{v})+k}),\\
    \mathrm{Cont}_{\Gamma}^{\mathrm{A}}(\mathfrak{e})
    =&\frac{(-1)^{b_{\mathfrak{e}1}+b_{\mathfrak{e}2}+1}}{n} \sum_{j=0}^{b_{\mathfrak{e}2}}(-1)^{j} \sum_{r=0}^{n-1}\frac{\widetilde{P}_{\mathrm{Inv}(r),\mathrm{p}(\mathfrak{v}_1)}^{b_{\mathfrak{e}1}+j+1,\KP^{n-1}}\widetilde{P}_{r,\mathrm{p}(\mathfrak{v}_2)}^{b_{\mathfrak{e}2}-j,\KP^{n-1}}}{\zeta^{(b_{\mathfrak{e}1}+j+1+\mathrm{Inv}(r))\mathrm{p}(\mathfrak{v}_1)}\zeta^{(b_{\mathfrak{e}2}-j+r)\mathrm{p}(\mathfrak{v}_2)}},\\
    \mathrm{Cont}_{\Gamma}^{\mathrm{A}}(\mathfrak{l})
    =&\frac{(-1)^{a_{\ell(\mathfrak{l})}+1}}{n}\frac{K^{\KP^{n-1}}_{\mathrm{Inv}(c_{\ell(\mathfrak{l})})}}{(L^{\KP^{n-1}})^{\mathrm{Inv}(c_{\ell(\mathfrak{l})})}}
    \frac{\widetilde{P}_{\mathrm{Inv}(c_{\ell(\mathfrak{l})}),\mathrm{p}(\nu(\mathfrak{l}))}^{{a_{\ell(\mathfrak{l})}},\KP^{n-1}}}{     \zeta^{({a_{\ell(\mathfrak{l})}}+{\mathrm{Inv}(c_{\ell(\mathfrak{l})})})\mathrm{p}(\nu(\mathfrak{l}))}},
\end{split}
\end{equation*}
where
\begin{equation*}
t_{\mathrm{p}(\mathfrak{v})}(z)=\sum_{i\geq{2}}\mathrm{T}_{\mathrm{p}(\mathfrak{v})i}z^i\quad\text{with}\quad \mathrm{T}_{\mathrm{p}(\mathfrak{v})i}=\frac{(-1)^{i+1}}{n}\widetilde{P}_{0,\mathrm{p}(\mathfrak{v})}^{i-1,\KP^{n-1}}\zeta^{-{(i-1)}\mathrm{p}(\mathfrak{v})}.
\end{equation*}
\end{prop}

\begin{proof}
We write $\{\tilde{e}\}$ for the normalized idempotent basis $\{\tilde{e}_0,\ldots,\tilde{e}_{n-1}\}$ and $\{H\}$ for the basis $\{1,H,\ldots,H^{n-1}\}$ to simplify the notation. Let $\mathcal{T}_{\tilde{e}}^{H}$ be the transition matrix from $\{\tilde{e}\}$ to $\{H\}$ and let $\mathcal{T}_{H}^{\tilde{e}}$ be its inverse. Then, we have $$\mathcal{T}_{\tilde{e}}^{H}=\Psi^{-1},\quad \mathcal{T}_{H}^{\tilde{e}}=\Psi.$$

Let $G$ and $\widetilde{G}$ be matrix representations of the metric $g^{\KP^{n-1}}$ with respect to basis $\{H\}$ and  $\{\tilde{e}\}$. Then, the relation between them is given by
\begin{equation}\label{eq:contproofeq1}
\widetilde{G}=\left(\Psi^{-1}\right)^TG\Psi^{-1}.
\end{equation}
It can easily be seen that the matrix $\widetilde{G}$ is the identity matrix.

Define $T(z)=z\left(\mathsf{Id}-\mathsf{R}^{\KP^{n-1}}(z)^{-1}\right)\cdot{1}$. We provided $R$-matrix action with respect to normalized idempotent basis. To be consistent we need to write $1=H^0$ in terms of $\{\tilde{e}\}$ basis. Since we have
\begin{equation}\label{eq:contproofeq2}
1=\sum_{i=0}^{n-1}\Psi_{i0}\tilde{e}_i=\frac{\sqrt{-1}}{n}\left(\tilde{e}_0+\ldots+\tilde{e}_{n-1}\right),
\end{equation}
we see that $T(z)=z\left(\mathsf{Id}-\mathsf{R}^{\KP^{n-1}}(z)^{-1}\right)v$ where $v=\frac{\sqrt{-1}}{n}[1\,\cdots\, 1]^T$.

We now find $\mathsf{R}^{\KP^{n-1}}(z)^{-1}$. By the symplectic condition, $\mathsf{R}^{\KP^{n-1}}(z)^{-1}=\mathsf{R}^{\KP^{n-1}}(-z)^{t}$. Here $\mathsf{R}^{\KP^{n-1}}(-z)^{t}$ means adjoint with respect to the metric $g^{\KP^{n-1}}$ in the basis $\{\tilde{e}\}$. We see that
\begin{equation}\label{eq:contproofeq3}
\mathsf{R}^{\KP^{n-1}}(z)^{-1}=\widetilde{G}^{-1}\mathsf{R}^{\KP^{n-1}}(-z)^{T}\widetilde{G}=\mathsf{R}^{\KP^{n-1}}(-z)^{T}=\left(\Psi \mathsf{P}^{\KP^{n-1}}(-z)\right)^T=\mathsf{P}^{\KP^{n-1}}(-z)^T\Psi^T.
\end{equation}

Also, observe that 
\begin{equation}\label{eq:contproofeq4}
\begin{split}
{\left[ \Psi^Tv\right]}_i
=&\frac{\sqrt{-1}}{n}\sum_{j=0}^{n-1}\Psi_{ij}^T=\frac{\sqrt{-1}}{n}\sum_{j=0}^{n-1}\Psi_{ji}\\
=&\frac{\sqrt{-1}}{n}\sum_{j=0}^{n-1}\frac{\sqrt{-1}}{n}\zeta^{ij}\frac{(L^{\KP^{n-1}})^i}{K^{\KP^{n-1}}_i}\\
=&-\frac{1}{n^2}\frac{(L^{\KP^{n-1}})^i}{K^{\KP^{n-1}}_i}\sum_{j=0}^{n-1}\zeta^{ij}=-\frac{1}{n}\delta_{i0},
\end{split}
\end{equation}
so we have $\Psi^Tv=-\frac{1}{n}\left[1\,0\,\cdots\,0\right]^T$. This implies that the translation vector
\begin{equation}\label{eq:contproofeq5}
T(z)=z\left(\mathsf{Id}-\mathsf{R}^{\KP^{n-1}}(z)^{-1}\right)v=T_2z^2+T_3z^3+\cdots
\end{equation}
where $T_k$ is the coefficient of $z^{k-1}$ in $-\mathsf{R}^{\KP^{n-1}}(z)^{-1}v$ given by 
\begin{equation}\label{eq:contproofeq6}
\begin{split}
 T_{jk}
 &=\text{the coefficient of $z^{k-1}$  in the $j^{\text{th}}$ entry of}-\mathsf{R}^{\KP^{n-1}}(z)^{-1}v\\
 &=\text{the coefficient of $z^{k-1}$  in the $j^{\text{th}}$ entry of} -\mathsf{P}^{\KP^{n-1}}(-z)^T\Psi^Tv\\
 &=\frac{(-1)^{k+1}}{n}P_{0j}^{k-1,\KP^{n-1}}.
\end{split}
\end{equation}
This allows us to comprehend the effects of the translation action by $T(z)$ and the contributions arising from vertices. However, the following computations are needed to understand the contributions originating from edges and legs.

Now observe that
\begin{equation}\label{eq:contproofeq8}
\begin{split}
{\left[\Psi^T \Psi\right]}_{l j}
&=
\sum_{r=0}^{n-1} \Psi_{r l} \Psi_{r j}=\sum_{r=0}^{n-1} \frac{\sqrt{-1}}{n}\zeta^{r l} \frac{(L^{\KP^{n-1}})^{l}}{K^{\KP^{n-1}}_{l}} \frac{\sqrt{-1}}{n}\zeta^{r j} \frac{(L^{\KP^{n-1}})^{j}}{K^{\KP^{n-1}}_{j}}\\
&=
-\frac{1}{n^2}\frac{(L^{\KP^{n-1}})^{l}}{K^{\KP^{n-1}}_{l}}\frac{(L^{\KP^{n-1}})^{j}}{K^{\KP^{n-1}}_{j}}\sum_{r=0}^{n-1}\zeta^{r(l-\mathrm{Inv}(j))}=-\frac{1}{n^2}\frac{(L^{\KP^{n-1}})^{l}}{K^{\KP^{n-1}}_{l}}\frac{(L^{\KP^{n-1}})^{j}}{K^{\KP^{n-1}}_{j}}n\delta_{l,\mathrm{Inv}(j)}\\
&=
-\frac{1}{n}\underbrace{\frac{(L^{\KP^{n-1}})^{\mathrm{Inv}(j)+j}}{K^{\KP^{n-1}}_{\mathrm{Inv}(j)}K^{\KP^{n-1}}_j}}_{=1}\delta_{l,\mathrm{Inv}(j)}=-\frac{1}{n}\delta_{l,\mathrm{Inv}(j)}.
\end{split}
\end{equation}

Next, in order to understand the edge contributions, we compute
\begin{equation}\label{eq:contproofeq9}
\begin{split}
&{\delta_{ij}-{\left[\mathsf{R}^{\KP^{n-1}}(z)^{-1}\left(\mathsf{R}^{\KP^{n-1}}(w)^{-1}\right)^T\right]}_{ij}}\\
&=\delta_{ij}-\sum_{s,r=0}^{n-1}\left(\mathsf{P}^{\KP^{n-1}}(-z)\right)_{i,s}^T\left[\Psi^T\Psi\right]_{sr} \mathsf{P}^{\KP^{n-1}}_{r,j}(-w)\\
&=\delta_{ij}+\sum_{s,r=0}^{n-1}P^{\KP^{n-1}}_{s,i}(-z)\frac{1}{n}\delta_{s\mathrm{Inv}(r)} P^{\KP^{n-1}}_{r,j}(-w)\\
&=\delta_{ij}+\frac{1}{n}\sum_{r=0}^{n-1}\sum_{c,d\geq{0}}(-1)^{c+d}P^{c,\KP^{n-1}}_{\mathrm{Inv}(r),i}P^{d,\KP^{n-1}}_{r,j}z^cw^d\\
&=\delta_{ij}+\frac{1}{n}\sum_{r=0}^{n-1}\sum_{c,d\geq{0}}(-1)^{c+d}\frac{K_{\mathrm{Inv}(r)}}{L^{\mathrm{Inv}(r)}}\frac{\widetilde{P}^{c,\KP^{n-1}}_{\mathrm{Inv}(r),i}}{\zeta^{(c+\mathrm{Inv}(r))i}}\frac{K_r}{L^r}\frac{\widetilde{P}^{d,\KP^{n-1}}_{r,j}}{\zeta^{(d+r)j}}z^cw^d\\
&=\delta_{ij}+\frac{1}{n}\sum_{r=0}^{n-1}\sum_{c,d\geq{0}}(-1)^{c+d}\frac{\widetilde{P}^{c,\KP^{n-1}}_{\mathrm{Inv}(r),i}\widetilde{P}^{d,\KP^{n-1}}_{r,j}}{\zeta^{(c+\mathrm{Inv}(r))i}\zeta^{(d+r)j}}z^cw^d.
\end{split}
\end{equation}
So, we have\footnote{To clarify this step, we refer the reader to \cite[Equation 3.20]{gt}.}
\begin{equation}\label{eq:contproofeq10}
\frac{{\delta_{ij}-{\left[\mathsf{R}^{\KP^{n-1}}(z)^{-1}\left(\mathsf{R}^{\KP^{n-1}}(w)^{-1}\right)^T\right]}_{ij}}}{z+w}=\sum_{b_1,b_2\geq{0}}\beta^{i,j}_{b_1,b_2}z^{b_1}w^{b_2}
\end{equation}
with
\begin{equation}\label{eq:contproofeq11}
\beta^{i,j}_{b_1,b_2}=\frac{(-1)^{b_1+b_2+1}}{n}\sum_{m=0}^{b_2}(-1)^m\sum_{r=0}^{n-1}\frac{\widetilde{P}^{b_1+m+1,\KP^{n-1}}_{\mathrm{Inv}(r),i}\widetilde{P}^{b_2-m,\KP^{n-1}}_{r,j}}{\zeta^{(b_1+m+1+\mathrm{Inv}(r))i}\zeta^{(b_2-m+r)j}}.
\end{equation}

In order to understand the leg contributions, we compute 
\begin{equation}\label{eq:contproofeq12}
\begin{split}
\left[\mathsf{R}^{\KP^{n-1}}(z)^{-1}\cdot H^j\right]_i
&=\left[\mathsf{P}^{\KP^{n-1}}(-z)^T\Psi^T\Psi\right]_{ij}=\sum_{a\geq{0}}(-1)^a\sum_{r=0}^{n-1}P_{r,i}^{a,\KP^{n-1}}\left(-\frac{1}{n}\right)\delta_{r,\mathrm{Inv}(j)}z^a\\
&=\sum_{a\geq{0}}\frac{(-1)^{a+1}}{n}\frac{K^{\KP^{n-1}}_{\mathrm{Inv}(j)}}{(L^{\KP^{n-1}})^{\mathrm{Inv}(j)}}\frac{\widetilde{P}^{a,\KP^{n-1}}_{\mathrm{Inv}(j),i}}{\zeta^{(a+\mathrm{Inv}(j))i}}z^a
\end{split}
\end{equation}
for each $0\leq{i,j}\leq{n-1}$.

The proof follows from the descriptions of $R$-matrix and $T$-vector actions.
\end{proof}

The following finite generation property of the Gromov--Witten potential $\mathcal{F}_{g, m}^{\KP^{n-1}}\left(H^{c_{1}}, \ldots, H^{c_{m}}\right)$ is a corollary of Proposition \ref{prop:contributions}.

\begin{cor}[Finite Generation Property]\label{cor:VertexEdgeCont}
The graph contributions $\mathrm{Cont}_{\Gamma}\left(H^{c_{1}}, \ldots, H^{c_{m}}\right)$ lie in certain polynomial rings. More, precisely
\begin{equation*}
\begin{split}
&\mathrm{Cont}_{\Gamma}^{\mathrm{A}}(\mathfrak{v})\in\mathbb{C}[(L^{\KP^{n-1}})^{\pm{1}}],\\
&\mathrm{Cont}_{\Gamma}^{\mathrm{A}}(\mathfrak{e})\in\mathbb{C}[(L^{\KP^{n-1}})^{\pm{1}}][\mathfrak{S}_n^{\KP^{n-1}}],\\
&\mathrm{Cont}_{\Gamma}^{\mathrm{A}}(\mathfrak{l})\in\mathbb{C}[(L^{\KP^{n-1}})^{\pm{1}}][\mathfrak{S}_n^{\KP^{n-1}}][\mathfrak{C}_n^{\KP^{n-1}}]=\mathds{F}_{\KP^{n-1}}
\end{split}
\end{equation*}
where $\mathfrak{C}_n^{\KP^{n-1}}=\{C^{\KP^{n-1}}_1,\ldots,C^{\KP^{n-1}}_{n-1}\}$. Hence, we have
\begin{equation*}
\mathcal{F}_{g, m}^{\KP^{n-1}}\left(H^{c_{1}}, \ldots, H^{c_{m}}\right)\in\mathbb{C}[(L^{\KP^{n-1}})^{\pm{1}}][\mathfrak{S}_n^{\KP^{n-1}}][\mathfrak{C}_n^{\KP^{n-1}}]=\mathds{F}_{\KP^{n-1}}.
\end{equation*}
\end{cor}

\begin{proof}
The integral in the expression of the vertex contribution $\mathrm{Cont}_{\Gamma}^{\mathrm{A}}(\mathfrak{v})$ is equal to
\begin{equation}
\int_{\overline{M}_{\mathrm{g}(\mathfrak{v}),\mathrm{n}(\mathfrak{v})+k}}\psi_1^{a_{\mathfrak{v}1}}\cdots  \psi_{\mathrm{l}(\mathfrak{v})}^{a_{\mathfrak{v}\mathrm{l}(\mathfrak{v})}} \psi_{\mathrm{l}(\mathfrak{v})+1}^{b_{\mathfrak{v}1}} \cdots \psi_{\mathrm{n}(\mathfrak{v})}^{b_{\mathfrak{v}\mathrm{h}(\mathfrak{v})}}\prod_{j=1}^{k}\underbrace{\left(\sum_{i_j\geq{2}}\frac{(-1)^{i_j}}{n}\frac{\widetilde{P}_{0,\mathrm{p}(\mathfrak{v})}^{i_j-1,\KP^{n-1}}\psi_{\mathrm{n}(\mathfrak{v})+j}^{i_j}}{\zeta^{{(i_j-1)}\mathrm{p}(\mathfrak{v})}}\right)}_{\eqqcolon\Lambda_j}.
\end{equation}
Each summand in this is $0$ unless
\begin{equation}\label{eqn:dim_condition_for_integral}
a_{\mathfrak{v}1}+\cdots + a_{\mathfrak{v}\mathrm{l}(\mathfrak{v})}+b_{\mathfrak{v}1}+\cdots + b_{\mathfrak{v}\mathrm{h}(\mathfrak{v})}+i_1+\cdots + i_k=\dim \overline{M}_{\mathrm{g}(\mathfrak{v}),\mathrm{n}(\mathfrak{v})+k}=3\mathrm{g}(\mathfrak{v})-3+\mathrm{n}(\mathfrak{v})+k.
\end{equation}
Since each $i_j\geq 2$, the integral is 0 when $k>3\mathrm{g}(\mathfrak{v})-3+\mathrm{n}(\mathfrak{v})$. So, the vertex contribution $\mathrm{Cont}_{\Gamma}^{\mathrm{A}}(\mathfrak{v})$ is a finite sum over $k$. In a similar way, the integral is $0$ when one of $i_j>3\mathrm{g}(\mathfrak{v})-2+\mathrm{n}(\mathfrak{v})$. So, each $\Lambda_j$ can also be considered as a finite sum. This implies that the vertex contribution $\mathrm{Cont}_{\Gamma}^{\mathrm{A}}(\mathfrak{v})$ is a polynomial in $\widetilde{P}_{0,\mathrm{p}(\mathfrak{v})}^{i_j-1,\KP^{n-1}}$. Hence, it is a polynomial in $L^{\KP^{n-1}}$ by Corollary \ref{cor:polynomiality_of_P_0j_KP}.

For the edge contribution $\mathrm{Cont}_{\Gamma}^{\mathrm{A}}(\mathfrak{e})$  and leg contribution $\mathrm{Cont}_{\Gamma}^{\mathrm{A}}(\mathfrak{l})$, the polynomiality claims follow from the lifting procedure (\ref{eq:ModifiedFlatnessLift}) and the definition of $K_i^{\KP^{n-1}}$.

Equation (\ref{eqn:dim_condition_for_integral}) also implies that all but finitely many flag $\mathrm{A}-$values have $0$ contribution to $\mathrm{Cont}_{\Gamma}\left(H^{c_{1}}, \ldots, H^{c_{m}}\right)$. This implies that
\begin{equation*}
\mathrm{Cont}_{\Gamma}\left(H^{c_{1}}, \ldots, H^{c_{m}}\right)\in \mathds{F}_{\KP^{n-1}}.
\end{equation*}
So, the finite generation result for the Gromov--Witten potential $\mathcal{F}_{g, m}^{\KP^{n-1}}\left(H^{c_{1}}, \ldots, H^{c_{m}}\right)$ follows.
\end{proof}

\subsection{Crepant resolution correspondence for \texorpdfstring{$\KP^{n-1}$}{KP{n-1}} and \texorpdfstring{$[\mathbb{C}^n/\mathbb{Z}_n]$}{CnZn}}\label{sec:CRC}
The Gromov--Witten potential of $\CnZn$ is also described as a graph sum formula in \cite{gt}:
\begin{equation}\label{eqn:formula_Fg2}
\mathcal{F}_{g, m}^{\left[\mathbb{C}^{n} / \mathbb{Z}_{n}\right]}\left(\phi_{c_{1}}, \ldots, \phi_{c_{m}}\right)=\sum_{\Gamma\in\mathrm{G}_{g,m}^{\text{Dec}}(n)}\mathrm{Cont}_{\Gamma}^{\CnZn}\left(\phi_{c_{1}}, \ldots, \phi_{c_{m}}\right).
\end{equation}
where ${\mathrm{Cont}}^{\CnZn}_{\Gamma}\left(\phi_{c_{1}}, \ldots, \phi_{c_{m}}\right)$ is given in \cite[Proposition 3.3]{gt} in a similar fashion to Proposition \ref{prop:contributions}. We restate this result for the convenience of readers.

\begin{prop}[\cite{gt}]\label{prop:CnZn_ 
  contributions}
For each decorated stable graph $\Gamma\in\mathrm{G}_{g,m}^{\text{Dec}}(n)$, the associated contribution is given by
\begin{equation*}
\mathrm{Cont}_{\Gamma}^{\CnZn}\left(\phi_{c_{1}}, \ldots, \phi_{c_{m}}\right)=\frac{1}{|\mathrm{Aut}(\Gamma^{\mathrm{St}})|} \sum_{\mathrm{A} \in \mathbb{Z}_{\geq 0}^{\mathrm{F}(\Gamma)}} \prod_{\mathfrak{v} \in \mathrm{V}_{\Gamma}} \widetilde{\mathrm{Cont}}_{\Gamma}^{\mathrm{A}}(\mathfrak{v}) \prod_{\mathfrak{e}\in \mathrm{E}_{\Gamma}} \widetilde{\mathrm{Cont}}_{\Gamma}^{\mathrm{A}}(\mathfrak{e}) \prod_{\mathfrak{l} \in \mathrm{L}_{\Gamma}} \widetilde{\mathrm{Cont}}_{\Gamma}^{\mathrm{A}}(\mathfrak{l})
\end{equation*}
where $\mathrm{F}(\Gamma)=\left\vert\mathrm{H}_{\Gamma}\right\vert$. Here, $\widetilde{\mathrm{Cont}}_{\Gamma}^{\mathrm{A}}(\mathfrak{v})$, $\widetilde{\mathrm{Cont}}_{\Gamma}^{\mathrm{A}}(\mathfrak{e})$, and $\widetilde{\mathrm{Cont}}_{\Gamma}^{\mathrm{A}}(\mathfrak{l})$ are the {\em vertex}, {\em edge} and {\em leg} contributions with flag $\mathrm{A}-$values $(a_1,\ldots,a_m,b_{m+1},\ldots,b_{\left\vert\mathrm{H}_{\Gamma}\right\vert})$ respectively, and they are given by\footnote{The vectors ${e}_{\mathrm{p}(\mathfrak{v})}$ are the idempotent basis associated to $\CnZn$, see \cite{gt}.}
\begin{equation*}
\begin{split}
    \widetilde{\mathrm{Cont}}_{\Gamma}^{\mathrm{A}}(\mathfrak{v})
    =&\sum_{k \geq 0} \frac{g^{\CnZn}({e}_{\mathrm{p}(\mathfrak{v})},{e}_{\mathrm{p}(\mathfrak{v})})^{-\frac{2\mathrm{g}(\mathfrak{v})-2+\mathrm{n}(\mathfrak{v})+k}{2}}}{k !}\\
    &\times\int_{\overline{M}_{\mathrm{g}(\mathfrak{v}),\mathrm{n}(\mathfrak{v})+k}}\psi_1^{a_{\mathfrak{v}1}}\cdots\psi_{\mathrm{l}(\mathfrak{v})}^{a_{\mathfrak{v}\mathrm{l}(\mathfrak{v})}}\psi_{\mathrm{l}(\mathfrak{v})+1}^{b_{\mathfrak{v}1}}\cdots\psi_{\mathrm{n}(\mathfrak{v})}^{b_{\mathfrak{v}\mathrm{h}(\mathfrak{v})}}t_{\mathrm{p}(\mathfrak{v})}(\psi_{\mathrm{n}(\mathfrak{v})+1})\cdots t_{\mathrm{p}(\mathfrak{v})}(\psi_{\mathrm{n}(\mathfrak{v})+k}),\\
    \widetilde{\mathrm{Cont}}_{\Gamma}^{\mathrm{A}}(\mathfrak{e})
    =&\frac{(-1)^{b_{\mathfrak{e}1}+b_{\mathfrak{e}2}}}{n} \sum_{j=0}^{b_{\mathfrak{e}2}}(-1)^{j} \sum_{r=0}^{n-1}\frac{\widetilde{P}_{\mathrm{Inv}(r),\mathrm{p}(\mathfrak{v}_1)}^{b_{\mathfrak{e}1}+j+1,\CnZn}\widetilde{P}_{r,\mathrm{p}(\mathfrak{v}_2)}^{b_{\mathfrak{e}2}-j,\CnZn}}{\zeta^{(b_{\mathfrak{e}1}+j+1+\mathrm{Inv}(r))\mathrm{p}(\mathfrak{v}_1)}\zeta^{(b_{\mathfrak{e}2}-j+r)\mathrm{p}(\mathfrak{v}_2)}},\\
    \widetilde{\mathrm{Cont}}_{\Gamma}^{\mathrm{A}}(\mathfrak{l})
    =&\frac{(-1)^{a_{\ell(\mathfrak{l})}}}{n}\frac{K_{\mathrm{Inv}(c_{\ell(\mathfrak{l})})}}{L^{\mathrm{Inv}(c_{\ell(\mathfrak{l})})}}
    \frac{\widetilde{P}_{\mathrm{Inv}(c_{\ell(\mathfrak{l})}),\mathrm{p}(\nu(\mathfrak{l}))}^{a_{\ell(\mathfrak{l})},\CnZn}}{     \zeta^{({a_{\ell(\mathfrak{l})}}+{\mathrm{Inv}(c_{\ell(\mathfrak{l})})})\mathrm{p}(\nu(\mathfrak{l}))}},
\end{split}
\end{equation*}
where
\begin{equation*}
t_{\mathrm{p}(\mathfrak{v})}(z)=\sum_{i\geq{2}}\mathrm{T}_{\mathrm{p}(\mathfrak{v})i}z^i\quad\text{with}\quad \mathrm{T}_{\mathrm{p}(\mathfrak{v})i}=\frac{(-1)^i}{n}\widetilde{P}_{0,\mathrm{p}(\mathfrak{v})}^{i-1,\CnZn}\zeta^{-(i-1)\mathrm{p}(\mathfrak{v})}.
\end{equation*}
\end{prop}

\begin{thm}[Crepant Resolution Correspondence]\label{thm:Main_Theorem}
For $g$ and $m$ in the stable range $2g-2+m>0$, the ring isomorphism $\Upsilon$ yields
\begin{equation*}
\mathcal{F}_{g, m}^{\CnZn}\left(\phi_{c_{1}}, \ldots, \phi_{c_{m}}\right)=(-1)^{1-g}\rho^{3g-3+m}\Upsilon \left(\mathcal{F}_{g, m}^{\KP^{n-1}}\left(H^{c_{1}}, \ldots, H^{c_{m}}\right)\right).
\end{equation*}
\end{thm}

\begin{proof}
For a decorated stable graph $\Gamma$, let $\widetilde{\mathrm{Cont}}_{\Gamma}^{\mathrm{A}}(\mathfrak{v})$, $\widetilde{\mathrm{Cont}}_{\Gamma}^{\mathrm{A}}(\mathfrak{e})$, and $\widetilde{\mathrm{Cont}}_{\Gamma}^{\mathrm{A}}(\mathfrak{l})$ be the vertex, edge, and leg contributions for the potential $\mathcal{F}_{g, m}^{\CnZn}\left(\phi_{c_{1}}, \ldots, \phi_{c_{m}}\right)$ described in \cite[Proposition 3.3]{gt} for a flag $A$-value. For the same flag $A$-value, and the same decorated stable graph $\Gamma$ let $\mathrm{Cont}_{\Gamma}^{\mathrm{A}}(\mathfrak{v})$, $\mathrm{Cont}_{\Gamma}^{\mathrm{A}}(\mathfrak{e})$, and $\mathrm{Cont}_{\Gamma}^{\mathrm{A}}(\mathfrak{l})$ be the vertex, edge, and leg contributions for $\mathcal{F}_{g, m}^{\KP^{n-1}}\left(H^{c_{1}}, \ldots, H^{c_{m}}\right)$ in Proposition \ref{prop:contributions}.

The isomorphism $\Upsilon$ identifies $P_{i,j}^{k,\KP^{n-1}}$ with $-\sqrt{-1}\rho^{-k}P_{i,j}^{k,\CnZn}$. Under this identification, we will analyze what happens to $\mathrm{Cont}_{\Gamma}^{\mathrm{A}}(\mathfrak{v})$, $\mathrm{Cont}_{\Gamma}^{\mathrm{A}}(\mathfrak{e})$, and $\mathrm{Cont}_{\Gamma}^{\mathrm{A}}(\mathfrak{l})$. We start with the effect of $\Upsilon$ on $\mathrm{Cont}_{\Gamma}^{\mathrm{A}}(\mathfrak{l})$:
\begin{equation}\label{eqn:leg_contribution_after_upsilon}
\begin{aligned}
\Upsilon\left(\mathrm{Cont}_{\Gamma}^{\mathrm{A}}(\mathfrak{l})\right)
=&\frac{(-1)^{a_{\ell(\mathfrak{l})}+1}}{n}\frac{K^{\CnZn}_{\mathrm{Inv}(c_{\ell(\mathfrak{l})})}}{(L^{\CnZn})^{\mathrm{Inv}(c_{\ell(\mathfrak{l})})}}
\frac{-\sqrt{-1}\rho^{-{a_{\ell(\mathfrak{l})}}}\widetilde{P}_{\mathrm{Inv}(c_{\ell(\mathfrak{l})}),\mathrm{p}(\nu(\mathfrak{l}))}^{{a_{\ell(\mathfrak{l})}},\CnZn}}{     \zeta^{({a_{\ell(\mathfrak{l})}}+{\mathrm{Inv}(c_{\ell(\mathfrak{l})})})\mathrm{p}(\nu(\mathfrak{l}))}}\\
=&\sqrt{-1}\rho^{-{a_{\ell(\mathfrak{l})}}}\frac{(-1)^{a_{\ell(\mathfrak{l})}}}{n}\frac{K^{\CnZn}_{\mathrm{Inv}(c_{\ell(\mathfrak{l})})}}{(L^{\CnZn})^{\mathrm{Inv}(c_{\ell(\mathfrak{l})})}}
\frac{\widetilde{P}_{\mathrm{Inv}(c_{\ell(\mathfrak{l})}),\mathrm{p}(\nu(\mathfrak{l}))}^{{a_{\ell(\mathfrak{l})}},\CnZn}}{     \zeta^{({a_{\ell(\mathfrak{l})}}+{\mathrm{Inv}(c_{\ell(\mathfrak{l})})})\mathrm{p}(\nu(\mathfrak{l}))}} \\
=&\sqrt{-1}\rho^{-{a_{\ell(\mathfrak{l})}}}\widetilde{\mathrm{Cont}}_{\Gamma}^{\mathrm{A}}(\mathfrak{l}).
\end{aligned}
\end{equation}
Observe the effect of $\Upsilon$ on $\Upsilon\left(\mathrm{Cont}_{\Gamma}^{\mathrm{A}}(\mathfrak{e})\right)$:
\begin{equation}\label{eqn:edge_contribution_after_upsilon}
\begin{aligned}
\Upsilon\left(\mathrm{Cont}_{\Gamma}^{\mathrm{A}}(\mathfrak{e})\right)
=&\frac{(-1)^{b_{\mathfrak{e}1}+b_{\mathfrak{e}2}+1}}{n} \sum_{j=0}^{b_{\mathfrak{e}2}}(-1)^{j} \sum_{r=0}^{n-1}\frac{\left(-\rho^{-(b_{\mathfrak{e}1}+b_{\mathfrak{e}2}+1)}\right)\widetilde{P}_{\mathrm{Inv}(r),\mathrm{p}(\mathfrak{v}_1)}^{b_{\mathfrak{e}1}+j+1,\CnZn}\widetilde{P}_{r,\mathrm{p}(\mathfrak{v}_2)}^{b_{\mathfrak{e}2}-j,\CnZn}}{\zeta^{(b_{\mathfrak{e}1}+j+1+\mathrm{Inv}(r))\mathrm{p}(\mathfrak{v}_1)}\zeta^{(b_{\mathfrak{e}2}-j+r)\mathrm{p}(\mathfrak{v}_2)}}\\
=&\rho^{-b_{\mathfrak{e}1}}\rho^{-b_{\mathfrak{e}2}}\rho^{-1}\frac{(-1)^{b_{\mathfrak{e}1}+b_{\mathfrak{e}2}}}{n} \sum_{j=0}^{b_{\mathfrak{e}2}}(-1)^{j} \sum_{r=0}^{n-1}\frac{\widetilde{P}_{\mathrm{Inv}(r),\mathrm{p}(\mathfrak{v}_1)}^{b_{\mathfrak{e}1}+j+1,\CnZn}\widetilde{P}_{r,\mathrm{p}(\mathfrak{v}_2)}^{b_{\mathfrak{e}2}-j,\CnZn}}{\zeta^{(b_{\mathfrak{e}1}+j+1+\mathrm{Inv}(r))\mathrm{p}(\mathfrak{v}_1)}\zeta^{(b_{\mathfrak{e}2}-j+r)\mathrm{p}(\mathfrak{v}_2)}}\\
=&\rho^{-b_{\mathfrak{e}1}}\rho^{-b_{\mathfrak{e}2}}\rho^{-1}\widetilde{\mathrm{Cont}}_{\Gamma}^{\mathrm{A}}(\mathfrak{e}).
\end{aligned}
\end{equation}
Since we moved all $\psi$-classes to the vertex contribution in Proposition \ref{prop:contributions}, we will move $\rho^{-{a_{\ell(\mathfrak{l})}}}$ in equation (\ref{eqn:leg_contribution_after_upsilon}) and $\rho^{-b_{\mathfrak{e}i}}$ $(i=1,2)$ in equation (\ref{eqn:edge_contribution_after_upsilon}) to the $\Upsilon\left(\mathrm{Cont}_{\Gamma}^{\mathrm{A}}(\mathfrak{v})\right)$ and view equations (\ref{eqn:leg_contribution_after_upsilon}), and (\ref{eqn:edge_contribution_after_upsilon}) as
\begin{equation}\label{eqn:Upsilon_edge_and_leg_final}
\begin{aligned}
\Upsilon\left(\mathrm{Cont}_{\Gamma}^{\mathrm{A}}(\mathfrak{l})\right)
=&\sqrt{-1}\widetilde{\mathrm{Cont}}_{\Gamma}^{\mathrm{A}}(\mathfrak{l}),\\
\Upsilon\left(\mathrm{Cont}_{\Gamma}^{\mathrm{A}}(\mathfrak{e})\right)
=&\rho^{-1}\widetilde{\mathrm{Cont}}_{\Gamma}^{\mathrm{A}}(\mathfrak{e}),
\end{aligned}
\end{equation}
and we can view $\Upsilon\left(\mathrm{Cont}_{\Gamma}^{\mathrm{A}}(\mathfrak{v})\right)$ as
\begin{equation*}
\Upsilon  \left(\mathrm{Cont}_{\Gamma}^{\mathrm{A}}(\mathfrak{v})\right)
=\sum_{k \geq 0} \frac{\left(\sqrt{-1}\right)^{2-2\mathrm{g}(\mathfrak{v})-\mathrm{n}(\mathfrak{v})-k}g^{\CnZn}({e}_{\mathrm{p}(\mathfrak{v})},{e}_{\mathrm{p}(\mathfrak{v})})^{-\frac{2\mathrm{g}(\mathfrak{v})-2+\mathrm{n}(\mathfrak{v})+k}{2}}}{k !}\Lambda_{\mathfrak{v},k}
\end{equation*}
where $\Lambda_{\mathfrak{v},k}$ is given by
\begin{equation*}
\left(\sqrt{-1}\right)^k\int_{\overline{M}_{\mathrm{g}(\mathfrak{v}),\mathrm{n}(\mathfrak{v})+k}}\frac{\psi_1^{a_{\mathfrak{v}1}}}{\rho^{a_{\mathfrak{v}1}}}\cdots  \frac{\psi_{\mathrm{l}(\mathfrak{v})}^{a_{\mathfrak{v}\mathrm{l}(\mathfrak{v})}}}{\rho^{a_{\mathfrak{v}\mathrm{l}(\mathfrak{v})}}} \frac{\psi_{\mathrm{l}(\mathfrak{v})+1}^{b_{\mathfrak{v}1}}}{\rho^{b_{\mathfrak{v}1}}} \cdots \frac{\psi_{\mathrm{n}(\mathfrak{v})}^{b_{\mathfrak{v}\mathrm{h}(\mathfrak{v})}}}{\rho^{b_{\mathfrak{v}\mathrm{h}(\mathfrak{v})}}}\prod_{j=1}^{k}\left(\sum_{i_j\geq{2}}\frac{(-1)^{i_j}}{n}\frac{\widetilde{P}_{0,\mathrm{p}(\mathfrak{v})}^{i_j-1,\CnZn}\psi_{\mathrm{n}(\mathfrak{v})+j}^{i_j}}{\rho^{i_j-1}\zeta^{{(i_j-1)}\mathrm{p}(\mathfrak{v})}}\right).
\end{equation*}
Since $\dim \overline{M}_{\mathrm{g}(\mathfrak{v}),\mathrm{n}(\mathfrak{v})+k}=3\mathrm{g}(\mathfrak{v})-3+\mathrm{n}(\mathfrak{v})+k$, the above integral is $0$ unless
\begin{equation*}
a_{\mathfrak{v}1}+\cdots + a_{\mathfrak{v}\mathrm{l}(\mathfrak{v})}+b_{\mathfrak{v}1}+\cdots + b_{\mathfrak{v}\mathrm{h}(\mathfrak{v})}+i_1+\cdots + i_k=3\mathrm{g}(\mathfrak{v})-3+\mathrm{n}(\mathfrak{v})+k.
\end{equation*}
In this case, the sum of powers of $\rho$ is
\begin{equation*}
-(a_{\mathfrak{v}1}+\cdots + a_{\mathfrak{v}\mathrm{l}(\mathfrak{v})}+b_{\mathfrak{v}1}+\cdots + b_{\mathfrak{v}\mathrm{h}(\mathfrak{v})}+i_1+\cdots + i_k-k)=3-3\mathrm{g}(\mathfrak{v})-\mathrm{n}(\mathfrak{v}).
\end{equation*}
Hence, we get
\begin{equation}\label{eqn:Upsilon_vertex_final}
\Upsilon\left(\mathrm{Cont}_{\Gamma}^{\mathrm{A}}(\mathfrak{v})\right)=\rho^{3-3\mathrm{g}(\mathfrak{v})-\mathrm{n}(\mathfrak{v})}\left(\sqrt{-1}\right)^{2-2\mathrm{g}(\mathfrak{v})-\mathrm{n}(\mathfrak{v})}\widetilde{\mathrm{Cont}}_{\Gamma}^{\mathrm{A}}(\mathfrak{v}).
\end{equation}
By Euler's graph formula, we have
\begin{equation*}
\vert\mathrm{V}_{\Gamma}\vert-\vert\mathrm{E}_{\Gamma}\vert+h^1(\Gamma)=1,
\end{equation*}
and hence we have
\begin{equation*}
\begin{aligned}
g-1
&=\vert\mathrm{E}_{\Gamma}\vert-\vert\mathrm{V}_{\Gamma}\vert+\sum_{\mathfrak{v}\in\mathrm{V}_{\Gamma}}\mathrm{g}(\mathfrak{v})\\
&=\vert\mathrm{E}_{\Gamma}\vert+\sum_{\mathfrak{v}\in\mathrm{V}_{\Gamma}}\left(\mathrm{g}(\mathfrak{v})-1\right).
\end{aligned}
\end{equation*}
Also, we know
\begin{equation*}
\sum_{\mathfrak{v}\in\mathrm{V}_{\Gamma}}\mathrm{n}(\mathfrak{v})=2\vert\mathrm{E}_{\Gamma}\vert+\vert\mathrm{L}_{\Gamma}\vert=2\vert\mathrm{E}_{\Gamma}\vert+m.
\end{equation*}
Together with this basic graph theory of stable graphs, equations (\ref{eqn:Upsilon_edge_and_leg_final}), and (\ref{eqn:Upsilon_vertex_final}) give us that contributions arising from $\Gamma$ to the Gromov--Witten potentials are related to each other via
\begin{equation*}
\mathrm{Cont}_{\Gamma}^{\CnZn}\left(\phi_{c_{1}}, \ldots, \phi_{c_{m}}\right)=(-1)^{1-g}\rho^{3g-3+m}\Upsilon \left(\mathrm{Cont}^{\KP^{n-1}}_{\Gamma}\left(H^{c_{1}}, \ldots, H^{c_{m}}\right)\right).
\end{equation*}
\end{proof}

\begin{rem}
Theorem \ref{thm:Main_Theorem} is the generalization of \cite[Theorem 4']{lho-p2} when $\rho$ is chosen to be $-1$.
\end{rem}

\begin{rem}
Theorem \ref{thm:Main_Theorem} implies that the Gromov--Witten potential satisfies the holomorphic anomaly equations proved in \cite{gt} after the identifications we introduced in Section \ref{subsubsec:Change_of_variables}.
\end{rem}

\section{Asymptotics of oscillatory integrals} \label{sec:asymptotics_of_oscillatory_integrals}
The goal of this section is to prove Lemma \ref{lem:Final_of_proof_strategy}, following the strategy of \cite[Appendix]{lho-p2}.

The (equivariant) Landau--Ginzburg mirror to $K\mathbb{P}^{n-1}$ is 
\begin{equation*}
    F=w_0+w_1+...+w_{n-1}+w_n+\sum_{i=0}^{n-1}\chi_i \log w_i,
\end{equation*}
defined on the family of affine varieties
\begin{equation*}
    Y_q=\{(w_0,...,w_{n-1}, w_n)\in \mathbb{C}^{n+1} \,|\, w_0w_1...w_{n-1}=qw_n^n\}.
\end{equation*}
The associated oscillatory integral is of the form
\begin{equation}\label{eqn:oscillatory_int}
\mathcal{I}=\int_{\Gamma\subset Y_q} e^{F/z}g(w)\omega,    
\end{equation}
where $\omega$ is the meromorphic volume form on $Y_q$:
\begin{equation*}
\omega=\frac{d\log w_0\wedge d\log w_1\wedge...\wedge d\log w_n}{d\log q}.    
\end{equation*}
In the coordinate system $(w_0,w_1,...,w_{n-1})$ on $Y_q$, we have 
\begin{equation*}
\mathcal{I}=\int_{\Gamma\subset (\mathbb{C}^*)^n} e^{(w_0+w_1+...+w_{n-1}+q^{-1/n}(w_0...w_{n-1})^{1/n}+\sum_{i=0}^{n-1}\chi_i \log w_i)/z}g(w)\frac{1}{n}\frac{dw_0...dw_{n-1}}{w_0...w_{n-1}}.    
\end{equation*}

We impose the specialization (\ref{eqn:specialization_KP}). The critical points of $F$ are calculated as follows. For $0\leq i\leq n-1$, the critical point equation $\frac{\partial F}{\partial w_i}=0$ reads 
\begin{equation*}
1+\frac{1}{n}q^{-1/n}(w_0...w_{n-1})^{1/n-1}(w_0...\widehat{w_i}...w_{n-1})+\frac{\chi_i}{w_i}=0,     
\end{equation*}
which is the same as 
\begin{equation}\label{eqn:crit_pt1}
w_i=-\frac{1}{n}q^{-1/n}(w_0...w_{n-1})^{1/n}-\chi_i.    
\end{equation}
Multiplying equation (\ref{eqn:crit_pt1}) for $0\leq i\leq n-1$, we obtain
\begin{equation*}
\prod_{i=0}^{n-1}w_i=\prod_{i=0}^{n-1}\left(-\frac{1}{n}q^{-1/n}(w_0...w_{n-1})^{1/n}-\chi_i\right).     
\end{equation*}
By the equation of $Y_q$, the left-hand side is $qw_n^n$. By the specialization (\ref{eqn:specialization_KP}), the right-hand side is 
\begin{equation*}
\left(-\frac{1}{n}q^{-1/n}(w_0...w_{n-1})^{1/n}\right)^n-1=\left(-\frac{1}{n}\right)^n w_n^n-1.    
\end{equation*}
This implies 
\begin{equation*}
w_n=\left(\left(-\frac{1}{n}\right)^n-q\right)^{-1/n}, \quad w_i=-\frac{1}{n}\left(\left(-\frac{1}{n}\right)^n-q\right)^{-1/n}-\chi_i, \quad 0\leq i\leq n-1.,    
\end{equation*}
i.e.,
\begin{equation}\label{eqn:crit_pt2}
w_n=-n(1-(-n)^nq)^{-1/n}, \quad w_i=(1-(-n)^nq)^{-1/n}-\chi_i, \quad 0\leq i\leq n-1.    
\end{equation}
The $n$ choices of the branch for $$L^{\KP^{n-1}}=(1-(-n)^nq)^{-1/n}$$ give rise to $n$ critical points.

Assume $q>0$ and choose the critical point corresponding to a real positive $(1-(-n)^nq)^{-1/n}$. Denote by $\mathrm{w}_{\text{cr}}$ the critical point (\ref{eqn:crit_pt2}). The corresponding critical value is 
\begin{equation*}
F(\mathrm{w}_{\text{cr}})=\sum_{i=0}^{n-1}\chi_i\log \left((1-(-n)^nq)^{-1/n}-\chi_i\right)=\sum_{i=0}^{n-1}\chi_i\log(L^{\KP^{n-1}}-\chi_i).    
\end{equation*}
Using the definition of $L^{\KP^{n-1}}$, we calculate (recall that we impose the specialization (\ref{eqn:specialization_KP}))
\begin{equation*}
q\frac{d}{dq} F(\mathrm{w}_{\text{cr}})=\sum_{i=0}^{n-1}\chi_i\frac{q\frac{d}{dq}L^{\KP^{n-1}}}{L^{\KP^{n-1}}-\chi_i}=\frac{n}{(L^{\KP^{n-1}})^n-1}q\frac{d}{dq}L^{\KP^{n-1}}=L^{\KP^{n-1}}.    
\end{equation*}
It follows that\footnote{Note that this is the decomposition of the critical value as a sum of "classical" and "quantum" parts, c.f. \cite[Lemma 6.4]{ccit_toric3}.} 
\begin{equation*}
F(\mathrm{w}_{\text{cr}})=\sum_{i=1}^{n-1}\chi_i\log(1-\chi_i)+\underbrace{\int_0^q \left(L^{\KP^{n-1}}-1\right)\frac{dq}{q}}_{=:\mu}+\log((-1)^nn^{n-1}q).    
\end{equation*}

We calculate the Hessian of $F$ at $\mathrm{w}_{\text{cr}}$ as follows. 
\begin{equation*}
\begin{split}
\frac{\partial}{\partial \log w_i} F=&w_i\frac{\partial}{\partial w_i} F=w_i\left(1+\frac{\chi_i}{w_i}+\frac{\partial}{\partial w_i} w_n \right)\\
=&w_i\left(1+\frac{\chi_i}{w_i}+\frac{\partial}{\partial w_i} q^{-1/n}(w_0...w_{n-1})^{1/n} \right)\\
=&w_i+\chi_i+\frac{1}{n}q^{-1/n}(w_0...w_{n-1})^{1/n}. 
\end{split}
\end{equation*}
\begin{equation*}
\frac{\partial^2}{\partial \log w_j\partial \log w_i} F=\delta_{i,j}w_j+\frac{1}{n^2}q^{-1/n}(w_0...w_{n-1})^{1/n}=\delta_{i,j}w_j+\frac{1}{n^2}w_n.
\end{equation*}
It follows that
\begin{equation*}
\text{det}\left(\frac{\partial^2 F(\mathrm{w}_{\text{cr}})}{\partial \log w_j\partial \log w_i} \right)= w_0...w_{n-1}+\frac{1}{n^2}w_n\sum_{i=0}^{n-1}(w_0...\widehat{w_i}...w_{n-1}).
\end{equation*}
Using (\ref{eqn:crit_pt2}) and the definition of $Y_q$, this is 
\begin{equation*}
\begin{split}
&w_0...w_{n-1}+\frac{1}{n^2}w_n \cdot n(L^{\KP^{n-1}})^{n-1}\\
=&qw_n^n +\frac{1}{n^2}(-nL^{\KP^{n-1}})n(L^{\KP^{n-1}})^{n-1}\\
=&q(-n)^n(L^{\KP^{n-1}})^n+\frac{1}{n^2}(-nL^{\KP^{n-1}})n(L^{\KP^{n-1}})^{n-1}\\
=&-1.
\end{split}
\end{equation*}
In summary, 
\begin{equation}
\text{det}\left(\frac{\partial^2 F(\mathrm{w}_{\text{cr}})}{\partial \log w_j\partial \log w_i} \right)=-1.
\end{equation}

In the notation of \cite[Section 6.2]{ccit_toric3}, the formal asymptotic expansion of the integral $\int_\Gamma e^{F/z}\omega$ takes the form
\begin{equation}
e^{F(\mathrm{w}_{\text{cr}})/z}(-2\pi z)^{n/2}\text{Asym}_{\mathrm{w}_{\text{cr}}}(e^{F/z}\omega),    
\end{equation}
where $\text{Asym}_{\mathrm{w}_{\text{cr}}}(e^{F/z}\omega)$ is of the form
\begin{equation}
\frac{1}{\sqrt{\text{Hessian}(F)_{\mathrm{w}_{\text{cr}}}}}(1+a_1z+a_2z^2+...).    
\end{equation}

We calculate $\text{Asym}_{\mathrm{w}_{\text{cr}}}(e^{F/z}\omega)|_{q=\infty}$ in two ways.

In the limit $q=\infty$, we have
\begin{equation*}
\int_\Gamma e^{F/z}\omega|_{q=\infty}=\frac{1}{n}\prod_{i=0}^{n-1}\left(\Gamma\left(\frac{\chi_i}{z}\right)(-z)^{\chi_i/z}\right)\\
=\frac{1}{n}\left(\prod_{i=0}^{n-1}\Gamma\left(\frac{\chi_i}{z}\right)\right)\cdot (-z)^{\sum_{i=0}^{n-1}\chi_i/z}.
\end{equation*}
In the specialization (\ref{eqn:specialization_KP}), we have $\sum_{i=0}^{n-1}\chi_i=0$, thus the above is
\begin{equation*}
\int_\Gamma e^{F/z}\omega|_{q=\infty}=\frac{1}{n}\prod_{i=0}^{n-1}\Gamma\left(\frac{\chi_i}{z}\right).
\end{equation*}
By\footnote{Note that the odd Bernoulli numbers $B_{2k+1}=B_{2k+1}(0)=0$, $k\geq 1$.} \cite[5.11.1]{dlmf}, 
\begin{equation*}
\begin{split}
&\log\Gamma(x)\sim \left(x-\frac{1}{2}\right)\log x -x +\frac{1}{2}\log (2\pi)+ \sum_{k\geq 1}\frac{B_{k+1}(0)}{k(k+1)x^k}, \quad |\text{arg}(x)|<\pi-\delta, |x|>>1.
\end{split}
\end{equation*}
Here $\log$ is taken with principle values.

We pick $\delta>0$ sufficiently small and assume that $z$ satisfies
\begin{equation*}
0<\text{arg}(z)+\pi<<1, \quad |\text{arg}(\chi_i/z)|<\pi-\delta,\quad i=0,...,n-1.    
\end{equation*}

Therefore, we have  
\begin{equation*}
\begin{split}
\int_\Gamma e^{F/z}\omega|_{q=\infty}=&\frac{1}{n}\prod_{i=0}^{n-1}\Gamma\left(\frac{\chi_i}{z}\right)\\
\sim &\frac{1}{n}\prod_{i=0}^{n-1}\left(e^{\frac{\chi_i}{z}\log\left(\frac{\chi_i}{z}\right)-\frac{\chi_i}{z}} \left(\frac{\chi_i}{z} \right)^{-1/2}\sqrt{2\pi}\exp\left(\sum_{k\geq 1}\frac{B_{k+1}(0)z^k}{k(k+1)\chi_i^k}\right) \right).\\
\end{split}    
\end{equation*}
Next we process this asymptotic expansion. Consider the product
\begin{equation*}
\prod_{i=0}^{n-1}\left(e^{\frac{\chi_i}{z}\log\left(\frac{\chi_i}{z}\right)-\frac{\chi_i}{z}} \right)=\exp\left(\sum_{i=0}^{n-1}\left(\frac{\chi_i}{z}\log\left(\frac{\chi_i}{z} \right)-\frac{\chi_i}{z}\right). \right)    
\end{equation*}
We know that $\sum_{i=0}^{n-1}\chi_i/z=0$. Also, we can check that  
\begin{equation*}
\log(\chi_i/z)-(\log\chi_i-\log z)
=\begin{cases}
0 & i=0\\
-2\pi\sqrt{-1} & i=1,...,[n/2]\\
-\pi\sqrt{-1}  & i=[n/2]+1,...,n-1
\end{cases}   
\end{equation*}
 Since $\sum_{i=0}^{n-1}\frac{\chi_i}{z}\cdot 2\pi\sqrt{-1}=0$, we have $$\sum_{i=0}^{n-1}\frac{\chi_i}{z}\log\left(\frac{\chi_i}{z} \right)=A/z+\sum_{i=0}^{n-1}\frac{\chi_i}{z}(\log\chi_i-\log z)=A/z+\sum_{i=0}^{n-1}\frac{\chi_i}{z}\log\chi_i.$$
Here $A$ is a $z$-independent scalar. Therefore,
\begin{equation*}
\prod_{i=0}^{n-1}\left(e^{\frac{\chi_i}{z}\log\left(\frac{\chi_i}{z}\right)-\frac{\chi_i}{z}} \right)=e^{(A+\sum_{i=0}^{n-1}\chi_i\log\chi_i)/z}.    
\end{equation*}

We know that $\sum_{i=0}^{n-1}\frac{1}{\chi_i^k}=0$ unless $k$ is a multiple of $n$, in which case the sum is $n$. So
\begin{equation*}
\prod_{i=0}^{n-1}\exp\left(\sum_{k\geq 1}\frac{B_{k+1}(0)z^k}{k(k+1)\chi_i^k}\right)=\exp\left(\sum_{k\geq 1}\frac{B_{k+1}(0)z^k}{k(k+1)}\left(\sum_{i=0}^{n-1}\frac{1}{\chi_i^k} \right)\right)=\exp\left(n\sum_{l\geq 1}\frac{B_{nl+1}(0)z^{nl}}{nl(nl+1)}\right).    
\end{equation*}

We next consider the product
\begin{equation*}
\prod_{i=0}^{n-1}\left(\frac{\chi_i}{z} \right)^{-1/2}\sqrt{2\pi}=\exp\left(-\frac{1}{2}\sum_{i=0}^{n-1}\log\left(\frac{\chi_i}{z} \right)+\frac{n}{2}\log(2\pi) \right)    
\end{equation*}

We can check that for $0<\text{arg}(z)+\pi<<1$, i.e. $z=re^{\sqrt{-1}(\theta-\pi)}$ with $0<\theta<<1$, we have\footnote{Again, $\log$ is taken with principle values.}
\begin{equation*}
-\frac{1}{2}\sum_{i=0}^{n-1}\log\left(\frac{\chi_i}{z} \right)=-\frac{1}{2}\pi\sqrt{-1}+\frac{n}{2}\log(-z).    
\end{equation*}

Thus we have
\begin{equation*}
\prod_{i=0}^{n-1}\left(\frac{\chi_i}{z} \right)^{-1/2}\sqrt{2\pi}=\frac{(-2\pi z)^{n/2}}{\sqrt{-1}}.    
\end{equation*}

Putting these together, we find
\begin{equation*}
\int_\Gamma e^{F/z}\omega|_{q=\infty}\sim \frac{1}{n}\frac{(-2\pi z)^{n/2}}{\sqrt{-1}}e^{(A+\sum_{i=0}^{n-1}\chi_i\log\chi_i)/z}\exp\left(n\sum_{l\geq 1}\frac{B_{nl+1}(0)z^{nl}}{nl(nl+1)}\right). 
\end{equation*}
By the definition of $\text{Asym}$ and uniqueness of asymptotical expansion, we obtain
\begin{equation}\label{eqn:formal_asym_infty_1}
n\sqrt{-1}\text{Asym}_{\mathrm{w}_{\text{cr}}}(e^{F/z}\omega)|_{q=\infty}=\exp\left(n\sum_{l\geq 1}\frac{B_{nl+1}(0)z^{nl}}{nl(nl+1)}\right).    
\end{equation}

On the other hand, using\footnote{Besides the specialization (\ref{eqn:specialization_KP}), what is needed here is the observation that $u_j(\sigma)$ in \cite[Proposition 6.9]{ccit_toric3} are given by $\chi_0-\chi_1$, $\chi_0-\chi_2$, ..., $\chi_0-\chi_{n-1}$ and $-n\chi_0$.} \cite[Proposition 6.9]{ccit_toric3}, we have 
\begin{equation}\label{eqn:formal_asymp1}
\text{Asym}_{\mathrm{w}_{\text{cr}}}(e^{F/z}\omega)=e^{-\mu/z}I^{\KP^{n-1}}(q,z)\big|_{p_0}   \cdot \frac{1}{n\sqrt{-1}}\exp\left(-\sum_{k=1}^\infty \frac{B_{k+1}}{k(k+1)}N_{k,0}z^k \right).
\end{equation}
Restriction to $p_0$ is the same as setting $H=1$. Then, by Lemma \ref{lem:Asymptotics_of_I_function_of_KP} we know that
\begin{equation}
e^{-\mu/z}I^{\KP^{n-1}}(q,z)\big|_{p_0}\sim \underbrace{\sum_{k=0}^{\infty} \Phi_{k}(q) {z}^k}_{\Phi(z) \coloneqq } \quad \text{as} \quad {z} \rightarrow 0.
\end{equation}
By Corollary \ref{cor:polynomiality_of_P_0j_KP} and Corollary \ref{cor:polynomiality_of_Phi_k} we know that $P_{0,0}^{k,\KP^{n-1}}$ and $\Phi_k$ satisfy equation (\ref{eqn:Corollary_Eqn_of_General_Polynomiality}). Since $\mathds{L}_{j,1}=n\mathsf{D}_{\KP^{n-1}}$, equation (\ref{eqn:Corollary_Eqn_of_General_Polynomiality}) determines $\mathsf{P}_{0,0}^{\KP^{n-1}}(z)$ and $\Phi(z)$ once their constant terms with respect to $q$ are known. Since $I^{\KP^{n-1}}(q=0,z)=1=\Phi(z)\vert_{q=0}$ and $\mathsf{P}_{0,0}^{\KP^{n-1}}(z)\vert_{q=0}=-\sqrt{-1}$ then we have
\begin{equation*}
\Phi(z)=\frac{\mathsf{P}_{0,0}^{\KP^{n-1}}(z)}{-\sqrt{-1}}.
\end{equation*}
This implies that equation (\ref{eqn:formal_asymp1}) takes of the form
\begin{equation}\label{eqn:formal_asymp2}
n\sqrt{-1} \text{Asym}_{\mathrm{w}_{\text{cr}}}(e^{F/z}\omega)=\frac{1}{-\sqrt{-1}}{\mathsf{P}}_{0,0}^{\KP^{n-1}}(z)\exp\left(-\sum_{k=1}^\infty \frac{B_{k+1}}{k(k+1)}N_{k,0}z^k \right).  
\end{equation}
Passing to $q=\infty$, then by (\ref{eqn:const_at_infty}), (\ref{eqn:formal_asymp2}) becomes
\begin{equation}\label{eqn:formal_asymp_infty_2}
n\sqrt{-1} \text{Asym}_{\mathrm{w}_{\text{cr}}}(e^{F/z}\omega)|_{q=\infty}=\frac{1}{-\sqrt{-1}}\left(\sum_{k\geq{0}}a_{0,0}^kz^k \right)\exp\left(-\sum_{k=1}^\infty \frac{B_{k+1}}{k(k+1)}N_{k,0}z^k \right).      
\end{equation}
(\ref{eqn:R_matrix_identity}) now follows by combining (\ref{eqn:formal_asym_infty_1}) and (\ref{eqn:formal_asymp_infty_2}). So, we obtain the following result.
\begin{prop}[=Lemma \ref{lem:Final_of_proof_strategy}]
We have \begin{equation}
-\sqrt{-1}\exp\left(n\sum_{l>0}\frac{B_{nl+1}\left(0\right)}{nl+1}\frac{z^{nl}}{nl} \right)
=\left(\sum_{k\geq{0}}a_{0,0}^kz^k\right)\exp\left(\sum_{m>0}N_{2m-1, 0}\frac{(-1)^{2m-1}B_{2m}}{2m(2m-1)}z^{2m-1}\right).
\end{equation}
\end{prop}


\appendix
\section{An Analysis for \texorpdfstring{$I$}{I}-function of \texorpdfstring{$\KP^{n-1}$}{KP{n-1}}}\label{Appendix:Analysis_of_I_function}
By equation (\ref{eqn:I_KP_interms_of_F_minus_1}), we argued that the $I$-function of $I^{K\mathbb{P}^{n-1}}(q,z)$ is related to the $1$-shifted version $\mathcal{F}_{-1}(w,x)$ of the main hypergeometric series $\mathcal{F}(w,x)$ of \cite{zz} via
\begin{equation*}
I^{K\mathbb{P}^{n-1}}(q,z)=\mathcal{F}_{-1}\left(H/z,(-1)^nq\right)
\end{equation*}
and concluded that the proof of Lemma \ref{lem:properties_of_C_functions} follows from Theorem 1 and Theorem 2 of \cite{zz}. In a similar vein, we will also explain how other theorems in \cite{zz} apply or can be adapted to the $I$-function $I^{K\mathbb{P}^{n-1}}(q,z)$ of $\KP^{n-1}$.

We note that equation (24) on \cite[Page 6]{zz} implies\footnote{The series $I_{n-1}(x)$ is not a summand of $I$-functions we defined in our paper, it resembles some other series in \cite{zz}. We did not change it to be consistent with the notation of \cite{zz}.} $\mathcal{F}_{n-1}(w,x)/I_{n-1}(x)=\mathcal{F}_{-1}(w,x)$. On \cite[Page 8]{zz}, it is stated that $\mathcal{F}_p(w,x)$ has an asymptotic expansion of the form
\begin{equation*}
\mathcal{F}_p(w, x) \sim e^{\mu(x) w} \sum_{s=0}^{\infty} \Phi_{s,p}(x) w^{-s} \quad \text{as} \quad (w \rightarrow \infty),
\end{equation*}
see equation (28) of \cite{zz}. So $\mathcal{F}_{-1}(w,x)$ has one as well. This implies the following result.

\begin{lem}\label{lem:Asymptotics_of_I_function_of_KP}
The function $I^{K\mathbb{P}^{n-1}}(q,z)$ of $\KP^{n-1}$ has an asymptotic expansion of the form
\begin{equation*}
I^{K\mathbb{P}^{n-1}}(q,z) \sim e^{H\mu(q)/z}\sum_{s=0}^{\infty} \Phi_{s}(q) \left(\frac{z}{H}\right)^s \quad \text{as} \quad \left(\frac{z}{H} \rightarrow 0\right).
\end{equation*}
\end{lem}

In \cite{zz}, for $m\geq j \geq 0$, the series $\mathcal{H}_{m, j}\in\mathbb{Q}[X]$ is defined via the recurrence
\begin{equation}\label{eqn:recursion_for_Hmj}
\mathcal{H}_{0, j}=\delta_{0, j}, \quad \mathcal{H}_{m, j}=\mathcal{H}_{m-1, j}+(X-1)\left(X \frac{d}{d X}+\frac{m-j}{n}\right) \mathcal{H}_{m-1, j-1} \quad \text { for } m \geq 1
\end{equation}
with $\mathcal{H}_{m, -1}=0$, and first few $\mathcal{H}_{m, j}$'s are provided

\begin{equation*}
\begin{aligned}
& \mathcal{H}_{m, 0}(X)=1, \quad \mathcal{H}_{m, 1}(X)=\frac{1}{n}\left(\begin{array}{c}
m \\
2
\end{array}\right)(X-1), \\
& \mathcal{H}_{m, 2}(X)=\frac{1}{n^2}\left(\begin{array}{c}
m \\
3
\end{array}\right)((n+1) X-1)(X-1)+\frac{3}{n^2}\left(\begin{array}{c}
m \\
4
\end{array}\right)(X-1)^2 .
\end{aligned}
\end{equation*}
Set 
\begin{equation}\label{eqn:XandY}
\begin{split}
X^{\KP^{n-1}}=&(L^{\KP^{n-1}})^n,\\
Y^{\KP^{n-1}}=&\frac{\mathsf{D}_{\KP^{n-1}}L^{\KP^{n-1}}}{L^{\KP^{n-1}}}=\frac{1}{n}\left((L^{\KP^{n-1}})^n-1 \right).
\end{split}
\end{equation}

In equation (\ref{eqn:DLj_mutildej}), for $0\leq j \leq n-1$, we defined
\begin{equation*}
\mathsf{D}_{L_j}=\mathsf{D}_{\KP^{n-1}}+\frac{L^{\KP^{n-1}}_j}{z} \quad\text{and}\quad \widetilde{\mu}_j=\int_0^q\frac{L^{\KP^{n-1}}_j(u)}{u}du
\end{equation*}
where $L_j^{\KP^{n-1}}=L^{\KP^{n-1}}\zeta^j$.

For $1\leq k \leq n$, define\footnote{Note that the definition of $\mathds{L}_{j,k}$ does not depend on $j$ since $\frac{\mathsf{D}_{\KP^{n-1}}L_j^{\KP^{n-1}}}{L_j^{\KP^{n-1}}}=\frac{\mathsf{D}_{\KP^{n-1}}L^{\KP^{n-1}}}{L^{\KP^{n-1}}}.$}
\begin{equation}\label{eqn:mathdsLjk}
\mathds{L}_{j,k}=\sum_{i=0}^{k}\left(\binom{n}{i} \mathcal{H}_{n-i, k-i}-\frac{\mathsf{D}_{\KP^{n-1}}L_j^{\KP^{n-1}}}{n^{n-1}L_j^{\KP^{n-1}}} \sum_{r=1}^{k-i}\binom{n-r}{i} (-1)^r s_{n, n-r}n^{n-r} \mathcal{H}_{n-i-r, k-i-r}\right) \mathsf{D}_{\KP^{n-1}}^{i}.
\end{equation}

First two $\mathds{L}_{j,k}$ are given by
\begin{equation*}
\begin{split}
\mathds{L}_{j,1}=&n\mathsf{D}_{\KP^{n-1}},\\
\mathds{L}_{j,2}=&\binom{n}{2}\mathsf{D}^2_{\KP^{n-1}}-\frac{1}{n}\binom{n}{2}(X^{\KP^{n-1}}-1)\mathsf{D}_{\KP^{n-1}}+\frac{1}{n^2}\binom{n+1}{4}(X^{\KP^{n-1}}-1)X^{\KP^{n-1}}.
\end{split}
\end{equation*}
For $0\leq j \leq n-1$, define the following operator
\begin{equation*}
\mathds{L}_j={\mathsf{D}}_{L_j}^n-\frac{(L_j^{\KP^{n-1}})^{n}}{z^n}-\frac{{\mathsf{D}}_{\KP^{n-1}}L_j^{\KP^{n-1}}}{n^{n-1}L_j^{\KP^{n-1}}}\sum_{k=0}^{n-1}(-1)^{n-k}{s}_{n,k}n^k{\mathsf{D}}_{L_j}^k.
\end{equation*}
We can rewrite this as
\begin{equation*}
\mathds{L}_j={\mathsf{D}}_{L_j}^n-\frac{(L_j^{\KP^{n-1}})^{n}}{z^n}-\left((L^{\KP^{n-1}})^n-1 \right)\sum_{k=1}^{n}\frac{\widetilde{s}_{n,n-k}}{n^{k}}{\mathsf{D}}_{L_j}^{n-k}
\end{equation*}
where $\widetilde{s}_{n,n-k}$ are unsigned Stirling numbers of the first kind, and also equal to $k^{\text{th}}$ elementary symmetric polynomials evaluated at $0,1,\ldots, n-1$. So, one can see that $\mathds{L}_j$ is similar to the operator $\mathfrak{L}$ in \cite[page 9]{zz}. The difference is that they have elementary symmetric polynomials\footnote{The notation used in \cite{zz} for elementary symmetric polynomials evaluated at $1,2,\ldots, n$ is $S_k(n)$.} evaluated at $1,2,\ldots, n$. Yet, we can conclude the following result, which is similar to that of \cite{zz}. 
\begin{lem}\label{lem:decomposition_of_mathdsLj}
For all $0\leq{j}\leq{n-1}$, we have
\begin{equation*}
\mathds{L}_j=\sum_{k=1}^n\left(\frac{L_j^{\KP^{n-1}}}{z}\right)^{n-k}\mathds{L}_{j,k}.
\end{equation*}
\end{lem}
Lemma \ref{lem:decomposition_of_mathdsLj} gives a decomposition of $\mathds{L}_j$ in terms of $\mathds{L}_{j,k}$ and its proof is similar to \cite[Lemma B.7]{gt}. 
\begin{lem}\label{lem:AymptoticPFSolution}
Assume for $0\leq{j}\leq{n-1}$ a function of the form $e^{\frac{\widetilde{\mu}_{j}}{z}}\Psi_{j}(z)$ satisfies the following Picard--Fuchs equation:

\begin{equation*}
(L^{\KP^{n-1}})^{-n}\left(\mathsf{D}_{\KP^{n-1}}^n-\frac{{\mathsf{D}}_{\KP^{n-1}}L_j^{\KP^{n-1}}}{n^{n-1}L_j^{\KP^{n-1}}}\sum_{k=0}^{n-1}(-1)^{n-k}{s}_{n,k}n^k\mathsf{D}_{\KP^{n-1}}^k\right)\left(e^{\frac{\widetilde{\mu}_{j}}{z}}\Psi_{j}(z)\right)=z^{-n}e^{\frac{\widetilde{\mu}_{j}}{z}}\Psi_{j}(z).
\end{equation*}
where
\begin{equation*}
\Psi_{j}(z)=\sum_{k=0}^{\infty} \Psi_{j,k} z^{k}\quad\text{with } \Psi_{j,k}\in\mathbb{C}[\![q]\!]\text{ and }\Psi_{j,k}=0\text{ if } k<0.
\end{equation*}
Then, we have $\Psi_{j,k}\in\mathbb{C}[L_j^{\KP^{n-1}}]=\mathbb{C}[L^{\KP^{n-1}}]$.
\end{lem}
By the commutation rule (\ref{eqn:Commutation_of_DL}), and the definition of $\mathds{L}_j$, we see that
\begin{equation*}
\mathds{L}_j(\Psi_j(z))=0.
\end{equation*}
Then, an immediate corollary of Lemma \ref{lem:decomposition_of_mathdsLj} is the following result.
\begin{cor}\label{cor:LjkPsiEquation}
For $k\geq 0$, we have
\begin{equation}\label{eqn:Corollary_Eqn_of_General_Polynomiality}
\mathds{L}_{j,1}(\Psi_{j,k})
+\frac{1}{(L_j^{K\mathbb{P}^{n-1}})}\mathds{L}_{j,2}(\Psi_{j,k-1})
+\ldots
+\frac{1}{(L_j^{K\mathbb{P}^{n-1}})^{n-1}}\mathds{L}_{j,n}(\Psi_{j,k+1-n})=0.
\end{equation}
\end{cor}
This corollary is analogous to \cite[Theorem 4.i]{zz}.

Let $$\mathcal{I}\subset\mathbb{C}\left[L^{\KP^{n-1}}\right]$$ be the ideal generated by the product $X^{\KP^{n-1}}Y^{\KP^{n-1}}$ as in \cite{zz}. 
\begin{lem}\label{lem:mathdsLk_modI}
For any $k>1$, we have
\begin{equation*}
\mathds{L}_{k}\equiv\binom{n}{k}(\mathsf{D}_{\KP^{n-1}})(\mathsf{D}_{\KP^{n-1}}-Y^{\KP^{n-1}})\cdots(\mathsf{D}_{\KP^{n-1}}-(k-1)Y^{\KP^{n-1}})\mod\mathcal{I}.
\end{equation*}
\end{lem}
The proof of this lemma is the same as the proof of \cite[Lemma 4]{zz}. The only difference arises having elementary symmetric polynomials evaluated at $0,1,\ldots,n-1$ in the expressions rather than elementary symmetric polynomials evaluated at $1,2,\ldots,n$. 

Again by the techniques of \cite{zz}, Lemma \ref{lem:AymptoticPFSolution} follows from Lemma \ref{lem:mathdsLk_modI}. The details are similar to \cite[Appendix B]{gt}.

Now, note that we have
\begin{align*}
 \mathsf{I}^{\KP^{n-1}}(q,z)\big\vert_{H=1}
 &=e^{\log q/z}I^{\KP^{n-1}}(q,z)\big\vert_{H=1}\\
&\sim e^{(\mu(q)+\log q)/z}\sum_{k=0}^{\infty} \Phi_{k}(q) {z}^k \quad \text{as} \quad {z} \rightarrow 0,
\end{align*}
and hence
\begin{equation*}
 \mathsf{I}^{\KP^{n-1}}(q,z)\big\vert_{H=1} \sim e^{\widetilde{\mu}_0/z}\sum_{k=0}^{\infty} \Phi_{k}(q) {z}^k \quad \text{as} \quad {z} \rightarrow 0. 
\end{equation*}
As a result, we obtain the following statement as a corollary of Lemma \ref{lem:AymptoticPFSolution}, since $\mathsf{I}^{\KP^{n-1}}(q,z)$ satisfies the Picard--Fuchs equation (\ref{eqn:PF_with_stirling_for_vert_I}).

\begin{cor}\label{cor:polynomiality_of_Phi_k}
For all $k \geq 0$, we have $\Phi_{k}(q)\in\mathbb{C}\left[L^{\KP^{n-1}}\right]$, and $\Phi_{k}$ satisfy equation (\ref{eqn:Corollary_Eqn_of_General_Polynomiality}).
\end{cor}

Recall that the starting point of Section \ref{subsubsec:Change_of_variables} is the identification $q=x^{-n}$. We will provide a comparison of the operators $\mathds{L}_{j,k}$ given in (\ref{eqn:mathdsLjk}) and the analogous operators $\mathds{L}^{\CnZn}_{j,k}$ defined in \cite[Appendix B]{gt}. Firstly, we have
\begin{equation}\label{eqn:identify_X_s}
X^{\KP^{n-1}}=(L^{\KP^{n-1}})^n=\frac{(-1)^{n+1}}{n^n}(L^{\CnZn})^{n}= \frac{(-1)^{n+1}}{n^n}{X^{\CnZn}}.
\end{equation}
Next, we analyze the recursion (\ref{eqn:recursion_for_Hmj}) after (\ref{eqn:identify_X_s}):

\begin{equation*}
\begin{aligned}
\mathcal{H}_{m, j}
&=\mathcal{H}_{m-1, j}+(X-1)\left(X \frac{d}{d X}+\frac{m-j}{n}\right) \mathcal{H}_{m-1, j-1}\\
&=\mathcal{H}_{m-1, j}+(\frac{(-1)^{n+1}}{n^n}{X}^{\CnZn}-1)\left({X}^{\CnZn} \frac{d}{d {X}^{\CnZn}}+\frac{m-j}{n}\right) \mathcal{H}_{m-1, j-1}.\\
\end{aligned}
\end{equation*}

Let $\mathcal{H}_{m, j}=\frac{(-1)^j}{n^j}H_{m,j}$. Since $\mathcal{H}_{0, j}=\delta_{0,j}$, we obtain ${H}_{0, j}=\delta_{0,j}$. Also, the above recursion becomes

\begin{equation*}
\frac{(-1)^j}{n^j}{H}_{m, j}
=\frac{(-1)^j}{n^j}{H}_{m-1, j}+\left(\frac{(-1)^{n+1}}{n^n}{X}^{\CnZn}-1\right)\left({X}^{\CnZn} \frac{d}{d {X}^{\CnZn}}+\frac{m-j}{n}\right) \frac{(-1)^{j-1}}{n^{j-1}}{H}_{m-1, j-1}.
\end{equation*}
Then, after multiplying both sides with $(-1)^jn^j$, we obtain

\begin{equation*}
{H}_{0, j}=\delta_{0,j}, \quad \text{and} \quad{H}_{m, j}
={H}_{m-1, j}+n(1+\frac{(-1)^{n}}{n^n}{X}^{\CnZn})\left({X}^{\CnZn} \frac{d}{d {X}^{\CnZn}}+\frac{m-j}{n}\right) {H}_{m-1, j-1}.
\end{equation*}
This is nothing but the recursion given in \cite[Equation (B.7)]{gt}. So, we have

\begin{equation*}
\mathcal{H}_{m, j}=\frac{(-1)^j}{n^j}H_{m,j}^{\CnZn}
\end{equation*}
after $q=x^{-n}$.

If we anaylze $\mathds{L}_{j,k}$ defined by equation (\ref{eqn:mathdsLjk}) under the change of variables $q=x^{-n}$, we obtain the following

\begin{equation*}
\begin{aligned}
\mathds{L}_{j,k}
=&\sum_{i=0}^{k}\Bigg(\binom{n}{i} \frac{(-1)^{k-i}}{n^{k-i}}{H}^{\CnZn}_{n-i, k-i}\\ 
&+\frac{\mathsf{D}_{\CnZn}L_j^{\CnZn}}{n^{n}L_j^{\CnZn}} \sum_{r=1}^{k-i}\binom{n-r}{i} (-1)^r s_{n, n-r}n^{n-r} \frac{(-1)^{k-i-r}}{n^{k-i-r}}{H}^{\CnZn}_{n-i-r, k-i-r}\Bigg) \frac{(-1)^i}{n^i}\mathsf{D}_{\CnZn}^{i}\\
=&\frac{(-1)^k}{n^k}\sum_{i=0}^{k}\left(\binom{n}{i}{H}^{\CnZn}_{n-i, k-i}+\frac{\mathsf{D}_{\CnZn}L_j^{\CnZn}}{L_j^{\CnZn}} \sum_{r=1}^{k-i}\binom{n-r}{i} s_{n, n-r}{H}^{\CnZn}_{n-i-r, k-i-r}\right)\mathsf{D}_{\CnZn}^{i}.
\end{aligned}
\end{equation*}
Comparing this to $\mathds{L}_{j,k}^{\CnZn}$ defined in \cite[Appendix]{gt} we see that 
\begin{equation}
\mathds{L}_{j,k}=\frac{(-1)^k}{n^k}\mathds{L}_{j,k}^{\CnZn}  
\end{equation}
after the identification $q=x^{-n}$.

\end{document}